\documentclass[11pt]{article}

\pdfoutput=1

\usepackage{endnotes}
\let\footnote=\endnote

\usepackage{booktabs}
\usepackage{url,hyperref, subcaption}
\usepackage{multirow}
\usepackage{mathtools}

\newcommand{\lBrack}{[\![}
\newcommand{\rBrack}{]\!]}
\newcommand\norm[1]{\left\lVert#1\right\rVert}
\newcommand\tp{^\mathsf{T}}

\DeclareMathOperator*{\proj}{proj}

\DeclareMathOperator*{\argmax}{arg\,max}

\usepackage[ruled,linesnumbered]{algorithm2e}

\usepackage{ninecolors}
\usepackage{pgfplots}
\pgfplotsset{compat=1.18}
\usetikzlibrary {arrows.meta} 
\usetikzlibrary{patterns}

\usepackage{natbib}
 \bibpunct[, ]{(}{)}{,}{a}{}{,}

\usepackage{multibib}
\newcites{appendix}{References}

\usepackage{amsthm,amssymb}
\usepackage{setspace}
\usepackage[margin=0.8in]{geometry}
\usepackage{apptools}

\newtheorem{theorem}{Theorem}
\newtheorem{lemma}{Lemma}
\newtheorem{proposition}{Proposition}

\theoremstyle{definition}
\newtheorem{assumption}{Assumption}

\AtAppendix{\counterwithin{lemma}{section}}
\AtAppendix{\counterwithin{proposition}{section}}
\AtAppendix{\counterwithin{corollary}{section}}
\AtAppendix{\counterwithin{figure}{section}}
\AtAppendix{\counterwithin{table}{section}}
\AtAppendix{\counterwithin{equation}{section}}
\hypersetup{
           breaklinks=true,
           colorlinks=false, 
           pdfusetitle=true,
           hypertexnames=false,
        }
        
\title{Alternating Methods for Large-Scale AC Optimal Power Flow with Unit Commitment}
\author{Matthew Brun\thanks{Operations Research Center, Massachusetts Institute of Technology, Cambridge, MA (\href{mailto:brunm@mit.edu}{brunm@mit.edu}).} \and Thomas Lee\thanks{Institute for Data, Systems, \& Society, Massachusetts Institute of Technology, Cambridge, MA (\href{mailto:t\_lee@mit.edu}{t\_lee@mit.edu}).} \and Dirk Lauinger\thanks{Sloan School of Management, Massachusetts Institute of Technology, Cambridge, MA (\href{mailto:lauinger@mit.edu}{lauinger@mit.edu}).} \and Xin Chen\thanks{Department of Electrical \& Computer Engineering,
Texas A\&M University, College Station, TX (\href{mailto:xin\_chen@tamu.edu}{xin\_chen@tamu.edu}).} \and Xu Andy Sun\thanks{Sloan School of Management, Massachusetts Institute of Technology, Cambridge, MA (\href{mailto:sunx@mit.edu}{sunx@mit.edu}).}}
\date{\vspace{-3em}}

\begin{document}

\maketitle

\begin{abstract}
Security-constrained unit commitment with alternating current optimal power flow (SCUC-ACOPF) is a central problem in power grid operations that optimizes commitment and dispatch of generators under a physically accurate power transmission model while encouraging robustness against component failures.  SCUC-ACOPF requires solving large-scale problems that involve multiple time periods and networks with thousands of buses within strict time limits.  In this work, we study a detailed SCUC-ACOPF model with a rich set of features of modern power grids, including price-sensitive load, reserve products, transformer controls, and energy-limited devices. We propose a decomposition scheme and a penalty alternating direction method to find high-quality solutions to this model.  Our methodology leverages spatial and temporal decomposition, separating the problem into a set of mixed-integer linear programs for each bus and a set of continuous nonlinear programs for each time period.  To improve the performance of the algorithm, we introduce a variety of heuristics, including restrictions of temporal linking constraints, a second-order cone relaxation, and a contingency screening algorithm. We quantify the quality of feasible solutions through a dual bound from a convex second-order cone program.  To evaluate our algorithm, we use large-scale test cases from Challenge~3 of the U.S. Department of Energy's Grid Optimization Competition that resemble real power grid data under a variety of operating conditions and decision horizons.  The experiments yield feasible solutions with an average optimality gap of 1.33\%, demonstrating that this approach generates near-optimal solutions within stringent time limits.
\end{abstract}


\section{Introduction}

Unit commitment (UC) is a key problem in power systems used to schedule the operations of generators.  This problem seeks minimal cost dispatch solutions that satisfy operational constraints and meet estimated load.  Underlying the problem are a set of physical constraints that dictate how power can be transmitted over the electric grid.  These alternating current optimal power flow (ACOPF) constraints introduce complex nonconvexities and form a mixed-integer nonlinear program (MINLP; \citealt{castillo2016unit}).  Current practice necessitates that commitment solutions be updated as often as every few minutes, requiring that the models be solved within strict time limits  \citep{neiso2023overview}.

The class of ACOPF problems is strongly NP-hard \citep{bienstock2019strong} and is NP-hard even on tree networks \citep{lehmann2015ac}.  A standard approach for addressing this complexity is to use the direct current optimal power flow (DCOPF), which replaces the nonconvex AC power flow constraints with a linearized approximation \citep{stott2009dc}.  However, even under mild assumptions, the solutions for DCOPF are infeasible for the AC problem \citep{baker2021solutions}.  In practice, AC feasibility can be restored from the DC solution via iterative procedures (\citealt{molzahn2019survey} Ch.~6, \citealt{fang2021ac}).  Despite common use, studies have shown significant divergence in dispatch solutions between the true AC model and the approximated DC variant \citep{castillo2016unit}.  Additionally, DC approximations may significantly increase carbon emissions associated with the generation solution \citep{winner2023carbon} and can introduce artificial grid congestion not found in AC-feasible decisions \citep{bichler2023getting}.

Historically, the adoption of mixed-integer linear optimization for UC problems has led to savings for system operators on the order of billions of dollars per year \citep{o2011recent,carlson2012miso}.  Improvements to solution quality from the adoption of AC-based models are estimated to yield annual savings of \$6 -- \$19 billion in the United States \citep{cain2012history}.  However, a methodology for solving industry-scale UC problems with ACOPF constraints (UC-ACOPF) under realistic time limits has not yet been demonstrated.  The value of improved dispatch efficiency and stringent limits on solution time make algorithmic improvements a critical research topic.

To facilitate developments on large-scale UC-ACOPF, the U.S. Department of Energy’s Advanced Research Projects Agency-Energy (ARPA-E) conducted the Grid Optimization Competition.  The competition consisted of three optimization challenges spanning nine years.  Challenge~1 \citep{aravena2023recent} focuses on single-period ACOPF with security constraints, which encourage operational robustness against component failures.  Challenge 2 \citep{elbert2024gochallenge2} increases the realism of the power flow model and introduces single-period commitment decisions for fast-starting units. The competition culminated in Grid Optimization Competition Challenge 3 (GOC3; \citealt{holzer2023grid}), which adopts a multi-period model with full unit commitment decisions, reserve products, and security constraints, yielding a large-scale security-constrained UC-ACOPF model (SCUC-ACOPF). This model considers a rich set of realistic features of modern power grids, including price-sensitive load, various real and reactive reserve products, transformer and transmission line controls, and energy-limited devices (e.g., energy storage and demand response), presenting a level of detail beyond that previously considered in the literature. 

To evaluate competing algorithms while complying with critical infrastructure information regulations, synthetic networks and multi-period datasets were generated to resemble realistic grid data under a variety of operating conditions \citep{birchfield2016grid, li2018load}.  In GOC3, the networks range in size up to 8,000 nodes, and instances of the SCUC-ACOPF models contain millions of binary variables and nonconvex constraints.  Three operational time scales are considered: a day-ahead problem for wholesale markets spanning two days with a one-hour resolution, a real-time problem to account for revised forecasts spanning 8~hours with a 15-minute to one-hour resolution, and a planning process for severe weather events spanning 7~days with a 4-hour resolution \citep{holzer2025go}.  In this work, we detail an iterative decomposition-based algorithm for the GOC3 model and evaluate our algorithm on the numerous datasets from the competition.

\subsection{Literature Review}

An early Benders' decomposition approach for UC-ACOPF was proposed by \cite{fu2005security}, who enforce feasibility for the AC power flow constraints by adding cuts to a master UC problem.  Subsequently, alternating methods have been applied to facilitate decomposition.  \cite{gholami2023admm} apply a bilevel alternating direction method of multipliers (ADMM) to single-period SC-ACOPF, allowing the security constraints to be separated from the ACOPF problem. \cite{zhang2023solving} solve UC-ACOPF on small networks ($\leq\!$ 300 buses) using bilevel ADMM, which facilitates decomposition over the discrete UC and continuous ACOPF portions of the model.  

Other work has applied spatial decompositions to ACOPF.
\cite{tu2020two} separate subnetworks from a master network (e.g., distribution grids from the transmission grid) and solve a master problem via an interior point method which queries subproblems for gradient and Hessian information.  \cite{sun2021two} decompose a network into regions and use bilevel ADMM to identify a feasible solution.  Many other spatial decompositions appear in the literature on distributed computing; for a survey of these techniques, see \cite{molzahn2017survey}.

\cite{parker2024benchmark} propose a benchmark algorithm for the GOC3 model that separates the nonlinear AC power flow constraints from the discrete unit commitment decisions.  The benchmark first solves a copper-plate multi-period unit commitment problem, which does not contain any network power flows.  Then, the commitment decisions from this model are fixed in a set of temporally-decomposed AC power flow problems, which ignore reserve and ramping requirements.  Feasibility for ramping bounds is restored by projecting the power injection solution onto these constraints.
Finally, reserve products are re-dispatched by solving a set of linear programs, with the power injection solution fixed.
\cite{chevalier2024parallelized} introduces an alternate approach for the GOC3 problem that combines gradient-based solvers for unconstrained optimization with problem-specific heuristics, demonstrating competitive performance on some instances.

\subsection{Contributions}
This paper demonstrates that realistic SCUC-ACOPF can be solved to near-optimality at industry scale within stringent time limits.  We introduce an iterative algorithm to identify high-quality feasible solutions to the SCUC-ACOPF model proposed in GOC3 and evaluate its performance on a diverse set of test cases.  Specifically, our contributions are summarized as follows:
\begin{enumerate}
    \item We identify a decomposition scheme that allows for the separation of the problem into spatially-independent mixed-integer linear programs (MILPs) for unit commitment, and temporally-independent continuous nonlinear programs for AC power flow.  We then propose a penalty alternating direction method that utilizes this decomposition;
    \item We prove the convergence of this algorithm to a partial optimal solution, strengthening existing convergence results for penalty alternating direction methods;
    \item We introduce heuristics to tailor the algorithm to the SCUC-ACOPF problem, including:
    \begin{enumerate}
        \item a final solve of AC problems with fixed unit commitment decisions to ensure feasibility,
        \item restrictions on ramping and energy output that allow for temporal decomposition of the final AC solve while maintaining flexibility in power injection decisions,
        \item a second-order cone relaxation of the AC constraints to speed up iterations, and
        \item a screening algorithm to identify high-impact security constraints;
    \end{enumerate}
    \item We construct a convex relaxation that yields a dual bound on the optimal objective value, allowing for characterization of the quality of feasible solutions; and
    \item We evaluate our algorithm on many large-scale test cases with up to 8,000 buses and 48 time periods with time limits as low as 10~minutes. Our algorithm finds high-quality feasible solutions with an average optimality gap of 1.33\% relative to the dual bound.
\end{enumerate}

The remainder of this paper is organized as follows.  Section~\ref{sec:model} introduces a compact representation of the problem formulation.  Section~\ref{sec:reformulate} describes a penalized model that facilitates spatial and temporal decomposition.  Section~\ref{sec:basic} presents a basic penalty alternating direction method from this decomposition and provides convergence results.  Section~\ref{sec:improve} proposes a set of heuristics which tailor this algorithm to the SCUC-ACOPF setting.  Section~\ref{sec:upperbound} defines a dual bound on the optimal objective value.  Section~\ref{sec:results} provides computational results on the GOC3 test cases, demonstrating the performance of our algorithm and the benefits of the heuristics.  Section~\ref{sec:conclusion} concludes the paper.
All proofs are provided in the appendix.

\section{The SCUC-ACOPF Model}
\label{sec:model}
We introduce a compact representation of the SCUC-ACOPF model proposed in \cite{holzer2023grid}.
The model includes device-level unit commitment, net power injection, and AC branch control decisions, subject to bus power balance, zonal reserve constraints, device ramping and reserve limits, and security constraints.

We denote $[\cdot]^+ := \max\{\cdot,0\}$ and $\lBrack \cdot \rBrack := \{1,\dots,\cdot\}$.  The model spans $\mathcal{T} := \lBrack T \rBrack$ time periods, where $t = 0$ represents the initial state of the system.  The length of period $t$ is given by $d_t > 0$.

\subsection{Device Unit Commitment}

Denote by $\mathcal{J}^{\mathrm{pr}}$ (resp.\ $\mathcal{J}^{\mathrm{cs}}$) the set of producing (resp.\ consuming) devices, and dispatchable devices by the union $\mathcal{J}^{\mathrm{sd}} := \mathcal{J}^{\mathrm{pr}} \cup \mathcal{J}^{\mathrm{cs}}$. Producing devices generate power and consuming devices consume power.  Let $u^{\mathrm{on}}_{jt}$, $u^{\mathrm{su}}_{jt}$, and $u^{\mathrm{sd}}_{jt}$ indicate whether device $j$ is on, starting up, or shutting down in period $t$, respectively.  These variables are subject to the following constraints:
\begin{subequations}
    \label{constr:UC}
    \begin{align}
        & u^{\mathrm{on}}_{jt} - u^{\mathrm{on}}_{j,t-1} = u^{\mathrm{su}}_{jt} - u^{\mathrm{sd}}_{jt} \quad & \forall j \in \mathcal{J}^{\mathrm{sd}},\ t \in \mathcal{T}, \label{constr:UCSUSDDef1}\\ 
        & u^{\mathrm{su}}_{jt} + u^{\mathrm{sd}}_{jt} \leq 1 \quad & \forall j \in \mathcal{J}^{\mathrm{sd}},\ t \in \mathcal{T}, \label{constr:UCSUSDDef2}\\
        & \{u^{\mathrm{on}}_{jt},u^{\mathrm{su}}_{jt},u^{\mathrm{sd}}_{jt}\}_{t \in \mathcal{T}} \in \mathcal{X}^{\mathrm{u}}_j \quad & \forall j \in \mathcal{J}^{\mathrm{sd}}, \label{constr:UCDeviceSet}\\
        & (u^{\mathrm{on}}_{jt},\ u^{\mathrm{su}}_{jt},\ u^{\mathrm{sd}}_{jt}) \in \{0,1\}^3 \quad & \forall j \in \mathcal{J}^{\mathrm{sd}},\ t \in \mathcal{T}. \label{constr:UCBinary}
    \end{align}
\end{subequations}
Constraints \eqref{constr:UCSUSDDef1}--\eqref{constr:UCSUSDDef2} connect the device startups and shutdowns to their on status. The initial state $u^{\mathrm{on}}_{j0}$ is provided as data.  Constraint \eqref{constr:UCDeviceSet} enforces device-level commitment logic, including must-run and must-outage, minimum downtime and uptime, and maximum startup requirements.  These constraints are given by the set $\mathcal{X}^{\mathrm{u}}_j$, which is defined by linear constraints and described in detail in the appendix.  

The cost associated with a commitment decision is denoted by $C^{\mathrm{uc}}_j(u)$, which contains fixed shutdown and online costs, as well as downtime-dependent startup costs:
$$C^{\mathrm{uc}}_j(u) := \sum_{t \in \mathcal{T}} \left ( c^{\mathrm{on}}_{jt} u^{\mathrm{on}}_{jt} + c^{\mathrm{su}}_j u^{\mathrm{su}}_{jt} + c^{\mathrm{sd}}_{j} u^{\mathrm{sd}}_{jt} - \sum_{\substack{t' \in \mathcal{T}\\ t' < t}} c^{\mathrm{dd}}_{jtt'} [u^{\mathrm{su}}_{jt} + u^{\mathrm{on}}_{jt'} - 1]^+ \right ) \quad \forall j \in \mathcal{J}^{\mathrm{sd}}.$$
The parameter $c^{\mathrm{on}}_{jt}$ gives fixed costs associated with committed devices.  Parameters $c^{\mathrm{su}}_j \geq 0$ and $c^{\mathrm{sd}}_j \geq 0$ give baseline startup and shutdown costs.  Parameters $c^{\mathrm{dd}}_{jtt'} \geq 0$ give downtime-dependent startup cost adjustments that reduce the cost of startup for devices that were recently on.  Constraints \eqref{constr:UC} and the cost functions $C^{\mathrm{uc}}_j$ are both representable in a mixed-integer linear setting and contain interactions across time periods, but are separable across devices.

\subsection{Device Net Power Injection and Reserves}

Let $p^{\mathrm{tot}}_{jt}$ and $q^{\mathrm{tot}}_{jt}$ be the total real and reactive power injection of device $j$ at time $t$, and let $p^{\mathrm{on}}_{jt}$ give the dispatchable real power.  We use the convention that these variables represent injections for producing devices and withdrawals for consuming devices.   Total power injection and dispatchable power may differ due to injection during device startup or shutdown.
Variables $p^{\mathrm{su}}_{jt}$ (resp.\ $p^{\mathrm{sd}}_{jt}$) give the quantity of injection during startup (resp.\ shutdown). The parameter $p^{\mathrm{supc}}_{jtt'} \geq 0$ gives the quantity of power injection at time $t$ if a startup is scheduled at future time $t'$.  The set $\mathcal{T}^{\mathrm{supc}}_{jt} := \{t' > t\ :\ p^{\mathrm{supc}}_{jtt'} > 0\}$ contains the periods in which a startup will cause power to be produced at period $t$. Parameters $p^{\mathrm{sdpc}}_{jtt'}$ and sets $\mathcal{T}^{\mathrm{sdpc}}_{jt} := \{t' \leq t\ :\ p^{\mathrm{sdpc}}_{jtt'} > 0\}$ are defined similarly for shutdowns.  

Let $\mathcal{R}$ be the set of real power reserve products.  Variable $p^{\mathrm{res}}_{jtr}$ gives the quantity of reserve type $r \in \mathcal{R}$ provided by a device.  Net real power injection and reserves are constrained as follows:
\begin{subequations}
    \label{constr:PIReal}
    \begin{align}
        & p^{\mathrm{tot}}_{jt} = p^{\mathrm{on}}_{jt} + p^{\mathrm{su}}_{jt} + p^{\mathrm{sd}}_{jt} \quad & \forall j \in \mathcal{J}^{\mathrm{sd}},\ t \in \mathcal{T}, \label{constr:PIRealDef}\\
        & p^{\mathrm{su}}_{jt} = \sum_{t' \in \mathcal{T}^{\mathrm{supc}}_{jt}} p^{\mathrm{supc}}_{jtt'} u^{\mathrm{su}}_{jt'} \quad & \forall j \in \mathcal{J}^{\mathrm{sd}},\ t \in \mathcal{T}, \label{constr:PISUDef}\\
        & p^{\mathrm{sd}}_{jt} = \sum_{t' \in \mathcal{T}^{\mathrm{sdpc}}_{jt}} p^{\mathrm{sdpc}}_{jtt'} u^{\mathrm{sd}}_{jt'} \quad & \forall j \in \mathcal{J}^{\mathrm{sd}},\ t \in \mathcal{T}, \label{constr:PISDDef}\\
        & (u^{\mathrm{on}}_{jt},p^{\mathrm{tot}}_{jt},p^{\mathrm{on}}_{jt},\{p^{\mathrm{res}}_{jtr}\}_{r \in \mathcal{R}}) \in \mathcal{X}^{\mathrm{p}}_{jt} \quad & \forall j \in \mathcal{J}^{\mathrm{sd}},\ t \in \mathcal{T}, \label{constr:PIDeviceReserveSet}\\
        & (p^{\mathrm{tot}}_{jt},p^{\mathrm{on}}_{jt},\{p^{\mathrm{res}}_{jtr}\}_{r \in \mathcal{R}}) \geq 0 \quad & \forall j \in \mathcal{J}^{\mathrm{sd}},\ t \in \mathcal{T}. \label{constr:PINonneg}
    \end{align}
\end{subequations}
Constraints \eqref{constr:PIRealDef}--\eqref{constr:PISDDef} define total real power in terms of dispatchable and startup/shutdown power.  Constraint \eqref{constr:PIDeviceReserveSet} enforces device-level constraints that limit real power and reserve provision.  The set $\mathcal{X}^{\mathrm{p}}_{jt}$ bounds power injection and reserve quantities with linear constraints that link the commitment status of the device to its real power and reserve variables; this set is described in the appendix.

Reactive power up and down reserves are given by the variables $q^{\mathrm{res}}_{jt\uparrow}$ and $q^{\mathrm{res}}_{jt\downarrow}$, respectively.  Reactive power injection and reserve products are constrained by the following logic:
\begin{subequations}
    \label{constr:PIReactive}
    \begin{align}
        & q^{\mathrm{tot}}_{jt} \leq q^{\mathrm{max}}_{jt} \left ( u^{\mathrm{on}}_{jt} +\! \sum_{t' \in \mathcal{T}^{\mathrm{supc}}_{jt}} u^{\mathrm{su}}_{jt'} +\! \sum_{t' \in \mathcal{T}^{\mathrm{sdpc}}_{jt}} u^{\mathrm{sd}}_{jt'} \right ) + \beta^{\mathrm{max}}_j p^{\mathrm{tot}}_{jt} - \begin{cases}
            q^{\mathrm{res}}_{jt\uparrow} & \text{if } j \in \mathcal{J}^{\mathrm{pr}}\\
            q^{\mathrm{res}}_{jt\downarrow} & \text{if } j \in \mathcal{J}^{\mathrm{cs}}
        \end{cases}  & \forall j \in \mathcal{J}^{\mathrm{sd}},\, t \in \mathcal{T}, \label{constr:PIReactReserveMax}\\
        & q^{\mathrm{tot}}_{jt} \geq q^{\mathrm{min}}_{jt} \left ( u^{\mathrm{on}}_{jt} +\! \sum_{t' \in \mathcal{T}^{\mathrm{supc}}_{jt}} u^{\mathrm{su}}_{jt'} +\! \sum_{t' \in \mathcal{T}^{\mathrm{sdpc}}_{jt}} u^{\mathrm{sd}}_{jt'} \right ) + \beta^{\mathrm{min}}_j p^{\mathrm{tot}}_{jt} + \begin{cases}
            q^{\mathrm{res}}_{jt\downarrow} & \text{if } j \in \mathcal{J}^{\mathrm{pr}}\\
            q^{\mathrm{res}}_{jt\uparrow} & \text{if } j \in \mathcal{J}^{\mathrm{cs}}
        \end{cases}  & \forall j \in \mathcal{J}^{\mathrm{sd}},\, t \in \mathcal{T}, \label{constr:PIReactReserveMin}\\
        & (q^{\mathrm{res}}_{jt\uparrow},q^{\mathrm{res}}_{jt\downarrow}) \geq 0  & \forall j \in \mathcal{J}^{\mathrm{sd}},\, t \in \mathcal{T}. \label{constr:PIReactNonneg}
    \end{align}
\end{subequations}
Constraints \eqref{constr:PIReactReserveMax}--\eqref{constr:PIReactReserveMin} give bounds on reactive power injection and reserve products, where reserve products count against the appropriate bound.  Parameters $q^{\mathrm{max}}_{jt}$ and $q^{\mathrm{min}}_{jt}$ give the maximum and minimum reactive power bounds.  If a device cannot produce real power (i.e., it is not online or producing startup/shutdown power), it may not provide any reactive power or reserves.  The reactive power bounds may depend linearly on real power output with coefficients $\beta^{\mathrm{min}}_j$ and $\beta^{\mathrm{max}}_j$.

The cost associated with a power injection decision for a producing device is given by the convex piecewise linear function $C^{\mathrm{pow}}_{jt}(p^{\mathrm{tot}}_{jt})$. Similarly, the utility associated with energy consumption by a consuming device is given by the concave piecewise linear function $R^{\mathrm{pow}}_{jt}(p^{\mathrm{tot}}_{jt})$.  These functions only depend on total real power injection for a device and time period.  When maximizing utility $R^{\mathrm{pow}}_{jt}$ and minimizing cost $C^{\mathrm{pow}}_{jt}$, these functions and the constraints \eqref{constr:PIReal}--\eqref{constr:PIReactive} are representable in a linear setting and are separable by device.  However, the constraints \eqref{constr:PIReal}--\eqref{constr:PIReactive} are not separable across time due to interactions with startup and shutdown unit commitment variables across periods.

\subsection{Device Energy Injection Limits}
\label{sec:modelEnergyLim}
Devices may incur a penalty for exceeding maximum or minimum net energy injection limits over some time interval.  A set $W \in \mathcal{W}^{\mathrm{en,min}}_j$ is a set of time periods subject to a minimum energy limit ($W \subseteq \mathcal{T}$), and $e^{\mathrm{min}}_j(W)$ gives the value of the corresponding limit.  Similarly, $\mathcal{W}^{\mathrm{en,max}}_j$ and $e^{\mathrm{max}}_j(W)$ give data for maximum energy limits.  The energy violation cost function is 
$$C^{\mathrm{e}}_j(p) := c^{\mathrm{e}} \left (
\sum_{W \in \mathcal{W}^{\mathrm{en,min}}_j} \left [e^{\mathrm{min}}_j(W) - \sum_{t \in W} d_t p^{\mathrm{tot}}_{jt} \right ]^+ + \sum_{W \in \mathcal{W}^{\mathrm{en,max}}_j} \left [-e^{\mathrm{max}}_j(W) + \sum_{t \in W} d_t p^{\mathrm{tot}}_{jt} \right ]^+ \right ) \quad \forall j \in \mathcal{J}^{\mathrm{sd}},$$
where $c^{\mathrm{e}} \geq 0$ is the energy limit penalty coefficient.  When minimizing cost, $C^{\mathrm{e}}_j$ is linearly representable and separable across devices but couples time periods.

\subsection{Device Real Power Ramping}
Changes in device real power injection are subject to ramp rate bounds:
\begin{subequations}
    \label{constr:Ramp}
    \begin{align}
        & p^{\mathrm{tot}}_{jt} - p^{\mathrm{tot}}_{j,t-1} \geq -d_t \left ( p^{\mathrm{rd}}_{j} u^{\mathrm{on}}_{jt} + p^{\mathrm{rd,sd}}_j (1 - u^{\mathrm{on}}_{jt}) \right ) \quad & \forall j \in \mathcal{J}^{\mathrm{sd}},\ t \in \mathcal{T}, \label{constr:RampDown}\\
        & p^{\mathrm{tot}}_{jt} - p^{\mathrm{tot}}_{j,t-1} \leq d_t \left ( p^{\mathrm{ru}}_j (u^{\mathrm{on}}_{jt} - u^{\mathrm{su}}_{jt}) + p^{\mathrm{ru,su}}_j (u^{\mathrm{su}}_{jt} - u^{\mathrm{on}}_{jt} + 1) \right ) \quad & \forall j \in \mathcal{J}^{\mathrm{sd}},\ t \in \mathcal{T}. \label{constr:RampUp}
    \end{align}
\end{subequations}
The parameters $p^{\mathrm{ru}}_j$ (resp.\ $p^{\mathrm{rd}}_j$) give the maximum rates for increasing (resp.\ decreasing) power injection.  The parameters $p^{\mathrm{ru,su}}_j$ and $p^{\mathrm{rd,sd}}_j$ give the adjusted rates during startup and shutdown.  Initial power injection $p^{\mathrm{tot}}_{j0}$ is provided as data.  These constraints are separable across devices but couple consecutive time periods.

\subsection{Neighborhood Reserve Requirements}
Reserve products are subject to minimum requirements across neighborhoods of devices.  We denote the set of neighborhoods for real power reserves by $\mathcal{N}^{\mathrm{p}}$, where $N \in \mathcal{N}^{\mathrm{p}}$ is a set of devices $j$ in the same neighborhood $(N \subseteq \mathcal{J}^{\mathrm{sd}})$.  We similarly define $\mathcal{N}^{\mathrm{q}}$ as the set of reactive power neighborhoods.  A reserve requirement applies to some subset of reserve products $R \subset \mathcal{R}$.  Violation of the minimum reserve requirement within a neighborhood is subject to a penalty.  The sum of reserve requirement penalties and the cost of providing reserves is given by the functions $C^{\mathrm{res}}_t(p,q)$:
\begin{equation*}
    \begin{aligned}
        C^{\mathrm{res}}_t(p,q) := && d_t \left ( \sum_{\substack{j \in \mathcal{J}^{\mathrm{sd}}\\r \in \mathcal{R}}} {c^{\mathrm{res,p}}_{jtr}} p^{\mathrm{res}}_{jtr} + \sum_{R \subset \mathcal{R}} c^{\mathrm{res}}_R \sum_{N \in \mathcal{N}^{\mathrm{p}}} \left [p^{\mathrm{res,min}}_{t}(p;R,N) - \sum_{r \in R}\sum_{j \in N} p^{\mathrm{res}}_{jtr} \right ]^+ \vphantom{+ \sum_{R \in \{+,-\}} c^{\mathrm{res}}_R \sum_{N \in \mathcal{N}^{\mathrm{q}}} \left [q^{\mathrm{res,min}}(q;R,N,t) - \sum_{j \in N} q^{\mathrm{res}}_{jtr} \right ]^+} \right .\\
        && \left . + \sum_{r \in \{\uparrow,\downarrow\}} \left ( \sum_{j \in \mathcal{J}^{\mathrm{sd}}} c^{\mathrm{res,q}}_{jtr} q^{\mathrm{res}}_{jtr} + c^{\mathrm{res}}_r \sum_{N \in \mathcal{N}^{\mathrm{q}}} \left [q^{\mathrm{res,min}}_{t}(q;r,N) - \sum_{j \in N} q^{\mathrm{res}}_{jtr} \right ]^+ \vphantom{\sum_{R \in \mathcal{R}} c^{\mathrm{res}}_R \sum_{N \in \mathcal{N}^{\mathrm{p}}} \left [p^{\mathrm{res,min}}(p;R,N) - \sum_{r \in R}\sum_{j \in N} p^{\mathrm{res}}_{jtr} \right ]^+} \right ) \vphantom{\sum_{\substack{j \in \mathcal{J}^{\mathrm{sd}}\\r \in \mathcal{R}}}} \right ) & \quad \forall t \in \mathcal{T}.
    \end{aligned}
\end{equation*}
Variable costs for reserve products are given by parameters $c^{\mathrm{res,p}}_{jtr}$ and $c^{\mathrm{res,q}}_{jtr}$.  The functions $p^{\mathrm{res,min}}_t$ and $q^{\mathrm{res,min}}_t$ compute the reserve requirement for the set of products $R$ and neighborhood $N$.  These functions take a constant value, or a proportion of the total ($\sum_{j \in N} p^{\mathrm{tot}}_{jt}$) or maximum ($\max_{j \in N} p^{\mathrm{tot}}_{jt}$) power injection in the neighborhood. Reserve requirement violations for a product set $R$ are penalized at rate $c^{\mathrm{res}}_R$.  When minimizing cost, these functions are linearly representable for the given choices of functions $p^{\mathrm{res,min}}_t$ and $q^{\mathrm{res,min}}_t$, and are separable over time but couple devices.

\subsection{Alternating Current Optimal Power Flow}
We now model the optimal power flow constraints over the transmission network.  For a thorough introduction to ACOPF formulations, see \cite{bienstock2022mathematical}. Let $\mathcal{I}$ be the set of buses and $\mathcal{J}^{\mathrm{br}}$ the set of branches.  These branches are either DC lines ($\mathcal{J}^{\mathrm{dc}}$) or AC branches ($\mathcal{J}^{\mathrm{ac}}$), where transformers are considered AC branches.  Each branch $j$ connects an origin bus $i_j$ to a destination bus $i'_j$.  The network may also include shunt devices ($\mathcal{J}^{\mathrm{sh}}$), with shunt $j$ located at bus $i_j$.  Each dispatchable device $j$ is located at bus $i_j$.

Each bus $i$ has a voltage magnitude control $v_{it}$ and a phase angle control $\theta_{it}$.  Transformer $j$ has a tap ratio control $\tau_{jt}$ and a phase shift control $\phi_{jt}$. For AC lines, these controls are fixed to $1$ and $0$, respectively. Power flows on DC lines are lossless and bounded.  We denote real power flow at the origin and destination nodes of branch $j$ by $p^{\mathrm{fr}}_{jt}$ and $p^{\mathrm{to}}_{jt}$, oriented from the bus into the branch.  Similarly, reactive power flows are denoted by $q^{\mathrm{fr}}_{jt}$ and $q^{\mathrm{to}}_{jt}$.  Shunt device $j$ is controlled by activating some integral number of steps, selected by variable $u^{\mathrm{sh}}_{jt}$.  Real and reactive power over the shunt are given by $p^{\mathrm{sh}}_{jt}$ and $q^{\mathrm{sh}}_{jt}$.  Branch and shunt power flows satisfy the following relations:
\begin{subequations}
    \label{constr:PowerFlow}
    \begin{align}
        & \Delta_{jt} = \theta_{i_j t} - \theta_{i'_j t} - \phi_{jt} \quad & \forall j \in \mathcal{J}^{\mathrm{ac}},\ t \in \mathcal{T}, \label{constr:ACAngleDiff}\\
        & p^{\mathrm{fr}}_{jt} = v_{i_jt} \left ( \frac{ g_{i_jj} v_{i_jt}}{\tau_{jt}^2} + \frac{v_{i'_jt}}{\tau_{jt}} \left ( -g_j \cos (\Delta_{jt}) - b_j \sin(\Delta_{jt}) \right )  \right) \quad & \forall j \in \mathcal{J}^{\mathrm{ac}},\ t \in \mathcal{T}, \label{constr:ACpfr}\\
        & p^{\mathrm{to}}_{jt} = v_{i'_jt} \left ( g_{i'_jj} v_{i'_jt} + \frac{v_{i_jt}}{\tau_{jt}} \left ( -g_j \cos (\Delta_{jt}) + b_j \sin(\Delta_{jt}) \right )  \right) \quad & \forall j \in \mathcal{J}^{\mathrm{ac}},\ t \in \mathcal{T}, \label{constr:ACpto}\\
        & q^{\mathrm{fr}}_{jt} = v_{i_jt} \left ( \frac{ -b_{i_jj} v_{i_jt}}{\tau_{jt}^2} + \frac{v_{i'_jt}}{\tau_{jt}} \left ( b_j \cos (\Delta_{jt}) - g_j \sin(\Delta_{jt}) \right )  \right) \quad & \forall j \in \mathcal{J}^{\mathrm{ac}},\ t \in \mathcal{T},  \label{constr:ACqfr}\\
        & q^{\mathrm{to}}_{jt} = v_{i'_jt} \left ( -b_{i'_jj} v_{i'_jt} + \frac{v_{i_jt}}{\tau_{jt}} \left ( b_j \cos (\Delta_{jt}) + g_j \sin(\Delta_{jt}) \right )  \right) \quad & \forall j \in \mathcal{J}^{\mathrm{ac}},\ t \in \mathcal{T}, \label{constr:ACqto}\\
        & p^{\mathrm{fr}}_{jt} + p^{\mathrm{to}}_{jt} = 0 \quad & \forall j \in \mathcal{J}^{\mathrm{dc}},\ t \in \mathcal{T}, \label{constr:DCBalance}\\
        & p^{\mathrm{sh}}_{jt} = g^{\mathrm{sh}}_j u^{\mathrm{sh}}_{jt} v_{i_jt}^2 \quad & \forall j \in \mathcal{J}^{\mathrm{sh}},\ t \in \mathcal{T}, \label{constr:ShuntReal}\\
        & q^{\mathrm{sh}}_{jt} = -b^{\mathrm{sh}}_j u^{\mathrm{sh}}_{jt} v_{i_jt}^2 \quad & \forall j \in \mathcal{J}^{\mathrm{sh}},\ t \in \mathcal{T}, \label{constr:ShuntReactive}\\
        & v_{it} \in [v^{\mathrm{min}}_i, v^{\mathrm{max}}_i] \quad & \forall i \in \mathcal{I},\ t \in \mathcal{T}, \label{constr:BusSets}\\
        & \phi_{jt} \in [\phi_j^{\mathrm{min}},\phi_j^{\mathrm{max}}],\ \tau_{jt} \in [\tau_j^{\mathrm{min}},\tau_j^{\mathrm{max}}] \quad & \forall j \in \mathcal{J}^{\mathrm{ac}},\ t \in \mathcal{T}, \label{constr:ACSets}\\
        & (p^{\mathrm{fr}}_{jt},p^{\mathrm{to}}_{jt}) \in [-p_j^{\mathrm{max}},p_j^{\mathrm{max}}]^2,\  (q^{\mathrm{fr}}_j,q^{\mathrm{to}}_j) \in [q_j^{\mathrm{fr,min}},q_j^{\mathrm{fr,max}}] \times [q_j^{\mathrm{to,min}},q_j^{\mathrm{to,max}}] \quad & \forall j \in \mathcal{J}^{\mathrm{dc}},\ t \in \mathcal{T}, \label{constr:DCSets}\\
        & u^{\mathrm{sh}}_{jt} \in [u^{\mathrm{sh,min}}_j,u^{\mathrm{sh,max}}_j] \cap \mathbb{Z}_+ \quad & \forall j \in \mathcal{J}^{\mathrm{sh}},\ t \in \mathcal{T}. \label{constr:ShuntSets}
    \end{align}
\end{subequations}
Constraints \eqref{constr:ACAngleDiff}--\eqref{constr:ACqto} define AC branch real and reactive power flows via branch and bus controls.  Parameters $g_j$ and $b_j$ give the series conductance and susceptance of branch~$j$.  Parameter $g_{i j}$ is the sum of $g_j$ and the conductance of the branch shunt at bus $i$, and $b_{ij}$ is the sum of $b_j$, the charging susceptance of the line at bus $i$, and the susceptance of the branch shunt at bus~$i$.  Constraint \eqref{constr:DCBalance} enforces lossless real power balance over DC lines.  Constraints \eqref{constr:ShuntReal}--\eqref{constr:ShuntReactive} define power flows over shunts, where $g^{\mathrm{sh}}_j$ and $b^{\mathrm{sh}}_j$ give the per-step conductance and susceptance of the shunt.  Constraints \eqref{constr:BusSets}--\eqref{constr:ShuntSets} establish appropriate variable domains, where tap ratio bounds $(\tau^{\mathrm{min}}_j,\tau^{\mathrm{max}}_j)$ and voltage bounds $(v^{\mathrm{min}}_{i},v^{\mathrm{max}}_i)$ are positive and shunt step bounds $(u^{\mathrm{sh,min}}_j,u^{\mathrm{sh,max}}_j)$ are integral. 

AC branches are subject to a penalty for violating apparent power limits, given by $C^{\mathrm{ac}}_{jt}(p,q)$:
$$C^{\mathrm{ac}}_{jt}(p,q) := c^{\mathrm{s}} d_t \left [ \max \left \{\norm{(p^{\mathrm{fr}}_{jt},q^{\mathrm{fr}}_{jt})}_2,\ \norm{(p^{\mathrm{to}}_{jt},q^{\mathrm{to}}_{jt})}_2 \right \} - s^{\mathrm{max}}_j \right ]^+ \quad \forall j \in \mathcal{J}^{\mathrm{ac}},\ t \in \mathcal{T}.$$
The parameter $s^{\mathrm{max}}_j$ gives the apparent power limit for branch $j$ and $c^{\mathrm{s}} \geq 0$ is the penalty coefficient for branch limit violations.  The function $C^{\mathrm{ac}}_{jt}$ is convex and separable both by branch and time.  The AC branch power flow definitions \eqref{constr:ACpfr}--\eqref{constr:ACqto} are nonconvex and spatially coupled by bus control decisions.  The shunt power flow constraints \eqref{constr:ShuntReal}--\eqref{constr:ShuntReactive} and \eqref{constr:ShuntSets} are also nonconvex.  Overall, the constraints \eqref{constr:PowerFlow} are spatially coupled but separable over time.

\subsection{Nodal Power Balance}
\label{sec:modelBalance}
The transmission network is subject to real and reactive power balance requirements at each node, which are enforced by penalty functions $C^{\mathrm{bal}}_{it}(p,q)$:
\begin{equation*}
    \begin{aligned}
        && C^{\mathrm{bal}}_{it}(p,q) := d_t \left ( c^{\mathrm{p}}  \left | \sum_{\substack{j \in \mathcal{J}^{\mathrm{pr}}\\i_j = i}} p^{\mathrm{tot}}_{jt} - \sum_{\substack{j \in \mathcal{J}^{\mathrm{cs}}\\i_j = i}} p^{\mathrm{tot}}_{jt} - \sum_{\substack{j \in \mathcal{J}^{\mathrm{sh}}\\i_j = i}} p^{\mathrm{sh}}_{jt} - \sum_{\substack{j \in \mathcal{J}^{\mathrm{br}}\\i_j = i}} p^{\mathrm{fr}}_{jt} - \sum_{\substack{j \in \mathcal{J}^{\mathrm{br}}\\i'_j = i}} p^{\mathrm{to}}_{jt} \right | \right. &\\
        && \left. + c^{\mathrm{q}}  \left | \sum_{\substack{j \in \mathcal{J}^{\mathrm{pr}}\\i_j = i}} q^{\mathrm{tot}}_{jt} - \sum_{\substack{j \in \mathcal{J}^{\mathrm{cs}}\\i_j = i}} q^{\mathrm{tot}}_{jt} - \sum_{\substack{j \in \mathcal{J}^{\mathrm{sh}}\\i_j = i}} q^{\mathrm{sh}}_{jt} - \sum_{\substack{j \in \mathcal{J}^{\mathrm{br}}\\i_j = i}} q^{\mathrm{fr}}_{jt} - \sum_{\substack{j \in \mathcal{J}^{\mathrm{br}}\\i'_j = i}} q^{\mathrm{to}}_{jt} \right | \right ) & \quad \forall i \in \mathcal{I},\ t \in \mathcal{T}.
    \end{aligned}
\end{equation*}
The penalty coefficients $c^{\mathrm{p}}$ and $c^{\mathrm{q}}$ are nonnegative.  When minimizing cost, $C^{\mathrm{bal}}_{it}$ is linearly representable and separable over buses and time periods.

\subsection{Security Constraints}
\label{sec:modelContingency}
The active and reactive power set points are optimized in consideration of apparent power flow overloads that may result from branch outages, i.e., contingencies. Such overloads cost
$$C^{\mathrm{ctg}}_{kt}(p,q) := c^\mathrm{s} d_t \sum_{\substack{j \in \mathcal{J}^{\mathrm{ac}}\\j \neq j_k}} \left[ \norm{(f_{kjt}(p),\max \{ |q^\mathrm{fr}_{jt}|,|q^\mathrm{to}_{jt}| \}) }_2 - s^\mathrm{max,ctg}_j \right]^+ \quad \forall k \in \mathcal{K},\ t \in \mathcal{T},$$
where $\mathcal{K} = \lBrack K \rBrack$ is a set of contingencies, $j_k$ is the branch outaged by contingency $k$, and the post-contingency apparent power limit on branch $j$ is given by $s^{\mathrm{max,ctg}}_j$.  Post-contingency branch flows are computed under a DC power flow model, which reallocates real power, while reactive power flows remain fixed at pre-contingency levels.  
The affine function $f_{kjt}(p) := \sum_{i \in \mathcal{I}} F_{kji} \left ( p^{\mathrm{inj}}_{it}(p) - \frac{p^{\mathrm{sl}}_t(p)}{|\mathcal{I}|}\right )$ computes the new real power flow on branch $j$ after contingency $k$ under the DC model. The function $p^{\mathrm{inj}}_{it}$ computes the net real power injection at bus $i$ from all elements except AC branches,
$$p^{\mathrm{inj}}_{it}(p) := \sum_{\substack{j \in \mathcal{J}^{\mathrm{pr}}\\i_{j} = i}} p^{\mathrm{tot}}_{jt} - \sum_{\substack{j \in \mathcal{J}^{\mathrm{cs}}\\i_{j} = i}} p^{\mathrm{tot}}_{jt} - \sum_{\substack{j \in \mathcal{J}^{\mathrm{sh}}\\i_{j} = i}} p^{\mathrm{sh}}_{jt} - \sum_{\substack{j \in \mathcal{J}^{\mathrm{dc}}\\i_j = i}} p^{\mathrm{fr}}_{jt} - \sum_{\substack{j \in \mathcal{J}^{\mathrm{dc}}\\i'_j = i}} p^{\mathrm{to}}_{jt},$$
and the function $p^{\mathrm{sl}}_t(p) := \sum_{i \in \mathcal{I}} p^{\mathrm{inj}}_{it}(p)$ computes the systemwide imbalance in net injections, which is allocated evenly across buses when computing post-contingency flows. Parameters $F_{kij}$ are generation shift factors (GSFs) from bus $i$ to line $j$ under contingency $k$. These factors are defined explicitly in the appendix; for details, see \citet[][App.~8C]{wood2013power} and \cite{holzer2023fast}.
The total contingency cost is the sum of average and worst-case contingency costs:
$$C^{\mathrm{ctg}}_t(p,q) := \frac{1}{K} \sum_{k \in \mathcal{K}} C^{\mathrm{ctg}}_{kt}(p,q) + \max_{k \in \mathcal{K}}\ C^{\mathrm{ctg}}_{kt}(p,q) \quad \forall t \in \mathcal{T}.$$
The functions $C^{\mathrm{ctg}}_t$ are convex and separable over time periods.

\subsection{The Complete SCUC-ACOPF Model}

We combine the cost and utility terms into grouped objective functions:
\begin{equation*}
    \begin{aligned}
        R^{\mathrm{T}}_t(p,q) & := -C^{\mathrm{res}}_t(p,q) - C^{\mathrm{ctg}}_t(p,q) - \sum_{j \in \mathcal{J}^{\mathrm{ac}}} C^{\mathrm{ac}}_{jt}(p,q) - \sum_{i \in \mathcal{I}} C^{\mathrm{bal}}_{it}(p,q) \quad & \forall t \in \mathcal{T},\\
        R^{\mathrm{J}}_j(p,u) & := - C^{\mathrm{uc}}_j(u) - C^{\mathrm{e}}_j(p) + \begin{cases}
            \sum_{t \in \mathcal{T}} R^{\mathrm{pow}}_{jt}(p) & \text{if } j \in \mathcal{J}^{\mathrm{cs}}\\
            \sum_{t \in \mathcal{T}} -C^{\mathrm{pow}}_{jt}(p) & \text{if } j \in \mathcal{J}^{\mathrm{pr}}
        \end{cases}
        \quad & \forall j \in \mathcal{J}^{\mathrm{sd}}.
    \end{aligned}
\end{equation*}
This partition of the objective functions emphasizes how various terms permit temporal or spatial decomposition.  We now construct the full SCUC-ACOPF model: 
\begin{equation}
    \tag{SCUC-ACOPF}
    \label{SC-ACOPF}
    \begin{aligned}
        Z^*\ :=\ && \max_{\substack{p,q,u,v \\ \Delta,\theta,\tau,\phi}} \quad & \sum_{t \in \mathcal{T}} R^{\mathrm{T}}_t(p,q) + \sum_{j \in \mathcal{J}^{\mathrm{sd}}} R^{\mathrm{J}}_j(p,u)\\
        && \text{s.t.} \quad& \eqref{constr:UC}-\eqref{constr:PowerFlow}.
    \end{aligned}
\end{equation}
The model identifies feasible unit commitments, real and reactive power injections, and transmission network controls that maximize utility from delivered power minus operational costs and penalties.  The problem is a nonconvex MINLP due to unit commitment decisions and AC power flow constraints.  Table~\ref{table:modelBreakdown} categorizes the sources of model complexity, including temporal and spatial coupling, nonlinearity, nonconvexity, and integrality, by constraints and objective terms.

\begin{table}[tp]
    \centering
    \begin{tabular}{c|c|c|c|c|c|c|c|c|c|c|c|c}
        & \eqref{constr:UC} & \eqref{constr:PIReal} -- \eqref{constr:Ramp} & \eqref{constr:PowerFlow} & $C^{\mathrm{uc}}_j$ & $C^{\mathrm{pow}}_{jt}$, $R^{\mathrm{pow}}_{jt}$ & $C^{\mathrm{e}}_j$ & $C^{\mathrm{res}}_t$ & $C^{\mathrm{ac}}_{jt}$ & $C^{\mathrm{bal}}_{it}$ & $C^{\mathrm{ctg}}_{t}$ & $R^{\mathrm{J}}_j$ & $R^{\mathrm{T}}_t$\\
        \hline
         Temporal Coupling & \checkmark & \checkmark & & \checkmark & & \checkmark & & & & & \checkmark & \\
         \hline
         Spatial Coupling & & & \checkmark & & & & \checkmark & & $\thicksim$ & \checkmark & & \checkmark\\
         \hline
         Nonlinearity & & & \checkmark & & & &  & \checkmark & & \checkmark & & \checkmark\\
         \hline
          Nonconvexity & & & \checkmark & & & & & & & & & \\
          \hline
         Integrality & \checkmark & & \checkmark & \checkmark & & & & & & & \checkmark &
    \end{tabular}
    \caption{Sources of complexity by model component. The objective terms $C^{\mathrm{bal}}_{it}$ couple devices at the same bus but do not couple different buses; this limited spatial coupling is indicated by the symbol $\thicksim$.}
    \label{table:modelBreakdown}
\end{table}

\subsection{Assumptions}
From now on, we make the following assumption on the problem data.  The assumption accompanies the nonnegativity and integrality assumptions introduced previously. 
\begin{assumption}
    \label{assump:DataAssumptions}
    The problem data is such that:
    \begin{enumerate}
        \item For all $j \in \mathcal{J}^{\mathrm{sd}}$ and $t \in \mathcal{T}$, it holds that $c^{\mathrm{su}}_j \geq \sum_{\substack{t' \in \mathcal{T}\\t' < t}} c^{\mathrm{dd}}_{jtt'}$;
        \item For all $u$ satisfying \eqref{constr:UC},
        \begin{equation}
        \label{eq:SUSDRelaxation}
            u^{\mathrm{on}}_{jt} + \sum_{t' \in \mathcal{T}^{\mathrm{supc}}_{jt}} u^{\mathrm{su}}_{jt'} \leq 1, ~~
            u^{\mathrm{on}}_{jt} + \sum_{t' \in \mathcal{T}^{\mathrm{sdpc}}_{jt}} u^{\mathrm{sd}}_{jt'} \leq 1
            \quad
            \forall j \in \mathcal{J}^{\mathrm{sd}},\ t \in \mathcal{T};
        \end{equation}
        \item The sets $\mathcal{X}^{\mathrm{p}}_{jt}$ are closed, convex, bounded, and polyhedral for all $j \in \mathcal{J}^{\mathrm{sd}}$ and $t \in \mathcal{T}$; and
        \item \eqref{SC-ACOPF} is feasible with optimal objective value $Z^* \geq 0$.
    \end{enumerate}
\end{assumption}
First, the startup cost adjustments cannot exceed the total startup cost; second, devices cannot simultaneously be committed and produce power related to a future startup or past shutdown;
third, the feasible region for real power injection is bounded and defined by linear constraints; and fourth, the model is feasible with a nonnegative optimal objective value.

\section{Spatial-Temporal Reformulation}
\label{sec:reformulate}
Problem~\eqref{SC-ACOPF} scales with the number of planning periods and the size of the network. Time periods are linked by unit commitment, ramping, and energy limits, and buses are linked by power flows and neighborhood reserves. In this section, we reformulate the model to facilitate a decomposition into continuous nonconvex temporally separate subproblems and mixed-integer linear spatially separate subproblems.  We use the notation $u_j$ to denote the subset of variables $u$ that contain index $j$; for instance, $u_j = \{(u^{\mathrm{on}}_{jt},u^{\mathrm{su}}_{jt},u^{\mathrm{sd}}_{jt})\}_{t \in \mathcal{T}}$ if $j \in \mathcal{J}^{\mathrm{sd}}$.  Subscripts $t$ are used similarly.  We construct the following sets:
\begin{itemize}
    \item $\mathcal{X}_j^{\mathrm{uc}} := \left \{ (p_j,q_j,u_j) \,:\, \eqref{constr:UC} -\eqref{constr:Ramp} \right \} \quad \forall j \in \mathcal{J}^{\mathrm{sd}}$;
    \item $\mathcal{X}_t^{\mathrm{ac}} := \{(p_t,q_t,u^{\text{sh}}_t)\,:\,\exists (v_t,\Delta_t,\theta_t,\tau_t,\phi_t) \text{ s.t. } \eqref{constr:ACAngleDiff}-\eqref{constr:DCSets}\} \quad \forall t \in \mathcal{T}$;
    \item $\mathcal{X}_{jt}^{\mathrm{sh}} := \{u^{\mathrm{sh}}_{jt} \,:\, \eqref{constr:ShuntSets}\} \quad \forall j \in \mathcal{J}^{\mathrm{sh}},\ t \in \mathcal{T}$.
\end{itemize}
The sets $\mathcal{X}^{\mathrm{uc}}_j$ impose temporal coupling on the decision variables, the sets $\mathcal{X}^{\mathrm{ac}}_{t}$ impose spatial coupling, and the sets $\mathcal{X}^{\mathrm{sh}}_{jt}$ impose neither type of coupling.  

We also construct relaxed versions of the sets that are compatible with the desired decompositions.  Specifically, these relaxations of $\mathcal{X}^{\mathrm{uc}}_j$ and $\mathcal{X}^{\mathrm{sh}}_{jt}$ are continuous and separable over time.  To facilitate this construction, we introduce the following constraints which relax \eqref{constr:PISUDef}--\eqref{constr:PISDDef}:
\begin{subequations}
    \label{constr:RelaxPI}
    \begin{align}
        & 0 \leq p^{\mathrm{su}}_{jt} \leq \sum_{t' \in \mathcal{T}^{\mathrm{supc}}_{jt}} p^{\mathrm{supc}}_{jtt'} u^{\mathrm{su}}_{jt'} \quad & \forall j \in \mathcal{J}^{\mathrm{sd}},\ t \in \mathcal{T}, \label{constr:RelaxPISUDef}\\
        & 0 \leq p^{\mathrm{sd}}_{jt} \leq \sum_{t' \in \mathcal{T}^{\mathrm{sdpc}}_{jt}} p^{\mathrm{sdpc}}_{jtt'} u^{\mathrm{sd}}_{jt'} \quad & \forall j \in \mathcal{J}^{\mathrm{sd}},\ t \in \mathcal{T}. \label{constr:RelaxPISDDef}
    \end{align}
\end{subequations}
This relaxation will be used to aggregate the variables $u^{\mathrm{su}}_{jt}$ and $u^{\mathrm{sd}}_{jt}$ in Section~\ref{sec:improve_objective}. The relaxed sets are defined for all $t \in \mathcal{T}$ as follows:
\begin{itemize}
    \item $\mathcal{Y}^{\mathrm{uc}}_{jt} := \{(p_{jt},q_{jt},u^{\mathrm{on}}_{jt})\,:\, \exists (u^{\mathrm{su}}_j,u^{\mathrm{sd}}_j) \in [0,1]^{2T} \text{ s.t. } \eqref{constr:PIRealDef},\,\eqref{constr:PIDeviceReserveSet},\,\eqref{constr:PINonneg},\,\eqref{constr:PIReactive},\,\eqref{eq:SUSDRelaxation},\,\eqref{constr:RelaxPI},\,u^{\mathrm{on}}_{jt} \geq 0\}\  \forall j \in \mathcal{J}^{\mathrm{sd}}$;
    \item $\mathcal{Y}^{\mathrm{sh}}_{jt} := \{u^{\mathrm{sh}}_{jt} \in [u^{\mathrm{sh,min}}_j,u^{\mathrm{sh,max}}_j]\} \quad \forall j \in \mathcal{J}^{\mathrm{sh}}$.
\end{itemize}
Sets $\mathcal{Y}^{\mathrm{uc}}_{jt}$ relax sets $\mathcal{X}^{\mathrm{uc}}_j$ by replacing constraints~\eqref{constr:UC} with~\eqref{eq:SUSDRelaxation} and relaxing the binary constraints~\eqref{constr:UCBinary} to $(u^{\mathrm{on}}_{jt},u^{\mathrm{sd}}_{jt},u^{\mathrm{su}}_{jt}) \in [0,1]^3$.  Additionally, constraints \eqref{constr:PISUDef}--\eqref{constr:PISDDef} are relaxed to \eqref{constr:RelaxPI} and constraints \eqref{constr:Ramp} are removed.  Sets $\mathcal{Y}^{\mathrm{sh}}_{jt}$ relax the integrality restrictions of $\mathcal{X}^{\mathrm{sh}}_{jt}$.  Lemma~\ref{lemma:setRelaxations} shows that the sets $\mathcal{Y}$ are valid relaxations of their counterparts $\mathcal{X}$.

\begin{lemma}
\label{lemma:setRelaxations}
It holds that
\begin{enumerate}
    \item $\mathcal{X}^{\mathrm{uc}}_j \subseteq \{(p_j,q_j,u_j) \,:\, (p_{jt},q_{jt},u^{\mathrm{on}}_{jt}) \in \mathcal{Y}^{\mathrm{uc}}_{jt} \ \forall t \in \mathcal{T}\} \quad \forall j \in \mathcal{J}^{\mathrm{sd}}$; and \label{lemprop:UCRelaxInclusion}
    \item $\mathcal{X}^{\mathrm{sh}}_{jt} \subseteq \mathcal{Y}^{\mathrm{sh}}_{jt} \quad \forall j \in \mathcal{J}^{\mathrm{sh}}$,\ $t \in \mathcal{T}$. \label{lemprop:ShuntRelaxInclusion}
\end{enumerate}
\end{lemma}

From these set representations of the constraints, we construct a copy-constrained model:
\begin{subequations}
    \makeatletter
    \def\@currentlabel{COPY}
    \makeatother
    \label{EQ}
    \renewcommand{\theequation}{COPY.\alph{equation}}
    \begin{align}
    \max_{\substack{p,q,u\\\overline{p},\overline{q},\overline{u}}} \quad & \sum_{t \in \mathcal{T}} R^{\mathrm{T}}_t(p_t,q_t) + \sum_{j \in \mathcal{J}^{\mathrm{sd}}} R^{\mathrm{J}}_j(\overline{p}_j,\overline{u}_j)\\
    \text{s.t.} \quad & (p_t,q_t,u^{\mathrm{sh}}_t) \in \mathcal{X}^{\mathrm{ac}}_t \quad & \forall t \in \mathcal{T}, & \hspace{10em} \label{EQFirstSet}\\
    & (p_{jt},q_{jt},u^{\mathrm{on}}_{jt}) \in \mathcal{Y}^{\mathrm{uc}}_{jt} \quad & \forall j \in \mathcal{J}^{\mathrm{sd}},\ t \in \mathcal{T},\\
    & u^{\mathrm{sh}}_{jt} \in \mathcal{Y}^{\mathrm{sh}}_{jt} \quad & \forall j \in \mathcal{J}^{\mathrm{sh}},\ t \in \mathcal{T},\\
    & (\overline{p}_j, \overline{q}_j,\overline{u}_j) \in \mathcal{X}^{\mathrm{uc}}_j \quad & \forall j \in \mathcal{J}^{\mathrm{sd}},\\
    & \overline{u}^{\mathrm{sh}}_{jt} \in \mathcal{X}^{\mathrm{sh}}_{jt} \quad & \forall j \in \mathcal{J}^{\mathrm{sh}},\ t \in \mathcal{T}, \label{EQLastSet}\\
    & \rlap{$p^{\mathrm{tot}} = \overline{p}^{\mathrm{tot}},\ p^{\mathrm{su}} = \overline{p}^{\mathrm{su}},\ p^{\mathrm{sd}} = \overline{p}^{\mathrm{sd}},\ q^{\mathrm{tot}} = \overline{q}^{\mathrm{tot}},\ q^{\mathrm{res}} = \overline{q}^{\mathrm{res}},\ u^{\mathrm{on}} = \overline{u}^{\mathrm{on}},\ u^{\mathrm{sh}} = \overline{u}^{\mathrm{sh}}.$}
    \end{align}
\end{subequations}
This model duplicates the variables into blocks $(p,q,u)$ and $(\overline{p},\overline{q},\overline{u})$.  The variables $(p,q,u)$ are subject to the spatially linked power flow constraints $\mathcal{X}^{\mathrm{ac}}_t$, while the variables $(\overline{p},\overline{q},\overline{u})$ are subject to the temporally linked unit commitment, power injection, and ramping constraints $\mathcal{X}^{\mathrm{uc}}_j$ and shunt step constraints $\mathcal{X}^{\mathrm{sh}}_{jt}$.  To tighten the formulation in an eventual decomposition, variables $(p,q,u)$ are also constrained by sets $\mathcal{Y}$, which impose a relaxed form of the constraints assigned to $(\overline{p},\overline{q},\overline{u})$ such that the temporal separability remains valid.  Using the decomposable structure of the objective, the functions $R^{\mathrm{T}}_t$ and $R^{\mathrm{J}}_j$ take only the local variables $(p_t,q_t)$ and $(\overline{p}_j,\overline{u}_j)$ as input.  The final set of constraints are copy constraints, applied to a set of variables chosen so that the model \eqref{EQ} is equivalent to \eqref{SC-ACOPF}.  This equivalence is given in Proposition~\ref{prop:EQequivalence}.

\begin{proposition}
    \label{prop:EQequivalence}
    The models \eqref{SC-ACOPF} and \eqref{EQ} are equivalent; that is, given a feasible solution for one model, there exists a corresponding feasible solution for the other model that has the same objective value and the same values for the copied variables.
\end{proposition}

As a final step towards decomposition, we penalize the copy constraints in the objective with penalty parameter $\rho \geq 0$ and degree $\ell \in \{1,2\}$.  We define the penalty function
$$\Gamma_{\ell} (p,q,u,\overline{p},\overline{q},\overline{u}) := \norm{(p^{\mathrm{tot}},p^{\mathrm{su}},p^{\mathrm{sd}},q^{\mathrm{tot}},q^{\mathrm{res}},u^{\mathrm{on}},u^{\mathrm{sh}}) - (\overline{p}^{\mathrm{tot}},\overline{p}^{\mathrm{su}},\overline{p}^{\mathrm{sd}},\overline{q}^{\mathrm{tot}},\overline{q}^{\mathrm{res}},\overline{u}^{\mathrm{on}},\overline{u}^{\mathrm{sh}})}_\ell^\ell$$
and introduce the penalized model:
\begin{equation}
    \tag{PEN}
    \label{PEN}
    \begin{aligned}
        \max_{\substack{p,q,u\\\overline{p},\overline{q},\overline{u}}} \quad & \sum_{t \in \mathcal{T}} R^{\mathrm{T}}_t(p_t,q_t) + \sum_{j \in \mathcal{J}^{\mathrm{sd}}} R^{\mathrm{J}}_j(\overline{p}_j,\overline{u}_j) - \rho \Gamma_{\ell}(p,q,u,\overline{p},\overline{q},\overline{u})\\
        \text{s.t.} \quad & \eqref{EQFirstSet} - \eqref{EQLastSet}.
    \end{aligned}
\end{equation}

\section{Basic Penalty Alternating Direction Method and Decomposition}
\label{sec:basic}

Penalizing violations of relaxed constraints in the objective is leveraged in popular algorithms, including penalty methods and augmented Lagrangian methods \citep[][Ch.~17]{nocedal1999numerical}. To allow for decomposition, we apply an alternating direction method to the penalized formulation \eqref{PEN}.  This method alternates between fixing the the second block of variables, i.e., $(\overline{p},\overline{q},\overline{u})$, and optimizing over the first block, i.e., $(p,q,u)$, and fixing the first block of variables and optimizing over the second block. This style of algorithm, applied to a problem with constraints penalized in the objective, is referred to as a penalty alternating direction method \citep{geissler2017penalty} and has been used to solve  MINLPs for optimal gas transport \citep{geissler2015solving}. These methods have structural similarities to the feasibility pump, a heuristic for identifying feasible solutions to problems with integrality constraints \citep{fischetti2005feasibility, bonami2009feasibility}.  In this section, we propose a basic version of the algorithm and describe its convergence.

As the first block of variables are not coupled across time periods, we define a set of temporally-decomposed block subproblems.   For each $t \in \mathcal{T}$, the block subproblem is given by 
\begin{equation*}
    \label{B2}
	\begin{aligned}
    Z^{\mathrm{T}}_t(\overline{p}_t,\overline{q}_t,\overline{u}_t;\rho) \ :=\  \max_{p_t,q_t,u_t} \quad & R^{\mathrm{T}}_t(p_t,q_t) - \rho \Gamma_2(p_t,q_t,u_t,\overline{p}_t,\overline{q}_t,\overline{u}_t)\\
    \text{s.t.} \quad & \begin{aligned}[t]
    & (p_t,q_t,u^{\mathrm{sh}}_t) \in \mathcal{X}^{\mathrm{ac}}_t,\\
    & (p_{jt},q_{jt},u^{\mathrm{on}}_{jt}) \in \mathcal{Y}^{\mathrm{uc}}_{jt} \quad & \forall j \in \mathcal{J}^{\mathrm{sd}},\\
    & u^{\mathrm{sh}}_{jt} \in \mathcal{Y}^{\mathrm{sh}}_{jt} \quad & \forall j \in \mathcal{J}^{\mathrm{sh}}.
    \end{aligned}
    \end{aligned}
\end{equation*}
These models select degree $\ell = 2$ so that auxiliary variables are not needed to model the $L_1$ norm.  As the constraints and function $R^{\mathrm{T}}_t$ contain nonlinearities, this choice does not increase the problem complexity.  The temporally-decomposed block subproblems are continuous nonconvex programs, which can be solved to stationarity by interior point methods \citep{wachter2005line}.

Next, we define the block subproblems for the second block of variables, which are coupled temporally but not spatially.  This block also does not contain interactions between the shunt step variables $\overline{u}^{\mathrm{sh}}_{jt}$ and other decisions, so the variables $\overline{u}^{\mathrm{sh}}_{jt}$ can be separated into their own blocks.  For this reason, we revise the copy penalty function to exclude these variables:
$$\Gamma'_{\ell} (p,q,u,\overline{p},\overline{q},\overline{u}) := \norm{(p^{\mathrm{tot}},p^{\mathrm{su}},p^{\mathrm{sd}},q^{\mathrm{tot}},q^{\mathrm{res}},u^{\mathrm{on}}) - (\overline{p}^{\mathrm{tot}},\overline{p}^{\mathrm{su}},\overline{p}^{\mathrm{sd}},\overline{q}^{\mathrm{tot}},\overline{q}^{\mathrm{res}},\overline{u}^{\mathrm{on}})}_\ell^\ell.$$
Then, the spatially-decomposed device-level block subproblems are given, for each $j \in \mathcal{J}^{\mathrm{sd}}$, by
\begin{equation*}
    \label{B1}
	\begin{aligned}
    Z^{\mathrm{J}}_j(p_j,q_j,u_j;\rho)\ :=\  \max_{\overline{p}_j,\overline{q}_j,\overline{u}_j} \quad &  R^{\mathrm{J}}_j(\overline{p}_j,\overline{u}_j) - \rho \Gamma'_1(p_j,q_j,u_j,\overline{p}_j,\overline{q}_j,\overline{u}_j)\\
    \text{s.t.} \quad & (\overline{p}_j, \overline{q}_j,\overline{u}_j) \in \mathcal{X}^{\mathrm{uc}}_j.
    \end{aligned}
\end{equation*}
These problems select degree $\ell = 1$ so that each problem can be represented as a mixed-integer linear program.  The shunt variables $\overline{u}^{\mathrm{sh}}_{jt}$ are only subject to interval and integrality constraints and thus can be optimized in closed form via rounding:
\begin{equation*}
    Z^{\mathrm{SH}}_{jt}(u^{\mathrm{sh}}_{jt})\ :=\  \left \lfloor \frac{1}{2} + \proj_{\left [u^{\mathrm{sh,min}}_j,u^{\mathrm{sh,max}}_j \right ]} (u^{\mathrm{sh}}_{jt}) \right \rfloor \ \in \ \argmax_{\overline{u}^{\mathrm{sh}}_{jt} \in \mathcal{X}^{\mathrm{sh}}_{jt}} \quad -\rho \norm{u^{\mathrm{sh}}_{jt} - \overline{u}^{\mathrm{sh}}_{jt}}_1.
\end{equation*}

We now define the basic penalty alternating direction method (pADM), given as Algorithm~\ref{alg:basicAPM}.  The method takes as input a set of penalty coefficients $\{\rho_\tau\}_{\tau = 1}^{\overline{\tau}}$ and outputs a feasible solution to the device-level subproblems $Z^{\mathrm{J}}_j$ in $\overline{\tau}$ iterations.  A feasible solution to the device-level subproblems corresponds to a feasible solution to \eqref{EQ}; this property is formalized in Proposition~\ref{prop:DLFeasible}.  Note that, in the first iteration, the subproblems $Z^{\mathrm{T}}_t$ are solved with $\rho = 0$, eliminating the penalty term.

\begin{algorithm}[t]
    \caption{Basic Penalty Alternating Direction Method (pADM)}\label{alg:basicAPM}
    \SetKwInOut{Input}{Input}
    \Input{Penalty schedule $\{\rho_\tau\}_{\tau = 1}^{\overline{\tau}} \subset \mathbb{R}_{>}$} 
    \For{$\tau \in \{1,\ldots,\overline{\tau}\}$}{
        \uIf{$\tau = 1$}{
            For $t \in \mathcal{T}$, solve $Z^{\mathrm{T}}_t(0,0,0;0)$ to solution $(p^{(\tau)}_t,q^{(\tau)}_t,u^{(\tau)}_t)$\;
        }
        \Else{
            For $t \in \mathcal{T}$, solve $Z^{\mathrm{T}}_t(\overline{p}^{(\tau-1)}_t,\overline{q}^{(\tau-1)}_t,\overline{u}^{(\tau-1)}_t;\rho_{\tau})$ to solution $(p^{(\tau)}_t,q^{(\tau)}_t,u^{(\tau)}_t)$\;
        }
        For $j \in \mathcal{J}^{\mathrm{sd}}$, solve $Z^{\mathrm{J}}_j(p^{(\tau)}_j,q^{(\tau)}_j,u^{(\tau)}_j;\rho_{\tau})$ to solution $(\overline{p}^{(\tau)}_j,\overline{q}^{(\tau)}_j,\overline{u}^{(\tau)}_j)$\;
        For $j \in \mathcal{J}^{\mathrm{sh}}$ and $t \in \mathcal{T}$, compute ${\overline{u}^{\mathrm{sh} (\tau)}_{jt}} = Z^{\mathrm{SH}}_{jt}(u^{\mathrm{sh} (\tau)}_{jt})$\;
    }
    \Return{$(\overline{p}^{(\overline{\tau})},\overline{q}^{(\overline{\tau})},\overline{u}^{(\overline{\tau})})$}
\end{algorithm}

\begin{proposition}
\label{prop:DLFeasible}
    Let $(\overline{p},\overline{q},\overline{u})$ satisfy $( \overline{p}_j, \overline{q}_j,\overline{u}_j) \in \mathcal{X}^{\mathrm{uc}}_j$ for all $j \in \mathcal{J}^{\mathrm{sd}}$ and $\overline{u}^{\mathrm{sh}}_{jt} \in \mathcal{X}^{\mathrm{sh}}_{jt}$ for all $j \in \mathcal{J}^{\mathrm{sh}}$ and $t \in \mathcal{T}$.  Then, there is some $(p,q,u)$ such that $(p,q,u,\overline{p},\overline{q},\overline{u})$ is feasible for \eqref{EQ}.
\end{proposition}

Proposition~\ref{prop:DLFeasible}, combined with Proposition~\ref{prop:EQequivalence}, implies that a feasible solution to the device-level subproblems corresponds to a feasible solution to \eqref{SC-ACOPF} where the copied variables are fixed to their values in the device-level solution.  As the output of Algorithm~\ref{alg:basicAPM} is a feasible solution to the device-level problems, this output corresponds to a feasible solution of \eqref{SC-ACOPF}.  

Theorem~\ref{thm:basicAPMInfeasibility} characterizes the convergence of the algorithm, showing that the infeasibility of the iterates for the copy constraint is bounded by the inverse square root of the penalty parameter.  The result also shows that a similar bound holds on the distance between the copied variables in consecutive iterations.  To state the result, we define $D_t$ as the diameter of the (bounded) feasible region of $Z^{\mathrm{T}}_t$ and $L_t$ as the Lipschitz constant of the functions $R^{\mathrm{T}}_t$ over this region, so that every subgradient of $R^{\mathrm{T}}_t$ on the region is bounded in magnitude by $L_t$.  Similarly, $R_{j}$ is the maximum difference in the value of the function $R^{\mathrm{J}}_j$ over the (bounded) feasible region of $Z^{\mathrm{J}}_j$, and $R := \sum_{j \in \mathcal{J}^{\mathrm{sd}}} R_j$.  Finally, let $d$ be the number of copied variables, i.e., $(p^{\mathrm{tot}},p^{\mathrm{su}},p^{\mathrm{sd}},q^{\mathrm{tot}},q^{\mathrm{res}},u^{\mathrm{on}},u^{\mathrm{sh}}) \in \mathbb{R}^{d}$.  
We require that the nonlinear subproblems $Z^{\mathrm{T}}_t$ are solved to stationarity (i.e., the solution satisfies the Karush-Kuhn-Tucker conditions, see \citealt{nocedal1999numerical} Thm.~12.1) and the mixed-integer linear subproblems $Z^{\mathrm{J}}_j$ are solved to global optimality; these are natural assumptions for each problem class.

\begin{theorem}
    \label{thm:basicAPMInfeasibility}
    Let $(p^{(\tau)},q^{(\tau)},u^{(\tau)},\overline{p}^{(\tau)},\overline{q}^{(\tau)},\overline{u}^{(\tau)})$ be a sequence of iterates generated by Algorithm~\ref{alg:basicAPM}, where the subproblems $Z^{\mathrm{T}}_t$ are solved to stationarity.  Then,
    $$\Gamma_1(p^{(\tau)},q^{(\tau)},u^{(\tau)},\overline{p}^{(\tau-1)},\overline{q}^{(\tau-1)},\overline{u}^{(\tau-1)}) \leq \sqrt{\frac{C}{2\rho_{\tau}}} \quad \forall \tau \in \lBrack \tau \rBrack \setminus \{1\},$$
    where $C:= \sum_{t \in \mathcal{T}} d L_t D_t$.  If the subproblems $Z^{\mathrm{J}}_j$ are solved to global optimality, then 
    \begin{equation*}
    \renewcommand\arraystretch{1.5}
    \begin{array}{rl}
            \Gamma_1(\overline{p}^{(\tau)},\overline{q}^{(\tau)},\overline{u}^{(\tau)},\overline{p}^{(\tau-1)},\overline{q}^{(\tau-1)},\overline{u}^{(\tau-1)}) & \ \leq\  \dfrac{R}{\rho_{\tau}} + \sqrt{\dfrac{2C}{\rho_{\tau}}}\\
            \Gamma_1(p^{(\tau+1)},q^{(\tau+1)},u^{(\tau+1)},p^{(\tau)},q^{(\tau)},u^{(\tau)}) & \ \leq\  \dfrac{R}{\rho_{\tau}} + \sqrt{\dfrac{2C}{\min\{\rho_{\tau},\rho_{\tau+1}\}}}
    \end{array} \quad \forall \tau \in \lBrack \tau \rBrack \setminus \{1\}.
    \end{equation*}
\end{theorem}

Theorem~\ref{thm:bAPMConvergence} gives that, as the penalty parameter grows arbitrarily large, the iterates converge to a partial optimum for \eqref{EQ}.  A \textit{partial optimum} is a solution $(p^*,q^*,u^*,\overline{p}^*,\overline{q}^*,\overline{u}^*)$ that is (globally) optimal for \eqref{EQ} when each block of variables is fixed at the solution value, i.e., $(p^*,q^*,u^*,\overline{p}^*,\overline{q}^*,\overline{u}^*)$ is optimal for \eqref{EQ} with the variables $(p,q,u)$ fixed to $(p^*,q^*,u^*)$ and is optimal for \eqref{EQ} with the variables $(\overline{p},\overline{q},\overline{u})$ fixed to $(\overline{p}^*,\overline{q}^*,\overline{u}^*)$.  Convergence to partial optima is common in the analysis of alternating methods \citep{gorski2007biconvex}.  
Existing results for penalty alternating direction methods require fixing a penalty parameter and alternating between blocks until convergence is observed; after convergence, the penalty coefficient is increased \citep{geissler2017penalty}. Our result leverages the structure of \eqref{EQ} to allow for changes to the penalty parameter in every iteration, reducing the number of required subproblem evaluations. Theorem~\ref{thm:bAPMConvergence} requires that both sets of subproblems are solved to global optimality, which may be impractical for the nonconvex subproblems~$Z^{\mathrm{T}}_t$.  Nevertheless, the result gives an improved theoretical understanding of the algorithm in this setting.

\begin{theorem}
    \label{thm:bAPMConvergence}
    Suppose that $\overline{\tau} = \infty$ and $\lim_{\tau \rightarrow \infty} \rho_{\tau} = \infty$.  Let $(p^{(\tau)},q^{(\tau)},u^{(\tau)},\overline{p}^{(\tau)},\overline{q}^{(\tau)},\overline{u}^{(\tau)})$ be iterates generated by Algorithm~\ref{alg:basicAPM}, where the subproblems $Z^{\mathrm{T}}_t$ and $Z^{\mathrm{J}}_j$ are solved to global optimality.
    Then, any limit point $(p^*,q^*,u^*,\overline{p}^*,\overline{q}^*,\overline{u}^*)$ of the sequence of iterates is a partial optimum for \eqref{EQ}.
\end{theorem}

Although the basic pADM algorithm satisfies desirable convergence properties, it often fails to efficiently generate high-quality feasible solutions for \eqref{SC-ACOPF}.  In Section~\ref{sec:improve}, we highlight several insufficiencies of this algorithm and propose approaches to improve performance.

\section{Tailoring the pADM: Improvements and Heuristics}
\label{sec:improve}

Algorithm~\ref{alg:basicAPM} gives a basic penalty alternating direction framework that allows for spatial and temporal decomposition of \eqref{SC-ACOPF}.  In this section, we adapt the algorithm to improve the solution quality and reduce the computation time.  Section~\ref{sec:improve_finalsolve} adds a set of final subproblems that fix the discrete variables and reoptimize over the continuous variables.  Section~\ref{sec:improve_soc} accelerates the iterations by replacing the nonconvex temporally-decomposed block subproblems with their second-order cone (SOC) relaxations.  Section~\ref{sec:improve_contingency} reduces the number of contingencies included in the subproblems via a screening algorithm.  Section~\ref{sec:improve_objective} adds terms to the block subproblem objectives to more closely approximate the true objective function.

\subsection{Exact Feasibility with Final Subproblems}
\label{sec:improve_finalsolve}

The convergence property in  Theorem~\ref{thm:basicAPMInfeasibility} establishes that, as the penalty coefficient~$\rho$ grows, the iterates of Algorithm~\ref{alg:basicAPM} converge towards feasibility for the copy constraints in the formulation \eqref{EQ}.  However, this convergence may be slow in practice, requiring many iterations and excessively large values for~$\rho$, which may introduce computational difficulties and numerical instability.  We ensure exact feasibility of the final solution by  fixing the discrete unit commitment variables~$u$ to their value from the last-iteration device-level block subproblems $Z^{\mathrm{J}}_j$ and solving the original formulation \eqref{SC-ACOPF}.  The resulting restricted problem is temporally coupled only by the ramping constraints~\eqref{constr:Ramp} and energy violation penalty functions~$C^{\mathrm{e}}_j$.  Using the device-level power injection solution from the last iteration of~$Z^{\mathrm{J}}_j$, we construct restrictions of the linking elements that allow for temporal decomposition of the final problem.

\subsubsection{Ramping Constraint Restriction}
\label{sec:safeRamp}
When solving the final models with fixed unit commitment decisions $\overline{u}$, fixing the real power injection variables $p^{\mathrm{tot}}_{jt}$ to the solution $\overline{p}^{\mathrm{tot}}_{jt}$ from the device-level block subproblems $Z^{\mathrm{J}}_j$ ensures feasibility for the ramping constraints~\eqref{constr:Ramp} and allows for temporal separability but empirically leads to poor solution quality.  Allowing real power injection to deviate from the device-level subproblem solution introduces flexibility, which leads to better solutions and improves the computational performance of the interior point solver. We introduce such flexibility by restricting $p^\mathrm{tot}_{jt}$ to lie in an interval around the solution $\overline{p}^{\mathrm{tot}}_{jt}$:
\begin{equation}
    \label{eq:rampSafeConstraint}
    \begin{aligned}
        & \overline{p}^{\mathrm{tot}}_{jt} - \delta^-_{jt} \leq p^{\mathrm{tot}}_{jt} \leq \overline{p}^{\mathrm{tot}}_{jt} + \delta^+_{jt} & \quad \forall j \in \mathcal{J}^{\mathrm{sd}},\ t \in \mathcal{T}.
    \end{aligned}
\end{equation}
For $j \in \mathcal{J}^{\mathrm{sd}}$, the parameters $\delta_j$ are optimal solutions to a set of auxiliary linear programs:
\begin{subequations}
\label{eq:rampSafeLP}
    \begin{align}
        H_j(\overline{p}_j,\overline{u}_j) := \argmax_{\delta^-_j,\delta^+_j} \quad & \sum_{t\in \mathcal{T}} (\delta^-_{jt} + \delta^+_{jt} ) \label{eq:rampSafeObj}\\
        \text{s.t.} \quad & \delta^-_{jt} + \delta^+_{j,t-1} \leq r^{\mathrm{d}}_{jt} + (\overline{p}^{\mathrm{tot}}_{jt} - \overline{p}^{\mathrm{tot}}_{j,t-1}) & \ \forall t \in \mathcal{T} ,\label{constr:rampSafeDown}\\
        & \delta^-_{j,t-1} + \delta^+_{jt} \leq r^{\mathrm{u}}_{jt} - (\overline{p}^{\mathrm{tot}}_{jt} - \overline{p}^{\mathrm{tot}}_{j,t-1}) & \ \forall t \in \mathcal{T}, \label{constr:rampSafeUp}\\
        & \delta^-_{jt} \geq \frac{1}{2} \min \left\{r^{\mathrm{d}}_{jt} + (\overline{p}^{\mathrm{tot}}_{jt} - \overline{p}^{\mathrm{tot}}_{j,t-1}),\ r^{\mathrm{u}}_{j,t+1} - (\overline{p}^{\mathrm{tot}}_{j,t+1} - \overline{p}^{\mathrm{tot}}_{jt}) \right\} & \ \forall t \in \mathcal{T}, \label{constr:rampSafeDownHeur}\\
        & \delta^+_{jt} \geq \frac{1}{2} \min \left\{r^{\mathrm{u}}_{jt} - (\overline{p}^{\mathrm{tot}}_{jt} - \overline{p}^{\mathrm{tot}}_{j,t-1}),\ r^{\mathrm{d}}_{j,t+1} + (\overline{p}^{\mathrm{tot}}_{j,t+1} - \overline{p}^{\mathrm{tot}}_{jt}) \right\}& \ \forall t \in \mathcal{T}, \label{constr:rampSafeUpHeur}\\
        & r^{\mathrm{d}}_{jt} := d_t \left ( p^{\mathrm{rd}}_{j} \overline{u}^{\mathrm{on}}_{jt} + p^{\mathrm{rd,sd}}_j (1 - \overline{u}^{\mathrm{on}}_{jt}) \right ) & \ \forall t \in \mathcal{T}, \label{constr:rampSafeDownRangeDef}\\ 
        & r^{\mathrm{u}}_{jt} := d_t \left ( p^{\mathrm{ru}}_j (\overline{u}^{\mathrm{on}}_{jt} - \overline{u}^{\mathrm{su}}_{jt}) + p^{\mathrm{ru,su}}_j (\overline{u}^{\mathrm{su}}_{jt} - \overline{u}^{\mathrm{on}}_{jt} + 1) \right ) & \ \forall t \in \mathcal{T}. \label{constr:rampSafeUpRangeDef}
    \end{align}
\end{subequations}
These problems maximize the cumulative length of all intervals while enforcing that any solution within the intervals is feasible for the ramping constraints.  
Constraints \eqref{constr:rampSafeDown}--\eqref{constr:rampSafeUp} enforce feasibility of the interval for the original ramping constraints \eqref{constr:Ramp}, where parameters $r^{\mathrm{d}}_{jt}$ and $r^{\mathrm{u}}_{jt}$ give the maximum ramp down and up power under unit commitment solution $\overline{u}$. To ease notation, we fix $\delta^-_{j0} = \delta^+_{j0} = 0$.  Constraints \eqref{constr:rampSafeDownHeur}--\eqref{constr:rampSafeUpHeur} implement a heuristic that requires the interval lengths above and below the solution $\overline{p}^{\mathrm{tot}}_{jt}$ to be at least as large as in the solution that distributes the ramping slack evenly across time periods.

Proposition~\ref{prop:rampLPFeas} states that the problems \eqref{eq:rampSafeLP} are feasible and the optimal parameters $(\delta^-_{j},\delta^+_{j})$ are nonnegative.  Proposition~\ref{prop:rampLPValid} shows that, when the variables $u$ are fixed to $\overline{u}$, the constraint \eqref{eq:rampSafeConstraint} is a restriction of \eqref{constr:Ramp}, and thus any solution $(p,\overline{p},\overline{u},\delta)$ satisfying \eqref{eq:rampSafeConstraint}--\eqref{eq:rampSafeLP} is feasible for~\eqref{constr:Ramp}.

\begin{proposition}
    \label{prop:rampLPFeas}
    Let $j \in \mathcal{J}^{\mathrm{sd}}$ and $(\overline{p},\overline{q},\overline{u})$ be a solution such that $(\overline{p}_j,\overline{q}_j,\overline{u}_j) \in \mathcal{X}_j^{\mathrm{uc}}$.  Then, problem \eqref{eq:rampSafeLP} is feasible and bounded, and every $(\delta^-_j,\delta^+_j)$ satisfying \eqref{constr:rampSafeDown}--\eqref{constr:rampSafeUpHeur} is nonnegative.
\end{proposition}

\begin{proposition}
    \label{prop:rampLPValid}
    Let $(\overline{p},\overline{u},\delta)$ be a solution that satisfies \eqref{constr:rampSafeDown}--\eqref{constr:rampSafeUp} for all $j \in \mathcal{J}^{\mathrm{sd}}$.  Then, for any $p$ satisfying \eqref{eq:rampSafeConstraint}, the solution $(p,\overline{u})$ is feasible for \eqref{constr:Ramp}.
\end{proposition}

By Proposition~\ref{prop:rampLPValid}, the heuristic cuts \eqref{constr:rampSafeDownHeur}--\eqref{constr:rampSafeUpHeur} are not needed to enforce a restriction of \eqref{constr:Ramp}.  In practice, without these cuts, problem~\eqref{eq:rampSafeLP} yields many near-zero interval lengths, which does not improve device controllability.  One alternative is to maximize $\log ( \prod_{t} (\delta^-_{jt} + \delta^+_{jt} + 1) ) = \sum_{t} \log(\delta^-_{jt} + \delta^+_{jt} + 1)$, representing the \textit{volume} generated by the intervals.  This objective decreases the fraction of near-zero interval lengths but renders the problem nonlinear.  Adding the heuristic cuts yields solutions that achieve most of the volume and non-zero interval lengths obtained by the volume objective and preserves linearity; see Section~\ref{sec:resultsAblationSafeRamp} for quantitative results.

\subsubsection{Energy Injection Limit Restriction}
\label{sec:energyMinMax}
We similarly design a restriction of the energy violation penalty functions $C^{\mathrm{e}}_j$ defined in Section~\ref{sec:modelEnergyLim}.  As \eqref{SC-ACOPF} minimizes the penalty, we define a restriction as an upper bound on the function. The bound will be the sum of time-separable functions $\overline{C}^{\mathrm{e}}_{jt}$ centered around some candidate injection $\overline{p}$:
\begin{equation*}
    \begin{aligned}
        && \overline{C}^{\mathrm{e}}_{jt}(p;\overline{p}) := c^{\mathrm{e}} d_t \left ( \sum_{W \in \mathcal{W}^{\mathrm{en,min}}_{jt}} \left [\frac{e^{\mathrm{min}}_j(W) - \sum_{t' \in W} d_{t'} \overline{p}^{\mathrm{tot}}_{jt'}}{\sum_{t' \in W} d_{t'}} + \overline{p}^{\mathrm{tot}}_{jt} - p^{\mathrm{tot}}_{jt} \right ]^+ \right .\\
        && \left . + \sum_{W \in \mathcal{W}^{\mathrm{en,max}}_{jt}} \left [\frac{- e^{\mathrm{max}}_j(W) + \sum_{t' \in W} d_{t'} \overline{p}^{\mathrm{tot}}_{jt'}}{\sum_{t' \in W} d_{t'}} + p^{\mathrm{tot}}_{jt} - \overline{p}^{\mathrm{tot}}_{jt} \right ]^+ \right) \quad & \forall j \in \mathcal{J}^{\mathrm{sd}},\ t \in \mathcal{T},
    \end{aligned}
\end{equation*}
where $\mathcal{W}^{\mathrm{en,min}}_{jt} := \{W \in \mathcal{W}^{\mathrm{en,min}}_{j}\,:\, t \in W\}$ is the set of minimum energy periods that intersect with time $t$; $\mathcal{W}^{\mathrm{en,max}}_{jt}$ is defined analogously. Proposition~\ref{prop:energyViolationRestriction} states that the functions $\overline{C}^{\mathrm{e}}_{jt}$ provide the desired bound. In practice, the  injection $\overline{p}$ will be a feasible solution to the device-level subproblems.

\begin{proposition}
    \label{prop:energyViolationRestriction}
    For any $\overline{p}$, $p$, and $j \in \mathcal{J}^{\mathrm{sd}}$, it holds that $C^{\mathrm{e}}_j(p) \leq \sum_{t \in \mathcal{T}} \overline{C}^{\mathrm{e}}_{jt}(p;\overline{p}).$
\end{proposition}

\subsubsection{Final Subproblem Formulation}
Let $(\overline{p},\overline{q},\overline{u})$ be a feasible solution to the device-level block subproblems, i.e., $(\overline{p}_j,\overline{q}_j,\overline{u}_j) \in \mathcal{X}_j^{\mathrm{uc}}$ for all $j \in \mathcal{J}^{\mathrm{sd}}$ and $\overline{u}^{\mathrm{sh}}_{jt} \in \mathcal{X}^{\mathrm{sh}}_{jt}$ for all $j \in \mathcal{J}^{\mathrm{sh}}$ and $t \in \mathcal{T}$.  Let $\delta$ satisfy \eqref{eq:rampSafeLP}. We define a final objective function $R^{\mathrm{F}}_t$ that excludes unit commitment and contingency costs and applies the energy restriction:
\begin{equation*}
    \begin{aligned}
        R^{\mathrm{F}}_t(p,q;\overline{p}) & := \sum_{j \in \mathcal{J}^{\mathrm{cs}}} R^{\mathrm{pow}}_{jt}(p) - \sum_{j \in \mathcal{J}^{\mathrm{pr}}} C^{\mathrm{pow}}_{jt}(p) - C^{\mathrm{res}}_t(p,q) - \sum_{j \in \mathcal{J}^{\mathrm{ac}}} C^{\mathrm{ac}}_{jt}(p,q) - \sum_{i \in \mathcal{I}} C^{\mathrm{bal}}_{it}(p,q) - \sum_{j \in \mathcal{J}^{\mathrm{sd}}} \overline{C}^{\mathrm{e}}_{jt}(p;\overline{p}).
    \end{aligned}
\end{equation*}
For $t \in \mathcal{T}$, the temporally-separate final subproblems that incorporate the ramping and energy restrictions are defined by
\begin{equation*}
    \label{FSP}
	\begin{aligned}
	   Z^{\mathrm{F}}_t(\overline{p},\overline{q},\overline{u},\emptyset;\delta) \ :=\  \max_{p_t,q_t,u_t} \quad & R^{\mathrm{F}}_t(p_t,q_t;\overline{p})\\
        \text{s.t.} \quad & \eqref{constr:PIReal} - \eqref{constr:PIReactive},\, \eqref{eq:rampSafeConstraint},\ (p_t,q_t,u^{\mathrm{sh}}_t) \in \mathcal{X}^{\mathrm{ac}}_t,\ u = \overline{u}.
	\end{aligned}
\end{equation*}
The final constraint fixes the unit commitment variables to the device-level solution $\overline{u}$.

Propositions~\ref{prop:EQequivalence}~and~\ref{prop:DLFeasible} imply that there is some feasible solution for \eqref{SC-ACOPF} with the copied variables fixed to the device-level solution; such a solution satisfies the ramping restriction~\eqref{eq:rampSafeConstraint} and thus the final subproblems $Z^{\mathrm{F}}_t$ are feasible.  The final argument $\delta_t$ of $Z^{\mathrm{F}}_t$ denotes the dependence of the ramping restriction on the interval parameters from the auxiliary problems~\eqref{eq:rampSafeLP}.  The argument $\emptyset$ indicates that the final subproblems exclude contingency costs. Section~\ref{sec:improve_contingency} expands this notation to incorporate contingencies.

\subsection{Second-Order Cone Relaxation}
\label{sec:improve_soc}

As the final subproblems $Z^{\mathrm{F}}_t$ only require a feasible solution to the device-level subproblems $Z^{\mathrm{J}}_j$, we can use a proxy in place of the temporally-decomposed subproblems~$Z^{\mathrm{T}}_t$ in the iterative algorithm. We choose an SOC relaxation similar to those by \cite{jabr2006radial} and \cite{kocuk2016strong}.  The relaxation accelerates each iteration and yields subproblems that can be solved to global optimality.  Feasibility for the original ACOPF constraints is restored by solving the final subproblems. The relaxation is applied to the constraints \eqref{constr:ACAngleDiff}--\eqref{constr:ShuntSets}.  We define the change of variables
\begin{equation*}
    \begin{gathered}
        c_{jt} = \frac{1}{\tau_{jt}} v_{i_j t} v_{i'_j t} \cos (\Delta_{jt}), \quad s_{jt} = \frac{1}{\tau_{jt}} v_{i_j t} v_{i'_j t} \sin (\Delta_{jt}), \quad
        \omega_{it} = v_{it}^2, \quad \mu_{jt} = \frac{v_{i_jt}^2}{\tau_{jt}^2}, \quad \text{and} \quad  \mu^{\mathrm{sh}}_{jt} = u^{\mathrm{sh}}_{jt} v^2_{i_j t}.\\
    \end{gathered}
\end{equation*}
This transformation is invertible if and only if $c_{jt}^2 + s_{jt}^2 = \mu_{j t} \omega_{i'_j t}$ and $\Delta_{jt} = \arctan(s_{jt} / c_{jt})$.  To convexify these conditions, we relax the first constraint to $c_{jt}^2 + s_{jt}^2 \leq \mu_{j t} \omega_{i'_j t}$ and remove the second constraint, eliminating the angle differences $\Delta_{jt}$ from the model. 
Then, the following defines the SOC relaxation:
\begin{subequations}
    \label{constr:SOC}
    \begin{align}
        & p^{\mathrm{fr}}_{jt} = g_{i_j j} \mu_{jt} - g_j c_{jt} - b_j s_{jt} \quad & \forall j \in \mathcal{J}^{\mathrm{ac}},\ t \in \mathcal{T}, \label{constr:SOCpfr}\\
        & p^{\mathrm{to}}_{jt} = g_{i'_j j} \omega_{i'_jt} - g_j c_{jt} + b_j s_{jt} \quad & \forall j \in \mathcal{J}^{\mathrm{ac}},\ t \in \mathcal{T}, \label{constr:SOCpto}\\
        & q^{\mathrm{fr}}_{jt} = -b_{i_j j} \mu_{jt} + b_j c_{jt} - g_j s_{jt} \quad & \forall j \in \mathcal{J}^{\mathrm{ac}},\ t \in \mathcal{T}, \label{constr:SOCqfr}\\
        & q^{\mathrm{to}}_{jt} = -b_{i'_j j} \omega_{i'_j t} + b_j c_{jt} + g_j s_{jt} \quad & \forall j \in \mathcal{J}^{\mathrm{ac}},\ t \in \mathcal{T}, \label{constr:SOCqto}\\
        & p^{\mathrm{sh}}_{jt} = g^{\mathrm{sh}}_j \mu^{\mathrm{sh}}_{jt} \quad & \forall j \in \mathcal{J}^{\mathrm{sh}},\ t \in \mathcal{T}, \label{constr:SOCShuntReal}\\
        & q^{\mathrm{sh}}_{jt} = -b^{\mathrm{sh}}_j \mu^{\mathrm{sh}}_{jt} \quad & \forall j \in \mathcal{J}^{\mathrm{sh}},\ t \in \mathcal{T}, \label{constr:SOCShuntReactive}\\
        & c_{jt}^2 + s_{jt}^2 \leq \mu_{jt} \omega_{i'_j t} \quad & \forall j \in \mathcal{J}^{\mathrm{ac}},\ t \in \mathcal{T}, \label{constr:SOCCosSin}\\
        & \frac{\omega_{i_j t}}{(\tau^{\mathrm{max}}_{j})^2} \leq \mu_{jt} \leq \frac{\omega_{i_j t}}{(\tau^{\mathrm{min}}_{j})^2} \quad & \forall j \in \mathcal{J}^{\mathrm{ac}},\ t \in \mathcal{T}, \label{constr:SOCTapRatio}\\
        & u^{\mathrm{sh,min}}_j \omega_{i_j t} \leq \mu^{\mathrm{sh}}_{jt} \leq u^{\mathrm{sh,max}}_j \omega_{i_j t} \quad & \forall j \in \mathcal{J}^{\mathrm{sh}},\ t \in \mathcal{T}, \label{constr:SOCShuntDomain}\\
        & \omega_{it} \in [(v^{\mathrm{min}}_i)^2, (v^{\mathrm{max}}_i)^2] \quad & \forall i \in \mathcal{I},\ t \in \mathcal{T}, \label{constr:SOCBusSets}\\
        & \eqref{constr:DCBalance}, \eqref{constr:DCSets}. \label{constr:SOCDC}
    \end{align}
\end{subequations}
We denote the region defined by these constraints as $\mathcal{X}^{\mathrm{soc}}_{t} := \{(p_t,q_t,\mu_t,\omega_t)\,:\,\exists(c_t,s_t) \text{ s.t. } \eqref{constr:SOC}\}$.  
As the constraints \eqref{constr:SOCCosSin} are SOC-representable and the other constraints are linear, the sets $\mathcal{X}^{\mathrm{soc}}_{t}$ are convex and SOC-representable. Proposition~\ref{prop:SOCRelaxation} states that these sets relax constraints \eqref{constr:PowerFlow}.
\begin{proposition}
    \label{prop:SOCRelaxation}
    For all $t \in \mathcal{T}$, the constraints \eqref{constr:SOC} are a convex relaxation of the power flow constraints \eqref{constr:PowerFlow}.  In particular, for all $t \in \mathcal{T}$, the set $\mathcal{X}^{\mathrm{soc}}_t$ satisfies 
    \begin{equation*}
        \begin{aligned}
            &\ \left \{(p_t,q_t,u^{\mathrm{sh}}_t)\,:\, \exists (\mu_t,\omega_t) \ \mathrm{s.t.}\ (p_t,q_t,\mu_t,\omega_t) \in \mathcal{X}^{\mathrm{soc}}_t,\ u^{\mathrm{sh}}_{jt} = \mu^{\mathrm{sh}}_{jt}/{\omega_{i_jt}}\ \forall j \in \mathcal{J}^{\mathrm{sh}} \right \} \\
            \supseteq &\ \{(p_t,q_t,u^{\mathrm{sh}}_t) \in \mathcal{X}^{\mathrm{ac}}_t\,:\, u^{\mathrm{sh}}_{jt} \in \mathcal{Y}^{\mathrm{sh}}_{jt}\ \forall j \in \mathcal{J}^{\mathrm{sh}}\}.
        \end{aligned}
    \end{equation*}
\end{proposition}

When constructing $\mathcal{X}^{\mathrm{soc}}_t$, the variables $u^{\mathrm{sh}}_{jt}$ are projected out.  However, the corresponding value for these variables can be recovered by $u^{\mathrm{sh}}_{jt} = \frac{\mu^{\mathrm{sh}}_{jt}}{\omega_{i_j t}}$.  The quadratic penalty terms for these variables, which appear in the penalty functions $\Gamma_2$, can instead be written as
$$(u^{\mathrm{sh}}_{jt} - \overline{u}^{\mathrm{sh}}_{jt})^2 = (\mu^{\mathrm{sh}}_{jt}/\omega_{i_j t} - \overline{u}^{\mathrm{sh}}_{jt} )^2 = (\mu^{\mathrm{sh}}_{jt} - \overline{u}^{\mathrm{sh}}_{jt} \omega_{i_j t})^2/\omega_{i_j t}^2 \approx (\mu^{\mathrm{sh}}_{jt} - \overline{u}^{\mathrm{sh}}_{jt} \omega_{i_j t})^2,$$
where the approximate equality is exact when voltages are $1$.  The approximating term is convex quadratic in $(\mu,\omega)$.  Any error from $\omega_{i_j t} \neq 1$ simply rescales the penalty coefficients of this term.

\subsection{Contingency Screening}
\label{sec:improve_contingency}
Large numbers of contingencies may significantly increase subproblem solution times, as each contingency $k$ requires the incorporation of its nonlinear (albeit convex) cost function $C^{\mathrm{ctg}}_{kt}$ into the objective (see Section~\ref{sec:modelContingency}).  Even computing the GSFs $F_{kji}$ requires a matrix inversion, which may be costly. Empirically, many contingencies incur minimal cost and thus have little impact on the optimal solution, which motivates contingency screening. We identify a set of contingencies $\overline{\mathcal{K}} \subset \mathcal{K}$ and branches $\overline{\mathcal{J}}_{kt} \subset \mathcal{J}^{\mathrm{ac}}$  for $k \in \overline{\mathcal{K}}$ that incur a large contingency cost at a candidate solution $(\overline{p}, \overline{q})$.  

To avoid evaluating every contingency, we first identify a set of contingencies that are likely to incur penalties.  Specifically, this set of prescreened contingencies is denoted by $\overline{\mathcal{K}}$ and contains the $N^{\mathrm{ctg}}$ contingencies with  largest real power flows on the outaged line $j_k$ in the candidate solution.  We then evaluate the post-contingency line flows for prescreened contingencies and select $\overline{\mathcal{J}}_{kt}$ as the set of the $N^{\mathrm{br}}$ branches with the largest post-contingency apparent power limit violations exceeding some threshold $s^{\mathrm{thr}} \geq 0$.  Post-contingency line flows are evaluated by computing the GSFs. To compute these parameters efficiently, we update the required matrix inverse by a Sherman-Morrison rank-1 update on a base matrix \citep{hager1989updating}, avoiding the inversion of a new matrix for each contingency \citep{holzer2023fast}. See the appendix for additional details.

We approximate the post-contingency real power flow functions by $${\overline{f}_{kjt}(p;\overline{\mathcal{I}}_{kjt})} := \sum_{i \in \overline{\mathcal{I}}_{kjt}} F_{kji} \left ( p^{\mathrm{inj}}_{it}(p) - \frac{p^{\mathrm{sl}}_t(p)}{|\mathcal{I}|}\right ).$$
The subset $\overline{\mathcal{I}}_{kjt} \subseteq \mathcal{I}$ consists of the $N^{\mathrm{bus}}$ buses with the largest absolute contributions to the post-contingency line flow at the incumbent solution, $\left | F_{kji} \left ( p^{\mathrm{inj}}_{it}(\overline{p})  - \frac{p^{\mathrm{sl}}_t(\overline{p})}{|\mathcal{I}|}  \right ) \right |$.  This approximation increases the sparsity in the power injection variables, improving computational performance. Algorithm~\ref{alg:contingency} details the contingency screening approach. To ease notation, we denote by $\argmax^k$ the operator that selects the $k$ largest elements, with ties broken arbitrarily. The algorithm returns the subsets $\overline{\mathcal{N}} := (\overline{\mathcal{I}}, \overline{\mathcal{J}}, \overline{\mathcal{K}})$. Based on these subsets, we construct new cost functions:
\begin{equation*}
    \begin{aligned}
    \overline{C}^{\mathrm{ctg}}_{kt}(p,q;\overline{\mathcal{N}}) & := c^\mathrm{s} d_t \sum_{j \in \overline{\mathcal{J}}_{kt}} \left[ \norm{({\overline{f}_{kjt}(p; \overline{\mathcal{I}}_{kjt})},\max \{|q^\mathrm{fr}_{jt}|,|q^\mathrm{to}_{jt}|\})}_2 - s^\mathrm{max,ctg}_j \right]^+ \quad \forall k \in \overline{\mathcal{K}},\ t \in \mathcal{T},\\
    \overline{C}^{\mathrm{ctg}}_t(p,q;\overline{\mathcal{N}}) & := \frac{\sum_{k \in \overline{\mathcal{K}}_t} C^{\mathrm{ctg}}_{kt}(\overline{p},\overline{q})}{N^{\mathrm{ctg}} \sum_{k \in \overline{\mathcal{K}}_t} \overline{C}^{\mathrm{ctg}}_{kt}(\overline{p},\overline{q};\overline{\mathcal{N}})} \left ( \sum_{k \in \overline{\mathcal{K}}_t} \overline{C}^{\mathrm{ctg}}_{kt}(p,q;\overline{\mathcal{N}}) \right ) + \max_{k \in \overline{\mathcal{K}}}\  \overline{C}^{\mathrm{ctg}}_{kt}(p,q;\overline{\mathcal{N}})  \quad  \forall t \in \mathcal{T}.
    \end{aligned}
\end{equation*}
These functions are based on fewer contingencies and branches and thus are easier to construct and optimize. The coefficient in the second function averages over the $N^{\mathrm{ctg}}$ selected contingencies and scales the cost by the proportion of the apparent power limit violations identified by the modified functions $\overline{f}_{kjt}$ to reduce bias in the average cost after screening.

\begin{algorithm}[t]
    \caption{Contingency Screening}\label{alg:contingency}
    \SetKwInOut{Input}{Input}
    \Input{Candidate power injection solution $(\overline{p},\overline{q})$, thresholds $N^{\mathrm{ctg}}$, $N^{\mathrm{br}}$, $N^{\mathrm{bus}}$, $s^{\mathrm{thr}}$}
    Compute the base-case GSF matrix\;
    Define $\overline{\mathcal{K}} := \underset{k \in \mathcal{K}}{\argmax}^{N^{\mathrm{ctg}}}\  \underset{t \in \mathcal{T}}{\max}\ \max \{|\overline{p}^{\mathrm{fr}}_{j_k t}|,|\overline{p}^{\mathrm{to}}_{j_k t}|\}$\;
    \For{$k \in \overline{\mathcal{K}},\ t \in \mathcal{T}$}{
        Compute post-contingency GSF matrix $F_k$ via a rank-1 update and construct functions $f_{kjt}$ for all $j \in \mathcal{J}^{\mathrm{ac}}$\;
        Define $\mathcal{J}^{\mathrm{viol}}_{kt} := \{j \in \mathcal{J}^{\mathrm{ac}} \setminus \{j_k\} \,:\, \norm{(f_{kjt}(\overline{p}),\max \{|\overline{q}^{\mathrm{fr}}_{jt}|,|\overline{q}^{\mathrm{to}}_{jt}|\})}_2 - s^{\mathrm{max,ctg}}_j \geq s^{\mathrm{thr}}\}$\;
        Define $\overline{\mathcal{J}}_{kt} := \underset{j \in \mathcal{J}^{\mathrm{viol}}_{kt}}{\argmax}^{N^{\mathrm{br}}} \ \norm{(f_{kjt}(\overline{p}),\max \{|\overline{q}^{\mathrm{fr}}_{jt}|,|\overline{q}^{\mathrm{to}}_{jt}|\})}_2 - s^{\mathrm{max,ctg}}_j$\;
            Define $\overline{\mathcal{I}}_{kjt} := \underset{i \in \mathcal{I}}{\argmax}^{N^{\mathrm{bus}}} \left | F_{kji}  \left ( p^{\mathrm{inj}}_{it}(\overline{p}) - \frac{p^{\mathrm{sl}}_t(\overline{p})}{|\mathcal{I}|}  \right ) \right |$\ for all $j \in \overline{\mathcal{J}}_{kt}$;
        }
    \Return{$\overline{\mathcal{N}} := (\overline{\mathcal{I}},\overline{\mathcal{J}},\overline{\mathcal{K}})$}
\end{algorithm}

In practice, the candidate solution $(\overline{p},\overline{q})$ comes from the solution to the final subproblems~$Z^{\mathrm{F}}_t$.  After contingency screening around this solution, we solve a version of the final subproblems that includes the approximated contingency objectives:
\begin{equation*}
    \label{FSPwithCtg}
	\begin{aligned}
    Z^{\mathrm{F}}_t(\overline{p},\overline{q},\overline{u},\overline{\mathcal{N}};\delta) \ :=\  \max_{p_t,q_t,u_t} \quad & R^{\mathrm{F}}_t(p_t,q_t;\overline{p}) - \overline{C}^{\mathrm{ctg}}_t(p_t,q_t;\overline{\mathcal{N}})\\
    \text{s.t.} \quad & \eqref{constr:PIReal} - \eqref{constr:PIReactive},\, \eqref{eq:rampSafeConstraint},\ (p_t,q_t,u^{\mathrm{sh}}_t) \in \mathcal{X}^{\mathrm{ac}}_t,\ u = \overline{u}.
    \end{aligned}
\end{equation*}
For a fixed unit commitment solution, this allows the power injection decisions to adapt to impactful contingencies under the same restrictions as in the original final subproblems.

\subsection{Objective Function Approximation and Device Subproblem Aggregation}
\label{sec:improve_objective}

The objective function of \eqref{SC-ACOPF} is partitioned into the subproblems used in Algorithm~\ref{alg:basicAPM}, such that every objective term is contained in exactly one subproblem.  This precludes each subproblem from considering terms assigned to other problems, even if these terms are relevant to its decision variables. Here, we modify the subproblem objectives by reintroducing relevant terms. 

In the device subproblems $Z^{\mathrm{J}}_j$, we introduce an approximation of the power balance violation penalty functions $C^{\mathrm{bal}}_{it}$ (see Section~\ref{sec:modelBalance}).  To facilitate approximation, we aggregate devices colocated at the same bus $i$.  In practice, this strategy is effective because the device-level subproblems are MILPs of relatively small size, so the benefit of decomposition at the device level, compared to the bus level, is minimal.  This aggregation allows for the total net power injection at a bus to be modeled in the approximation. As the spatially-decomposed subproblems $Z^{\mathrm{J}}_j$ do not consider any power flows, we fix the line and shunt power flow variables $(p^{\mathrm{fr}},p^{\mathrm{to}},p^{\mathrm{sh}},q^{\mathrm{fr}},q^{\mathrm{to}},q^{\mathrm{sh}})$ to a solution from the temporally-decomposed subproblems $Z^{\mathrm{T}}_t$.  We then define the approximated power balance penalty functions at a fixed power flow solution $(\tilde{p},\tilde{q})$ by
\begin{equation*}
    \begin{aligned}
        && \overline{C}^{\mathrm{bal}}_{it}(p,q;\tilde{p},\tilde{q}) := d_t \left ( c^{\mathrm{p}}  \left | \sum_{\substack{j \in \mathcal{J}^{\mathrm{pr}}\\i_j = i}} p^{\mathrm{tot}}_{jt} - \sum_{\substack{j \in \mathcal{J}^{\mathrm{cs}}\\i_j = i}} p^{\mathrm{tot}}_{jt} - \sum_{\substack{j \in \mathcal{J}^{\mathrm{sh}}\\i_j = i}} \tilde{p}^{\mathrm{sh}}_{jt} - \sum_{\substack{j \in \mathcal{J}^{\mathrm{br}}\\i_j = i}} \tilde{p}^{\mathrm{fr}}_{jt} - \sum_{\substack{j \in \mathcal{J}^{\mathrm{br}}\\i'_j = i}} \tilde{p}^{\mathrm{to}}_{jt} \right | \right. &\\
        && \left. + c^{\mathrm{q}}  \left | \sum_{\substack{j \in \mathcal{J}^{\mathrm{pr}}\\i_j = i}} q^{\mathrm{tot}}_{jt} - \sum_{\substack{j \in \mathcal{J}^{\mathrm{cs}}\\i_j = i}} q^{\mathrm{tot}}_{jt} - \sum_{\substack{j \in \mathcal{J}^{\mathrm{sh}}\\i_j = i}} \tilde{q}^{\mathrm{sh}}_{jt} - \sum_{\substack{j \in \mathcal{J}^{\mathrm{br}}\\i_j = i}} \tilde{q}^{\mathrm{fr}}_{jt} - \sum_{\substack{j \in \mathcal{J}^{\mathrm{br}}\\i'_j = i}} \tilde{q}^{\mathrm{to}}_{jt} \right | \right ) & \quad \forall i \in \mathcal{I},\ t \in \mathcal{T}.
    \end{aligned}
\end{equation*}

Next, the temporal problems $Z^{\mathrm{T}}_t$ do not consider the cost or utility of a power injection $p^{\mathrm{tot}}_{jt}$.  We add the corresponding functions $R^{\mathrm{pow}}_{jt}$ and $C^{\mathrm{pow}}_{jt}$ to the objective.  
Similarly, the device-level problems $Z^{\mathrm{J}}_j$ do not have costs for reserve products; we add the corresponding linear cost terms $c^{\mathrm{res,p}}_{jtr} p^{\mathrm{res}}_{jtr}$ and $c^{\mathrm{res,q}}_{jtr} q^{\mathrm{res}}_{jtr}$ to the objective.

The startup and shutdown variables $u^{\mathrm{su}}_{jt}$ and $u^{\mathrm{sd}}_{jt}$ are implicitly included in the sets $\mathcal{Y}^{\mathrm{uc}}_{jt}$ to model feasible real and reactive power injections during startup and shutdown (see Section~\ref{sec:reformulate}).  However, since no objective cost is associated with these variables in the models $Z^{\mathrm{T}}_t$, they offer a path to generate power without incurring unit commitment cost.  For instance, by constraint \eqref{constr:RelaxPISUDef}, real power at time~$t$ may be produced by starting up in a future time period, which will incur an online and startup cost. These costs are not considered in the objectives $R^{\mathrm{T}}_t$.  In the penalty framework, producing power in such a way encourages device startups at a later time in the device-level models $Z^{\mathrm{J}}_j$; this solution may be suboptimal, especially for devices with high commitment costs.  

To remedy this issue, we reimpose a cost on startups and shutdowns by introducing aggregate variables $\chi^{\mathrm{su}}_{jt} = \sum_{t' \in \mathcal{T}^{\mathrm{supc}}_{jt}} u^{\mathrm{su}}_{jt'}$ and $\chi^{\mathrm{sd}}_{jt} = \sum_{t' \in \mathcal{T}^{\mathrm{sdpc}}_{jt}} u^{\mathrm{sd}}_{jt'}$ and rewriting relevant constraints in $\mathcal{Y}^{\mathrm{uc}}_{jt}$:
\begin{subequations}
    \label{constr:AltPI}
    \begin{align}
        & 0 \leq p^{\mathrm{su}}_{jt} \leq \chi^{\mathrm{su}}_{jt} \max_{t' \in \mathcal{T}^{\mathrm{supc}}_{jt}} p^{\mathrm{supc}}_{jtt'} \quad & \forall j \in \mathcal{J}^{\mathrm{sd}},\ t \in \mathcal{T} ,\label{constr:AltPISUDef}\\
        & 0 \leq p^{\mathrm{sd}}_{jt} \leq \chi^{\mathrm{sd}}_{jt} \max_{t' \in \mathcal{T}^{\mathrm{sdpc}}_{jt}} p^{\mathrm{sdpc}}_{jtt'} \quad & \forall j \in \mathcal{J}^{\mathrm{sd}},\ t \in \mathcal{T} ,\label{constr:AltPISDDef}\\
        & q^{\mathrm{tot}}_{jt} \leq q^{\mathrm{max}}_{jt} \left ( u^{\mathrm{on}}_{jt} + \chi^{\mathrm{su}}_{jt} + \chi^{\mathrm{sd}}_{jt} \right ) + \beta^{\mathrm{max}}_j p^{\mathrm{tot}}_{jt} - \begin{cases}
            q^{\mathrm{res}}_{jt\uparrow} & \text{if } j \in \mathcal{J}^{\mathrm{pr}}\\
            q^{\mathrm{res}}_{jt\downarrow} & \text{if } j \in \mathcal{J}^{\mathrm{cs}}
        \end{cases}  & \forall j \in \mathcal{J}^{\mathrm{sd}},\ t \in \mathcal{T}, \label{constr:AltPIReactReserveMax}\\
        & q^{\mathrm{tot}}_{jt} \geq q^{\mathrm{min}}_{jt} \left ( u^{\mathrm{on}}_{jt} + \chi^{\mathrm{su}}_{jt} + \chi^{\mathrm{sd}}_{jt} \right ) + \beta^{\mathrm{min}}_j p^{\mathrm{tot}}_{jt} + \begin{cases}
            q^{\mathrm{res}}_{jt\downarrow} & \text{if } j \in \mathcal{J}^{\mathrm{pr}}\\
            q^{\mathrm{res}}_{jt\uparrow} & \text{if } j \in \mathcal{J}^{\mathrm{cs}}
        \end{cases}  & \forall j \in \mathcal{J}^{\mathrm{sd}},\ t \in \mathcal{T}, \label{constr:AltPIReactReserveMin}\\
        & u^{\mathrm{on}}_{jt} + \chi^{\mathrm{su}}_{jt} \leq 1 & \forall j \in \mathcal{J}^{\mathrm{sd}},\ t \in \mathcal{T}, \label{constr:AltPISUMax}\\
        & u^{\mathrm{on}}_{jt} + \chi^{\mathrm{sd}}_{jt} \leq 1 & \forall j \in \mathcal{J}^{\mathrm{sd}},\ t \in \mathcal{T}, \label{constr:AltPISDMax}\\
        & (u^{\mathrm{on}}_{jt}, \chi^{\mathrm{su}}_{jt}, \chi^{\mathrm{sd}}_{jt}) \in [0,1]^3 & \forall j \in \mathcal{J}^{\mathrm{sd}},\ t \in \mathcal{T}. \label{constr:AltPIDomain}
    \end{align}
\end{subequations}
We denote the feasible region for this formulation by the set $\overline{\mathcal{Y}}^{\mathrm{uc}}_{jt} := \{(p_{jt},q_{jt},u^{\mathrm{on}}_{jt},\chi_{jt})\,:\, \eqref{constr:PIRealDef},\,\eqref{constr:PIDeviceReserveSet}-\eqref{constr:PINonneg},\,\eqref{constr:PIReactNonneg},\,\eqref{constr:AltPI}\}$.  Proposition~\ref{prop:AltPIEquivalent} gives that these regions are equivalent to the sets $\mathcal{Y}^{\mathrm{uc}}_{jt}$. 
\begin{proposition}
    \label{prop:AltPIEquivalent}
    For all $j \in \mathcal{J}^{\mathrm{sd}}$ and $t \in \mathcal{T}$, the sets $\mathcal{Y}^{\mathrm{uc}}_{jt}$ and $\overline{\mathcal{Y}}^{\mathrm{uc}}_{jt}$ are equivalent; that is, 
    $$\mathcal{Y}^{\mathrm{uc}}_{jt} = \{(p_{jt},q_{jt},u^{\mathrm{on}}_{jt})\,:\, \exists \chi_{jt} \ \mathrm{s.t.}\  (p_{jt},q_{jt},u^{\mathrm{on}}_{jt},\chi_{jt}) \in \overline{\mathcal{Y}}^{\mathrm{uc}}_{jt}\}.$$
\end{proposition}

As nonzero $\chi^{\mathrm{su}}_{jt}$ (resp.\ $\chi^{\mathrm{sd}}_{jt}$) implies a future startup (resp.\ past shutdown) and therefore implies that the device must be on in some other time period, we associate online and startup (resp.\ shutdown) costs with these variables.  The revised unit commitment cost functions are  
$$\overline{C}^{\mathrm{uc}}_{t}(u,\chi) := \sum_{j \in \mathcal{J}^{\mathrm{sd}}} c^{\mathrm{on}}_{jt}  u^{\mathrm{on}}_{jt} + \left ( \max_{t' \in \mathcal{T}^{\mathrm{supc}}_{jt}} c^{\mathrm{on}}_{jt'}+ c^{\mathrm{su}}_j \right )  \chi^{\mathrm{su}}_{jt} + \left ( \max_{t' \in \mathcal{T}^{\mathrm{sdpc}}_{jt}} c^{\mathrm{on}}_{j,t'-1} + c^{\mathrm{sd}}_{j} \right ) \chi^{\mathrm{sd}}_{jt} \quad \forall t \in \mathcal{T}.$$
These functions overestimate the cost coefficients of $\chi^{\mathrm{su}}_{jt}$ and $\chi^{\mathrm{sd}}_{jt}$ in two ways: first, downtime-dependent cost adjustments $c^{\mathrm{dd}}_{jtt'}$ are ignored; and second, the largest plausible commitment cost $c^{\mathrm{on}}_{jt}$ for future (resp.\ past) online status is used.

In summary, we define updated subproblem objective functions that implement these changes:
\begin{equation*}
    \begin{aligned}
        \overline{R}^{\mathrm{T}}_t(p,q,u,\chi) & := \sum_{j \in \mathcal{J}^{\mathrm{cs}}} R^{\mathrm{pow}}_{jt}(p) - \sum_{j \in \mathcal{J}^{\mathrm{pr}}}C^{\mathrm{pow}}_{jt}(p) - \sum_{j \in \mathcal{J}^{\mathrm{ac}}} C^{\mathrm{ac}}_{jt}(p,q) - \sum_{i \in \mathcal{I}} C^{\mathrm{bal}}_{it}(p,q) - \overline{C}^{\mathrm{uc}}_{t}(u,\chi) - C^{\mathrm{res}}_t(p,q),\\
        \overline{R}^{\mathrm{I}}_i(p,q,u;\tilde{p},\tilde{q}) & := \sum_{\substack{j \in \mathcal{J}^{\mathrm{sd}}\\i_j = i}}  \left ( R^{\mathrm{J}}_j(p,u) - \sum_{\substack{t \in \mathcal{T}\\r \in \mathcal{R}}} d_t c^{\mathrm{res,p}}_{jtr} p^{\mathrm{res}}_{jtr} - \sum_{\substack{t \in \mathcal{T}\\r \in \{\uparrow,\downarrow\}}} d_t c^{\mathrm{res,q}}_{jtr} q^{\mathrm{res}}_{jtr} \right ) - \sum_{t \in \mathcal{T}} \overline{C}^{\mathrm{bal}}_{it}(p,q;\tilde{p},\tilde{q}).
    \end{aligned}
\end{equation*}
Contingency screening adds contingency costs to the final subproblem objectives, so these costs are removed from the intermediate subproblem objectives $\overline{R}^{\mathrm{T}}_t$.  Additionally, reserve requirement and branch limit penalties are withheld from $\overline{R}^{\mathrm{I}}_i$ and energy penalties are withheld from $\overline{R}^{\mathrm{T}}_t$.  Computational experiments demonstrate that these omissions do not degrade solution quality.

\subsection{The Tailored pADM}

Now, we combine the adaptations discussed in the previous sections and present Algorithm~\ref{alg:tailoredAPM}. The algorithm uses new versions of the subproblems $Z^{\mathrm{T}}_t$ and $Z^{\mathrm{J}}_j$  defined as follows, where $u_i$ denotes the subset of variables $u$ that contain index $j$ with $i_j = i$:
\begin{subequations}
    \begin{align*}
        \overline{Z}^{\mathrm{T}}_t(\overline{p}_t,\overline{q}_t,\overline{u}_t;\rho) \ :=\ & 
        \begin{aligned}[t]
        \max_{\substack{p_t,q_t,u_t,\\\chi_t,\mu_t,\omega_t}} \quad & \overline{R}^{\mathrm{T}}_t(p_t,q_t,u_t,\chi_t) - \rho \Gamma'_2 (p_t,q_t,u_t,\overline{p}_t,\overline{q}_t,\overline{u}_t) - \rho \sum_{j \in \mathcal{J}^{\mathrm{sh}}} (\mu^{\mathrm{sh}}_{jt} - \overline{u}^{\mathrm{sh}}_{jt} \omega_{i_j t})^2\\
        \text{s.t.} \quad & \begin{aligned}[t]
            & (p_t,q_t,\mu_t,\omega_t) \in \mathcal{X}^{\mathrm{soc}}_t,\\
            &(p_{jt},q_{jt},u^{\mathrm{on}}_{jt},\chi_{jt}) \in \overline{\mathcal{Y}}^{\mathrm{uc}}_{jt} \quad & \forall j \in \mathcal{J}^{\mathrm{sd}};
        \end{aligned}
        \end{aligned}\\
        \overline{Z}^{\mathrm{I}}_i(p_i,q_i,u_i;\rho) \ :=\  & 
        \begin{aligned}[t]
        \max_{\overline{p}_i,\overline{q}_i,\overline{u}_i} \quad &  \overline{R}^{\mathrm{I}}_i(\overline{p}_i,\overline{q}_i,\overline{u}_i;p_i,q_i) - \rho \Gamma'_1(p_i,q_i,u_i,\overline{p}_i,\overline{q}_i,\overline{u}_i)\\
        \text{s.t.} \quad & (\overline{p}_j, \overline{q}_j, \overline{u}_j) \in \mathcal{X}^{\mathrm{uc}}_j \quad \forall j \in \mathcal{J}^{\mathrm{sd}} \text{ with } i_j = i.
        \end{aligned}
    \end{align*}
\end{subequations}

\begin{algorithm}[t]
    \caption{Tailored Penalty Alternating Direction Method}\label{alg:tailoredAPM}
    \SetKwInOut{Input}{Input}
    \Input{Penalty schedule $\{\rho_\tau\}_{\tau = 1}^{\overline{\tau}} \subset \mathbb{R}_{>}$, contingency thresholds $N^{\mathrm{ctg}},N^{\mathrm{br}},N^{\mathrm{bus}},s^{\mathrm{thr}}$}
    \For{$\tau \in \{1,\ldots,\overline{\tau}\}$}{
        \uIf{$\tau = 1$}{
            For $t \in \mathcal{T}$, solve $\overline{Z}^{\mathrm{T}}_t(0,0,0;0)$ to solution $(p^{(\tau)}_t,q^{(\tau)}_t,u^{(\tau)}_t,\chi^{(\tau)}_t,\mu^{(\tau)}_t,\omega^{(\tau)}_t)$\;
        }
        \Else{
            For $t \in \mathcal{T}$, solve $\overline{Z}^{\mathrm{T}}_t(\overline{p}^{(\tau-1)}_t,\overline{q}^{(\tau-1)}_t,\overline{u}^{(\tau-1)}_t;\rho_{\tau})$ to solution $(p^{(\tau)}_t,q^{(\tau)}_t,u^{(\tau)}_t,\chi^{(\tau)}_t,\mu^{(\tau)}_t,\omega^{(\tau)}_t)$\;
        }
        For $i \in \mathcal{I}$, solve $\overline{Z}^{\mathrm{I}}_i(p^{(\tau)}_i,q^{(\tau)}_i,u^{(\tau)}_i;\rho_{\tau})$ to solution $(\overline{p}^{(\tau)}_i,\overline{q}^{(\tau)}_i,\overline{u}^{(\tau)}_i)$\;
        For $j \in \mathcal{J}^{\mathrm{sh}}$ and $t \in \mathcal{T}$, compute ${\overline{u}^{\mathrm{sh} (\tau)}_{jt}} = Z^{\mathrm{SH}}_{jt}( \mu^{\mathrm{sh} (\tau)}_{jt}/\omega_{i_j t}^{(\tau)})$\;
    }
    For $j \in \mathcal{J}^{\mathrm{sd}}$, solve $H_j(\overline{p}^{(\overline{\tau})}_j,\overline{u}^{(\overline{\tau})}_j)$ to solution $(\delta^-_j, \delta^+_j)$\;
    For $t \in \mathcal{T}$, solve $Z^{\mathrm{F}}_t(\overline{p}^{(\overline{\tau})},\overline{q}^{(\overline{\tau})},\overline{u}^{(\overline{\tau})},\emptyset;\delta)$ to solution $(p^{\mathrm{F}}_t,q^{\mathrm{F}}_t,u^{\mathrm{F}}_t)$\;
    Compute $\overline{\mathcal{N}} = \verb|Contingency Screening|(p^{\mathrm{F}},q^{\mathrm{F}},N^{\mathrm{ctg}},N^{\mathrm{br}},N^{\mathrm{bus}},s^{\mathrm{thr}})$, using Algorithm~\ref{alg:contingency}\;
    For $t \in \mathcal{T}$, solve $Z^{\mathrm{F}}_t(\overline{p}^{(\overline{\tau})},\overline{q}^{(\overline{\tau})},\overline{u}^{(\overline{\tau})},\overline{\mathcal{N}};\delta)$ to solution $(p_t,q_t,u_t)$\;
    \Return{$(p,q,u)$}
\end{algorithm}

First, the algorithm alternates between solving the SOC relaxations of the temporally decomposed subproblems and the spatially decomposed bus-level subproblems.  Then, the power injection and unit commitment decisions from the final-iteration bus-level models $\overline{Z}^{\mathrm{I}}_i$ are used to compute ramping intervals for constraints \eqref{eq:rampSafeConstraint}.  These bounds and the last iterate generate the final ACOPF subproblems, which are solved without contingencies.  In the last step, the contingencies are screened based on the final subproblem solution. The final subproblems are solved again, incorporating costs from the selected contingencies, to generate a candidate solution to \eqref{SC-ACOPF}. We show the validity of this algorithm in Theorem~\ref{thm:tailoredAlgorithmFeasible}.

\begin{theorem}
    \label{thm:tailoredAlgorithmFeasible}
    Let $(p,q,u)$ be a solution returned by Algorithm~\ref{alg:tailoredAPM}.  Then, there is some $(v,\Delta,\theta,\tau,\phi)$ such that $(p,q,u,v,\Delta,\theta,\tau,\phi)$ is feasible for \eqref{SC-ACOPF}.
\end{theorem}

\section{A Temporally-Decomposed Dual Bound}
\label{sec:upperbound}

The temporally-decomposed subproblems $Z^{\mathrm{T}}_t$, combined with the SOC relaxation proposed in Section~\ref{sec:improve_soc}, yield a set of reasonably sized convex subproblems that can be efficiently solved to global optimality.  We slightly adapt these subproblems so that they generate an upper (dual) bound on the objective value of the problem \eqref{SC-ACOPF}.  Access to such a bound on the primal objective is a useful tool for computational analysis of our algorithms.  Specifically, the upper bound allows for certification of the quality of a primal solution, in that feasible solutions with objective values near the upper bound cannot be significantly improved. We define the upper bound as
\begin{equation}
\tag{UB}
\label{eq:SOCUB}
    \begin{aligned}
        Z^{\mathrm{UB}}\ :=\ \sum_{t \in \mathcal{T}} && \max_{\substack{p_t,q_t,u_t,\\\mu_t,\omega_t}} \quad & C^{\mathrm{UB}}_t (p_t,q_t,u_t)\\
        && \text{s.t.} \quad & (p_t,q_t,\mu_t,\omega_t) \in \mathcal{X}^{\mathrm{soc}}_t,\ (p_{jt},q_{jt},u^{\mathrm{on}}_{jt}) \in \mathcal{Y}^{\mathrm{uc}}_{jt} \quad \forall j \in \mathcal{J}^{\mathrm{sd}},
    \end{aligned}
\end{equation}
where
$$C^{\mathrm{UB}}_t (p,q,u) := \sum_{j \in \mathcal{J}^{\mathrm{cs}}} R^{\mathrm{pow}}_{jt}(p) - \sum_{j \in \mathcal{J}^{\mathrm{pr}}}C^{\mathrm{pow}}_{jt}(p) - \sum_{j \in \mathcal{J}^{\mathrm{ac}}} C^{\mathrm{ac}}_{jt}(p,q) - \sum_{i \in \mathcal{I}} C^{\mathrm{bal}}_{it}(p,q) - \sum_{j \in \mathcal{J}^{\mathrm{sd}}} c^{\mathrm{on}}_{jt} u^{\mathrm{on}}_{jt} - C^{\mathrm{res}}_t(p,q).$$
As the feasible set relaxes that of \eqref{SC-ACOPF} and the objective overestimates the true objective, $Z^\mathrm{UB}$ is an upper bound on $Z^*$. Theorem~\ref{thm:SOCUpperBound} formalizes this result.  This bound relaxes several aspects of the primal problem.  Specifically, penalties from contingencies, energy limit violations, and device startup/shutdown costs are ignored, ramping constraints are removed, integrality constraints are relaxed, and AC power flows are replaced with their SOC relaxation.

\begin{theorem}
    \label{thm:SOCUpperBound}
    The problem \eqref{eq:SOCUB} generates an upper bound on \eqref{SC-ACOPF}; that is, $Z^* \leq Z^{\mathrm{UB}}$.  
\end{theorem}

\section{Computational Results}
\label{sec:results}

We implement Algorithms~\ref{alg:contingency}~and~\ref{alg:tailoredAPM} and evaluate them on the test cases provided in the fourth and final event of GOC3 \citep{elbert2024godata}.  Exact details of our implementation are provided in Section~\ref{sec:result_details}.  The test cases mimic realistic grid data under a diverse set of operating conditions.  The dataset contains 584 cases across~8 networks which vary in scale from 73~buses to 8,316~buses.  We refer to each network by the number of buses it contains. The cases are separated into three divisions, each with a different lookahead window.  Cases in Division~1 (D1) are real-time market optimization problems with an 8 hour lookahead, Division~2 (D2) contains day-ahead problems with a two day lookahead, and Division~3 (D3) contains advisory problems with a one week lookahead.  D1 cases span 18 time periods ($T = 18$), D2 cases 48 periods, and D3 cases 42 periods.  Table~\ref{table:caseBreakdown} shows the number of cases by network size and division.  Cases with the same network and division may differ in other parameters, such as generation and load profiles.
\begin{table}[tp]
    \centering
    \begin{tabular}{c|rrrrrrrr|r}
         \# Buses & 73 & 617 & 1,576 & 2,000 & 4,224 & 6,049 & 6,717 & 8,316 & Total \\
        \hline
         D1 & 24 & 38 & 24 & 18 & 24 & 36 & 36 & 48 & 248\\
         D2 & 40 & 24 & 0 & 18 & 24 & 18 & 18 & 24 & 166\\
         D3 & 24 & 24 & 24 & 3 & 24 & 24 & 3 & 44 & 170\\
         \hline
         Total & 88 & 86 & 48 & 39 & 72 & 78 & 57 & 116 & 584
    \end{tabular}
    \caption{Number of test cases by network size and division.}
    \label{table:caseBreakdown}
\end{table}
\subsubsection*{Computing Resources}
Our experiments are conducted in Julia 1.9 with modeling language JuMP \citep{Lubin2023jump} on the MIT SuperCloud \citep{reuther2018interactive}.  Each case runs on a single 48-core Intel Xeon Platinum 8260@2.40 GHz processor with 192 GB of RAM.  The (mixed-integer) linear bus-level models $\overline{Z}^{\mathrm{I}}_i$ and ramping restriction problems \eqref{eq:rampSafeLP} are solved with Gurobi 10.0, the temporally-decomposed SOC relaxations $\overline{Z}^{\mathrm{T}}_t$ are solved with Mosek 10.0, and the final subproblems $Z^{\mathrm{F}}_t$ are solved with Ipopt 3.14.10 \citep{wachter2006implementation} using linear solver \verb|ma57|.  Feasibility is evaluated with a tolerance of $10^{-8}$; all subproblems are solved with this tolerance.

\subsection{Implementation Details}
\label{sec:result_details}
The decomposed models are solved in parallel on 48 threads.  This allows for the temporally decomposed models ($\overline{Z}^{\mathrm{T}}_t$ or $Z^{\mathrm{F}}_t$) to be solved simultaneously across all time periods, while bus-level subproblems $\overline{Z}^{\mathrm{I}}_i$ are batched into threads.  The default scheme runs 10 iterations of Algorithm~\ref{alg:tailoredAPM} ($\overline{\tau} = 10$) with $\rho_{\tau} = 10^{\tau-1}$.  We adapt the number of iterations and penalty coefficients to conform with time limits, as discussed in the following paragraph.

In accordance with GOC3, a strict runtime limit is imposed: 10 minutes for D1, two hours for D2, and four hours for D3.  As our algorithm runs, we estimate the maximum number of iterations that can be completed within the remaining time and distribute the remaining penalty coefficients evenly on a log scale between the current coefficient value and $10^9$.  The iteration count may not exceed the default value of 10.  We report the average number of completed iterations by network and division in Table~\ref{table:iterationAndContingencyCount}.  For every case in D2 and D3, there is sufficient time to run 10 iterations.  In D1, for networks with over 6,000 buses, the iteration count is often decreased, whereas smaller cases typically run all 10 iterations.  If the time limit is reached while solving the final subproblems $Z^{\mathrm{F}}_t$, we return the current iterate and evaluate its feasibility.  Solutions with negative objective values are considered infeasible.

\begin{table}[tp]
    \centering
    \begin{tabular}{c|cccccccc|c}
         Network & 73 & 617 & 1,576 & 2,000 & 4,224 & 6,049 & 6,717 & 8,316 & Average \\
        \hline
         D1 & $10\,/\,0.83$ & $10\,/\,1$ & $10\,/\,1$ & $9.6\,/\,0.72$ & $10\,/\,0$ & $4.8\,/\,0$ & $2.8\,/\,0$ & $3.7\,/\,0$ & $7\,/\,0.38$\\
         D2 & $10\,/\,0.90$ & $10\,/\,1$ & ---$\,/\,$--- & $10\,/\,1$ & $10\,/\,1$ & $10\,/\,0.94$ & $10\,/\,0.72$ & $10\,/\,0$ & $10\,/\,0.80$\\
         D3 & $10\,/\,1$ & $10\,/\,1$ & $10\,/\,1$ & $10\,/\,1$ & $10\,/\,1$ & $10\,/\,1$ & $10\,/\,1$ & $10\,/\,0.02$ & $10\,/\,0.75$\\
         \hline
         Average & $10\,/\,0.91$ & $10\,/\,1$ & $10\,/\,1$ & $9.8\,/\,0.87$ & $10\,/\,0.67$ & $7.6\,/\,0.53$ & $5.5\,/\,0.28$ & $7.4\,/\,0.01$ & $8.7\,/\,0.61$
    \end{tabular}
    \caption{Average number of iterations / proportion of solutions that consider screened contingencies by network size and division. No 1,576-bus D2 cases are in the dataset.}
    \label{table:iterationAndContingencyCount}
\end{table}

The contingency screening thresholds are selected by validation on a set of cases from previous competition events.  The values differ between large and small networks and are set as follows:
\begin{equation*}
    \begin{aligned}
        & \text{Small Networks (73 and 617 buses): } & N^{\mathrm{ctg}} = 30, \quad N^{\mathrm{br}} = 20, \quad N^{\mathrm{bus}} = 50, \quad s^{\mathrm{thr}} = 10^{-3},\\
        & \text{Large Networks ($\geq\!$ 1,576 buses): } & N^{\mathrm{ctg}} = 30, \quad N^{\mathrm{br}} = 20, \quad N^{\mathrm{bus}} = 20, \quad s^{\mathrm{thr}} = 10^{-2}.
    \end{aligned}
\end{equation*}
If sufficient time is available, we solve the final models $Z^{\mathrm{F}}_t$ with screened contingencies in step~13 of Algorithm~\ref{alg:tailoredAPM}; otherwise, we evaluate the solution to the contingency-free models from step~11.  Table~\ref{table:iterationAndContingencyCount} gives the proportion of cases which include screened contingencies in the evaluated solution.  Cases that do not include contingencies do not have sufficient remaining time for contingency screening and post-screening optimization, do not yield any valid contingencies via Algorithm~\ref{alg:contingency}, or fail to demonstrate a sufficient objective increase in the post-contingency solution relative to the contingency-free solution.
In D1, time limits are binding for large cases, so these cases typically do not consider contingencies. In D2 and D3, most cases consider contingencies.
Across all divisions, we identify impactful contingencies on only a few instances from the 8,316-bus network.  

\subsection{Solution Quality}
\label{sec:results_quality}

In this section, we investigate the quality of the feasible solution returned by Algorithm~\ref{alg:tailoredAPM}.  
Our decomposition scheme yields dual bounds via problem \eqref{eq:SOCUB}.  We evaluate feasible solutions to \eqref{SC-ACOPF} against this upper bound, quantifying the optimality gap of the solution.  From the primal perspective, we compare our solutions against those generated by the benchmark algorithm of \citeauthor{parker2024benchmark} (\citeyear{parker2024benchmark}, ``Benchmark'').  The benchmark is subject to the time limits described in Section~\ref{sec:result_details}.
It solves (MI)LPs with Gurobi and nonlinear problems with Ipopt, using linear solver \verb|ma27|.
We highlight several key differences between the benchmark and our algorithm: first, we iterate between unit commitment and power flow problems to fix a unit commitment solution instead of relying on a single pass; second, we use \textit{ex-ante} bounds to incorporate ramping constraints into the final subproblems instead of projecting onto the feasible set; third, we include reserve and security constraints in the final models.

Figure~\ref{fig:ecdf_scores} compares the quality of feasible solutions to the upper bound.  We compute the percent optimality gap of a feasible solution with objective $z$ by $100 \times \frac{Z{^\mathrm{UB}} - z}{Z^{\mathrm{UB}}}$, where infeasible solutions are assigned a gap of $100\%$.  The figure shows the empirical probability that a test case has a gap no larger than some value, with results reported on the solutions returned by the tailored pADM (Algorithm~\ref{alg:tailoredAPM}) and the benchmark.  The figure also shows distributions for modified versions of our algorithm; these results are discussed in Section~\ref{sec:resultsAblation}

\begin{figure}[tp]
    \centering
    \begin{subfigure}[h]{\linewidth}
    \begin{tikzpicture}[font=\scriptsize]
        \begin{axis}[
          xlabel={Gap (\%)},
          ylabel={Empirical Probability},
          xmin = 0, xmax= 101,
          ymin= 0, ymax= 1,
          ytick distance=0.1, 
          xtick distance=10,
          grid=both,
          legend style={
            cells={anchor=west},
            legend pos=south east,
            }, 
          width=\textwidth,
          height=\axisdefaultheight
        ]
        \addplot+[blue, very thick, no marks, const plot mark left, opacity=1] table [x=tg_ub, y=probability, col sep=space] {data_ecdf_scores_sc.txt};
        \addplot+[cyan, densely dashed, very thick, no marks, const plot mark left, opacity=0.7] table [x=tg_noctg, y=probability, col sep=space] {data_ecdf_scores_sc.txt};
        \addplot+[violet, densely dashdotdotted, very thick, no marks, const plot mark left, opacity=0.7] table [x=tg_noflex, y=probability, col sep=space] {data_ecdf_scores_sc.txt};
        \addplot+[orange, densely dotted, very thick, no marks, const plot mark left, opacity=1] table [x=bm_ub, y=probability, col sep=space] {data_ecdf_scores_sc.txt};
        \legend{Tailored pADM,
                without Screened Contingency,
                without Ramping Restriction,
                Benchmark
                }
        \draw[draw=none, fill=black, fill opacity=.1] (0,0) rectangle (2,1);
        \draw[draw=none, fill=black, fill opacity=.1] (0,0.8) rectangle (102,1);
        \end{axis}
    \end{tikzpicture}
    \end{subfigure}
    \\
    \begin{subfigure}[h]{0.47\linewidth}
    \centering
    \begin{tikzpicture}[font=\scriptsize]
        \begin{axis}[
          xlabel={Gap (\%)},
          ylabel={Empirical Probability},
          width=\textwidth,
          height=0.9\textwidth,
          xmin = 0, xmax= 2.0,
          ymin= 0, ymax= 1,
          ytick distance=0.1, 
          xtick distance=0.2,
          grid=both,
          legend style={
            cells={anchor=west},
            legend pos=south east,
            }, 
        ]
        \addplot+[blue, very thick, no marks, const plot mark left, opacity=1] table [x=tg_ub, y=probability, col sep=space] {data_ecdf_scores_sc.txt};
        \addplot+[cyan, densely dashed, very thick, no marks, const plot mark left, opacity=0.7] table [x=tg_noctg, y=probability, col sep=space] {data_ecdf_scores_sc.txt};
        \addplot+[violet, densely dashdotdotted, very thick, no marks, const plot mark left, opacity=0.7] table [x=tg_noflex, y=probability, col sep=space] {data_ecdf_scores_sc.txt};
        \addplot+[orange, densely dotted, very thick, no marks, const plot mark left, opacity=1] table [x=bm_ub, y=probability, col sep=space] {data_ecdf_scores_sc.txt};
        \end{axis}
    \end{tikzpicture}
    \end{subfigure}
    \hfill
    \begin{subfigure}[h]{0.47\linewidth}
    \centering
    \begin{tikzpicture}[font=\scriptsize]
        \begin{axis}[
          xlabel={Gap (\%)},
          width=\textwidth,
          height=0.9\textwidth,
          xmin = 0, xmax= 102,
          ymin= 0.8, ymax= 1,
          ytick distance=0.02, 
          grid=both,
          legend style={
            cells={anchor=west},
            legend pos=south east,
            font=\tiny,
            }, 
        ]
        \addplot+[blue, very thick, no marks, const plot mark left, opacity=1] table [x=tg_ub, y=probability, col sep=space] {data_ecdf_scores_sc.txt};
        \addplot+[cyan, densely dashed, very thick, no marks, const plot mark left, opacity=0.7] table [x=tg_noctg, y=probability, col sep=space] {data_ecdf_scores_sc.txt};
        \addplot+[violet, densely dashdotdotted, very thick, no marks, const plot mark left, opacity=0.7] table [x=tg_noflex, y=probability, col sep=space] {data_ecdf_scores_sc.txt};
        \addplot+[orange, densely dotted, very thick, no marks, const plot mark left, opacity=1] table [x=bm_ub, y=probability, col sep=space] {data_ecdf_scores_sc.txt};
        \end{axis}
    \end{tikzpicture}
    \end{subfigure}
    \caption{(top) Empirical cumulative distribution function of the optimality gap.  Distributions are shown for feasible solutions obtained from the tailored pADM, tailored pADM without adding contingencies, tailored pADM without the ramping restriction, and the benchmark algorithm. (left) Distribution on the gap interval of $0\% - 2\%$, corresponds to the leftmost shaded region of (top).  (right) Distribution on the upper quintile, corresponds to the uppermost shaded region of (top).
    }
    \label{fig:ecdf_scores}
\end{figure}

The average optimality gap for feasible solutions returned by Algorithm~\ref{alg:tailoredAPM} is 1.33\%, compared to an average gap of 6.93\% for benchmark solutions.  Our algorithm fails to identify a feasible solution within the time limit for 19 of the 584 cases, while the benchmark gives infeasible solutions for 13 cases.  This infeasibility is due to early termination by the time limit; for our algorithm, this only occurs in~D1.  The 80th percentile gap for our algorithm is 1.06\%, compared to 11.83\% for the benchmark.  At the 90th percentile, the corresponding values are 1.54\% and 21.88\%.  On average, our solution objective improves over the benchmark by 6.94\%.  We note that there are four cases with gaps $\geq 80\%$ for both algorithms.  These cases are among those used to evaluate the impact of switching off AC lines, which is not permitted in our model.  Removing these cases reduces the average gap for our algorithm to 0.70\% and for the benchmark to 6.34\%.  Overall, these results demonstrate that our algorithm identifies a near-optimal solution in almost every case and significantly improves upon the benchmark solutions.

\subsection{Algorithm Runtime}
Although we enforce a maximum runtime, our algorithm often terminates before this limit.  Figure~\ref{fig:ecdf_times} shows the empirical probability that a case terminates within some amount of time, by division and network.  For small and medium size networks ($\leq\!$ 4,224 buses), the time limits are rarely binding; that is, the cases can be solved well within the imposed time limit.  For large cases in D1, especially the 6,049 and 6,717-bus networks, the 10 minute time limit is often binding and causes early termination.  In D2, the time limit is only binding for about $15\%$ of cases on the 6,717-bus network; for other networks, the time limit is not binding.  In D3, the time limits are never binding; all cases are solved in under half of the allotted time.

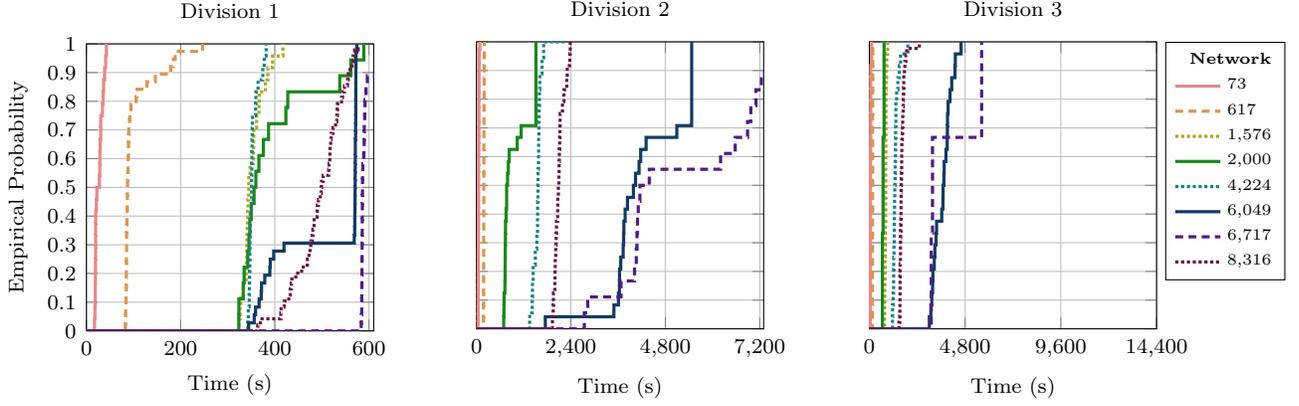
\begin{figure}[tp]
\begin{subfigure}[h]{0.33\linewidth}
\centering
\begin{tikzpicture}[font=\scriptsize]
    \begin{axis}[
      xlabel={Time (s)},
      ylabel={Empirical Probability},
      title={Division 1},
      width=5.4cm,
      height=5.4cm,
      xmin = 0, xmax= 610,
      ymin= 0, ymax= 1,
      ytick distance=0.1, 
      grid=both,
      legend style={
        cells={anchor=west},
        legend pos=south east,
        }, 
    ]
    \addplot+[red7, solid, very thick, no marks, const plot mark left, opacity=1] table [x=N00073D1, y=24, col sep=space] {data_ecdf_time_d1.txt};
    \addplot+[brown7, densely dashed, very thick, no marks, const plot mark left, opacity=1] table [x=N00617D1, y=38, col sep=space] {data_ecdf_time_d1.txt};
    \addplot+[yellow7, densely dotted, very thick, no marks, const plot mark left, opacity=1] table [x=N01576D1, y=24, col sep=space] {data_ecdf_time_d1.txt};
    \addplot+[green5, solid, very thick, no marks, const plot mark left, opacity=1] table [x=N02000D1, y=18, col sep=space] {data_ecdf_time_d1.txt};
    \addplot+[cyan5, densely dotted, very thick, no marks, const plot mark left, opacity=1] table [x=N04224D1, y=24, col sep=space] {data_ecdf_time_d1.txt};
    \addplot+[azure2, solid, very thick, no marks, const plot mark left, opacity=1] table [x=N06049D1, y=36, col sep=space] {data_ecdf_time_d1.txt};
    \addplot+[violet2, densely dashed, very thick, no marks, const plot mark left, opacity=1] table [x=N06717D1, y=36, col sep=space] {data_ecdf_time_d1.txt};
    \addplot+[purple2, densely dotted, very thick, no marks, const plot mark left, opacity=1] table [x=N08316D1, y=48, col sep=space] {data_ecdf_time_d1.txt};
    \end{axis}
\end{tikzpicture}
\end{subfigure}
\hfill
\begin{subfigure}[h]{0.30\linewidth}
\centering
\begin{tikzpicture}[font=\scriptsize]
    \begin{axis}[
      xlabel={Time (s)},
      title={Division 2},
      width=5.4cm,
      height=5.4cm,
      xmin = 0, xmax= 7300,
      ymin= 0, ymax= 1,
      ytick distance=0.1,
      xtick distance=2400,
      yticklabel = \empty,
      grid=both,
    ]
    \addplot+[red7, solid, very thick, no marks, const plot mark left, opacity=1] table [x=N00073D2, y=40, col sep=space] {data_ecdf_time_d2.txt};
    \addplot+[brown7, densely dashed, very thick, no marks, const plot mark left, opacity=1] table [x=N00617D2, y=24, col sep=space] {data_ecdf_time_d2.txt};
    \addplot+[green5, solid, very thick, no marks, const plot mark left, opacity=1] table [x=N02000D2, y=24, col sep=space] {data_ecdf_time_d2.txt};
    \addplot+[cyan5, densely dotted, very thick, no marks, const plot mark left, opacity=1] table [x=N04224D2, y=18, col sep=space] {data_ecdf_time_d2.txt};
    \addplot+[azure2, solid, very thick, no marks, const plot mark left, opacity=1] table [x=N06049D2, y=24, col sep=space] {data_ecdf_time_d2.txt};
    \addplot+[violet2, densely dashed, very thick, no marks, const plot mark left, opacity=1] table [x=N06717D2, y=18, col sep=space] {data_ecdf_time_d2.txt};
    \addplot+[purple2, densely dotted, very thick, no marks, const plot mark left, opacity=1] table [x=N08316D2, y=24, col sep=space] {data_ecdf_time_d2.txt};
    \end{axis}
\end{tikzpicture}
\end{subfigure}
\hfill
\begin{subfigure}[h]{0.35\linewidth}
\centering
\begin{tikzpicture}[font=\scriptsize]
    \begin{axis}[
      xlabel={Time (s)},
      title={Division 3},
      width=5.4cm,
      height=5.4cm,
      xmin = 0, xmax= 14400,
      ymin= 0, ymax= 1,
      ytick distance=0.1,
      xtick distance=4800,
      scaled x ticks = false,
      yticklabel = \empty,
      grid=both,
      legend style={
        cells={anchor=west},
        legend pos=outer north east,
        font=\tiny,
        smooth
        }, 
    ]
    \addlegendimage{empty legend}
    \addplot+[red7, solid, very thick, no marks, const plot mark left, opacity=1] table [x=N00073D3, y=24, col sep=space] {data_ecdf_time_d3.txt};
    \addplot+[brown7, densely dashed, very thick, no marks, const plot mark left, opacity=1] table [x=N00617D3, y=24, col sep=space] {data_ecdf_time_d3.txt};
    \addplot+[yellow7, densely dotted, very thick, no marks, const plot mark left, opacity=1] table [x=N01576D3, y=24, col sep=space] {data_ecdf_time_d3.txt};
    \addplot+[green5, solid, very thick, no marks, const plot mark left, opacity=1] table [x=N02000D3, y=3, col sep=space] {data_ecdf_time_d3.txt};
    \addplot+[cyan5, densely dotted, very thick, no marks, const plot mark left, opacity=1] table [x=N04224D3, y=24, col sep=space] {data_ecdf_time_d3.txt};
    \addplot+[azure2, solid, very thick, no marks, const plot mark left, opacity=1] table [x=N06049D3, y=24, col sep=space] {data_ecdf_time_d3.txt};
    \addplot+[violet2, densely dashed, very thick, no marks, const plot mark left, opacity=1] table [x=N06717D3, y=3, col sep=space] {data_ecdf_time_d3.txt};
    \addplot+[purple2, densely dotted, very thick, no marks, const plot mark left, opacity=1] table [x=N08316D3, y=44, col sep=space] {data_ecdf_time_d3.txt};
    \legend{
    \textbf{\hspace{-.6cm} Network}\\
    73\\
    617\\
    1,576\\
    2,000\\
    4,224\\
    6,049\\
    6,717\\
    8,316\\
    }
    \end{axis}
\end{tikzpicture}
\end{subfigure}
\caption{Empirical cumulative distribution function of solution time by division and network.}
\label{fig:ecdf_times}
\end{figure}

Figure~\ref{fig:time_breakdown} shows the distribution of time spent solving each subproblem class as a proportion of total time.  This runtime includes the total time for the bus-level models $\overline{Z}^{\mathrm{I}}_i$ in the iterative portion (UC Iterations), for the temporal models $\overline{Z}^{\mathrm{T}}_t$ in the iterative portion (SOC Iterations), and for the final models $Z^{\mathrm{F}}_t$ (Final Models).  The figure shows data only for cases that run all 10 iterations and are not terminated early due to the time limit.  Runtime includes neither the time for contingency screening nor for solving final models with contingencies.

\begin{figure}[tp]
\centering
\begin{tikzpicture}[font=\scriptsize]
\begin{axis}[
    xlabel = {Network (Average Runtime [s])},
    ylabel = Runtime Proportion,
    width=.85\textwidth,
    height=5cm,
    ybar stacked,
    bar width = .04\textwidth,
    grid = both,
    xmajorgrids=false,
    xminorgrids=false,
    xtick = data,
    xtick style={draw=none},
    xticklabels={
        {73\\(27s)},
        {617\\(68s)},
        {1,576\\(380s)},
        {2,000\\(366s)},
        {4,224\\(540s)},
        {6,049\\(1141s)},
        {6,717\\(2663s)},
        {8,316\\(1379s)}
    },
    xticklabel style={align=center},
    ymin= 0, ymax= 1,
    legend style={
        cells={anchor=west},
        legend pos=outer north east,
        font=\tiny
        }, 
    reverse legend,
]
    \addplot[blue, fill=blue!40] table [y=ac, meta = Network, x expr=\coordindex, col sep=space] {data_time_cumulative.txt} ;
    \addlegendentry{Final Models}
    \addplot[red, pattern=north east lines, pattern color = red] table [y=soc, meta = Network, x expr=\coordindex, col sep=space] {data_time_cumulative.txt} ;
    \addlegendentry{SOC Iterations}
    \addplot[violet, pattern=crosshatch dots, pattern color=violet] table [y=uc, meta = Network, x expr=\coordindex, col sep=space] {data_time_cumulative.txt} ;
    \addlegendentry{UC Iterations}
    \end{axis}
\end{tikzpicture}
\caption{
Proportion of runtime allocated to UC iterations, SOC iterations, and final models on cases for which the time limit is not binding and 10 iterations are run ($N = 457$).  Runtime is totaled across cases with the same network and across iterations for UC and SOC models.  Average runtime for each network is given in parentheses.
}
\label{fig:time_breakdown}
\end{figure}
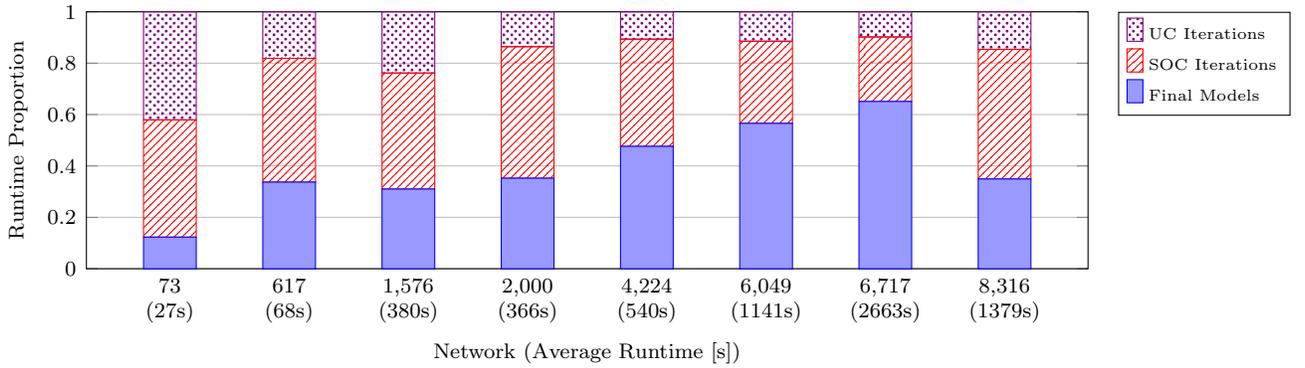

UC iterations occupy a small portion of the total runtime, especially on large networks.  Large networks experience a growth in the proportion of time spent solving the final models; that is, the time spent on UC and SOC iterations does not grow as rapidly as the time spent on the final models.  This suggests that efforts to reduce the runtime of our algorithm on large networks are best spent on accelerating the final subproblems.  On the other hand, for small networks, a majority of the time is spent in the UC and SOC iterations; runtime reductions for these cases can be achieved by decreasing the total number of iterations.

\subsection{Benefits of the Tailored Algorithm}
\label{sec:resultsAblation}

We now evaluate the impact of the ramping restriction (Section~\ref{sec:safeRamp}), the energy limit restriction (Section~\ref{sec:energyMinMax}), the SOC relaxation (Section~\ref{sec:improve_soc}), and contingency screening (Section~\ref{sec:improve_contingency}).

\subsubsection{Analysis of the Ramping Constraint Restriction} \label{sec:resultsAblationSafeRamp}
First, we demonstrate the quality of the ramping intervals generated by the auxiliary linear programs \eqref{eq:rampSafeLP}. Table~\ref{tab:rampsafe} compares the linear programs to other models that yield valid restrictions of the ramping constraint.  These models are evaluated on the number of devices and periods for which the ramping bounds yield a near-empty interval, the total length of the intervals, and a proxy for the volume generated by the intervals, computed as the sum of log interval lengths.  Directly maximizing the volume (i.e., the first row of Table~\ref{tab:rampsafe}) gives the best realization of the volume metric; however, this (convex) nonlinear problem has a nontrivial solve time. Maximizing the sum of interval lengths with the linear program \eqref{eq:rampSafeObj}--\eqref{constr:rampSafeUp} speeds up computation but generates an overly sparse solution with many near-empty intervals, even when using an interior point solver without crossover.  Such a solution reduces the amount of power injection flexibility in the final models.  The solution at equality for the heuristic expressions \eqref{constr:rampSafeDownHeur}--\eqref{constr:rampSafeUpHeur} distributes slack to the ramping bounds evenly between contiguous periods.  This yields intervals which are relatively non-sparse and a volume that is within 99.9\% of the upper bound, with a reduced total interval length. Adding the heuristic constraints to the linear program \eqref{eq:rampSafeObj}--\eqref{constr:rampSafeUpHeur} combines the benefits: a relatively fast method that generates near-optimal intervals by the sum and volume metrics and yields few near-empty intervals.

\begin{table}[tp]
    \centering
    \begin{tabular}{r|c|r|r|r}
        Model & Time (s)
        & $|\{j,t: \delta_{jt} \leq 10^{-8}\}|$
        &
        $\sum_{j,t} \delta_{jt} \times 10^{-3}$
        &
        $\sum_{j,t} v(\delta_{jt}) \times 10^{-3}$ 
\\
\hline 
        $\max \sum_{j,t} v(\delta_{jt})$ s.t. \eqref{constr:rampSafeDown}--\eqref{constr:rampSafeUp} & 9.1 & 354 (0.5\%) & 99.87 (100\%) & 42.35 (100\%)
\\
        LP \eqref{eq:rampSafeObj}--\eqref{constr:rampSafeUp}  &  0.2 & 7,874 (11.7\%) & 99.87 (100\%) & 36.67 (86.6\%)
\\
        Heuristic \eqref{constr:rampSafeDownHeur}--\eqref{constr:rampSafeUpHeur} & 0.1 & 325 (0.5\%) & 89.70 (89.8\%) & 42.28 (99.9\%)
\\
        LP \eqref{eq:rampSafeObj}--\eqref{constr:rampSafeUpHeur} 
         &  0.2 & 319 (0.5\%) & 99.80 (99.9\%) & 42.31 (99.9\%)
    \end{tabular}
    \caption{Comparison of ramping interval quality on a representative 6,049-bus case. Interval lengths are defined as $\delta_{jt} = \delta^-_{jt} + \delta^+_{jt}$, and the volume objective is given by $v(\delta_{jt})=\log(\delta_{jt}+1)$. Percentages are relative to the total number of elements or the global maxima of $\sum_{j,t}\delta_{jt}$ and $\sum_{j,t}v(\delta_{jt})$.} 
    \label{tab:rampsafe}
\end{table}

Next, we compare Algorithm~\ref{alg:tailoredAPM} to an approach that fixes real power injection $p^{\mathrm{tot}}$ to the final bus-level solution ${\overline{p}^{\mathrm{tot}}}^{(\overline{\tau})}$ in the final subproblems $Z^{\mathrm{F}}_t$.  This is equivalent to taking $\delta = 0$ and demonstrates the impact of the flexibility provided by the ramping restriction.  Fixing power injection significantly degrades the solution quality, as shown in Figure~\ref{fig:ecdf_scores}.  With fixed power injection, the number of infeasible cases increases from 19 to 66.  Relative to fixed injection, the ramping restriction increases the objective for feasible cases by an average of 3.09\% with a maximum improvement of 2.6x.  The total objective improvement across all cases (including infeasible cases) is 12.68\%.

\subsubsection{Analysis of the Energy Limit Restriction}
Even with the energy limit penalty ignored in the final subproblems, none of the test cases incur any penalty at our reported solution.  To demonstrate the benefit of the restriction, we consider 9 cases on the 1,576-bus network, taken from the third event of GOC3.  These are the only cases from this event that have nonzero energy limit penalty in our reported solution when the penalty is ignored in the final subproblems. 
In these 9 cases, including the energy limit penalty restriction in the final objective reduces the total penalty $\sum_{j \in \mathcal{J}^{\mathrm{sd}}} C^{\mathrm{e}}_j(p)$ by an average of 98.99\% and improves the solution objective by an average of 8.68\%.  Table~\ref{table:eminmaxCaseSummary} shows the solution objectives and penalty contributions for these cases.

\begin{table}[tp]
    \centering
    \begin{tabular}{cc|r|rr|rr|}
        &&& \multicolumn{2}{c|}{without Restriction} & \multicolumn{2}{c|}{with Restriction}\\
        Division & Scenario & \multicolumn{1}{c|}{$Z^{\mathrm{UB}}$} & \multicolumn{1}{c}{Objective} & \multicolumn{1}{c|}{Penalty} & \multicolumn{1}{c}{Objective} & \multicolumn{1}{c|}{Penalty} \\
        \hline
        D2 & 027 & 572,227,184 & 498,453,819 & 58,480,001 & 564,196,225 & 0 \\
        D2 & 031 & 484,996,329 & 411,492,289 & 37,520,166 & 474,577,847 & 599,096 \\
        D2 & 032 & 650,963,270 & 588,681,039 & 37,680,001 & 640,863,341 & 375,629 \\
        D2 & 033 & 949,087,700 & 899,846,594 & 37,680,001 & 941,670,329 & 418,576 \\
        D2 & 036 & 652,367,566 & 607,441,569 & 37,680,001 & 648,867,311 & 360,291 \\
        D2 & 041 & 484,735,515 & 418,757,179 & 37,553,010 & 475,422,381 & 887,888 \\
        D2 & 042 & 650,664,039 & 587,119,183 & 37,680,001 & 636,952,776 & 382,855 \\
        D2 & 043 & 948,728,190 & 899,529,862 & 37,680,001 & 937,687,552 & 395,104 \\
        D3 & 027 & 1,920,855,345 & 1,849,812,212 & 20,800,000 & 1,906,039,779 & 0
    \end{tabular}
    \caption{Performance on 1,576-bus cases from event 3 with and without the energy restriction included in the final subproblem objective.  The SOC upper bound, feasible solution objective, and energy penalty are shown.}
    \label{table:eminmaxCaseSummary}
\end{table}

\subsubsection{Analysis of the Second-Order Cone Relaxation}
The main benefit of the SOC relaxation is to reduce the computational cost of each iteration.  We quantify this benefit by comparing the solve time of the final models $Z^{\mathrm{F}}_t$, which contain exact nonconvex ACOPF constraints, to that of the SOC relaxations $\overline{Z}^{\mathrm{T}}_t$.  Although the final models differ in other areas, including the ramping restriction and fixing of the unit commitment variables, the comparison indicates the relative scale of the runtimes.  Figure~\ref{fig:time_breakdown} shows that the time spent on 10 iterations of the SOC models is on the same order as the time spent on the final supbroblems.  Comparing the time to solve the first-iteration SOC models (which have penalty coefficient $\rho = 0$) against the final subproblems, we find that the SOC models solve on average 76.62\% faster than the AC models.  Across all cases, the total time spent on the first SOC iteration is 88.92\% less than time spent on the final subproblems.  This indicates that the SOC relaxation yields large reductions in solve times and that the reduction is more significant in cases with long solve times.  The speed-up allows more iterations to be run and permits large cases to be solved within the time limits.  

The quality of the SOC relaxation can be observed from the results in Section~\ref{sec:results_quality}, which demonstrate the existence of AC-feasible solutions with small optimality gaps relative to the dual bounds.  As these bounds are derived from the SOC relaxation, small gaps suggest that this relaxation does not significantly reduce accuracy relative to the AC model.

\subsubsection{Analysis of the Contingency Screening Heuristic}

We compare the final solution from Algorithm~\ref{alg:tailoredAPM} to the solution $(p^{\mathrm{F}},q^{\mathrm{F}},u^{\mathrm{F}})$ generated by the contingency-free final subproblems in step 11 of the algorithm.  Figure~\ref{fig:ecdf_scores} plots an empirical cumulative distribution function for the optimality gap of these contingency-free solutions.  Per Table~\ref{table:iterationAndContingencyCount}, 39\% of reported solutions for the full algorithm already are from the contingency-free solve.  For the remaining cases, Figure~\ref{fig:ecdf_scores} shows improvement in the gap when contingencies are considered.  In cases where contingency screening is completed, the post-contingency solutions improve the total objective by an average of 0.45\% and reduce the value of the contingency penalty $\sum_{t \in \mathcal{T}} C^{\mathrm{ctg}}_t(p,q)$ by an average of 24.45\%. Across all cases, the post-contingency solutions reduce the total contingency penalty by 71.47\%.  These results demonstrate that contingency screening eliminates most high-impact contingency costs.

\subsection{Differences from GO Competition Evaluation}

The GOC3 test cases allow for some AC branches to be dynamically switched off.  In this work, we describe our algorithm and evaluate its performance in a switching-free environment, where all lines remain on.  
We convert all cases to a switching-free version, and duplicate cases (included in the dataset with and without switching) are consolidated.

GOC3 evaluated software from 14 teams.  For reproducibility, we compare only against the open-source benchmark solver \citep{parker2024benchmark}.  \cite{elbert2024godata} provide the results of an independent evaluation of submissions from every team, including the solver described in this work.  These results include additional experiments on a 6,708-bus network generated from industry data that was not released publicly and a 23,643-bus network withheld from this work due to binding memory constraints.

\section{Conclusion}
\label{sec:conclusion}

In this paper, we demonstrate that high-quality solutions to realistic industry-scale SCUC-ACOPF problems can be found within stringent time limits.  We present a decomposition scheme that separates the SCUC-ACOPF problem into a set of device-level mixed-integer linear programs and a set of temporally-separate nonlinear programs.  We apply a penalty alternating direction method to this decomposition and prove its convergence.  We introduce a number of heuristics to improve the solution quality and computational cost of the algorithm, including restrictions of time-coupling constraints, a second-order cone relaxation, and a contingency screening algorithm.  The algorithm tailored with these heuristics demonstrates superior performance on large-scale test cases and satisfies strict time limits, identifying near-optimal solutions relative to a dual bound.

\section*{Acknowledgement}
We acknowledge funding support from the ARPA-E GO Competition Challenge 3. We also thank the MIT SuperCloud and Lincoln Laboratory Supercomputing Center for providing access to high performance computing resources as well as consultation and support through the development process.  Finally, we thank the GOC3 operations and administration teams at Pacific Northwest National Laboratory and the U.S. Department of Energy for their efforts and accessibility over the course of the competition.

\bibliographystyle{informs2014}
\bibliography{AlternatingMethods_SCUCACOPF.bib}

\newpage
\appendix
\section{Appendix}

This appendix is organized as follows.  Section~\ref{ecsec:UCSet} provides a formulation for the feasible region of device-level unit commitment decisions.  Section~\ref{ecsec:PISet} describes the feasible region for device-level real power and reserve decisions.  Section~\ref{ecsec:gsf} defines generalized shift factors and describes the scheme for efficiently computing them across contingencies.  Section~\ref{ecsec:stylizednotation} introduces compact notation for the temporally-decomposed subproblems, which is used in analysis.  Section~\ref{ecsec:technicallemma} gives a pair of relevant technical lemmas.  Section~\ref{ecsec:Proofs} provides proofs of our results.

\subsection{Device-Level Unit Commitment Constraints}
\label{ecsec:UCSet}

In this section, we describe the constraints which define the sets $\mathcal{X}^{\mathrm{u}}_j$.  These constraints are as follows:
\begin{subequations}
    \label{constr:UCExtraSet}
    \begin{align}
        & u^{\mathrm{on}}_{jt} = 1 \quad & \forall j \in \mathcal{J}^{\mathrm{sd}},\ t \in \mathcal{T}^{\mathrm{mr}}_j, \label{constr:UCMustRun}\\
        & u^{\mathrm{on}}_{jt} = 0 \quad & \forall j \in \mathcal{J}^{\mathrm{sd}},\ t \in \mathcal{T}^{\mathrm{mo}}_j, \label{constr:UCMustOutage}\\
        & u^{\mathrm{su}}_{jt} \leq 1 - \sum_{t' \in \mathcal{T}^{\mathrm{dn,min}}_{jt}} u^{\mathrm{sd}}_{jt'} \quad & \forall j \in \mathcal{J}^{\mathrm{sd}},\ t \in \mathcal{T}, \label{constr:UCMinDowntime}\\
        & u^{\mathrm{sd}}_{jt} \leq 1 - \sum_{t' \in \mathcal{T}^{\mathrm{up,min}}_{jt}} u^{\mathrm{su}}_{jt'} \quad & \forall j \in \mathcal{J}^{\mathrm{sd}},\ t \in \mathcal{T}, \label{constr:UCMinUptime}\\
        & \sum_{t \in W} u^{\mathrm{su}}_{jt} \leq u^{\mathrm{su,max}}_{j}(W) \quad & \forall j \in \mathcal{J}^{\mathrm{sd}},\ W \in \mathcal{W}^{\mathrm{su,max}}_j. \label{constr:UCMaxStartup}
    \end{align}
\end{subequations}
Constraints \eqref{constr:UCMustRun}--\eqref{constr:UCMustOutage} enforce must-run and must-outage requirements, where $\mathcal{T}^{\mathrm{mr}}_j$ (resp.\ $\mathcal{T}^{\mathrm{mo}}_j$) gives the set of time periods during which device $j$ must be on (resp.\ off).  Constraints \eqref{constr:UCMinDowntime}--\eqref{constr:UCMinUptime} enforce minimum downtime and uptime requirements, and \eqref{constr:UCMaxStartup} limits the number of device startups over some time interval.  The sets $\mathcal{T}^{\mathrm{dn,min}}_{jt}$ (resp.\  $\mathcal{T}^{\mathrm{up,min}}_{jt}$) contain prior time periods $t' < t$ within the minimum downtime (resp.\ uptime) ranges relative to period $t$ for device $j$.  A set $W \in \mathcal{W}^{\mathrm{su,max}}_{j}$ is a set of time periods subject to a startup limit $(W \subseteq \mathcal{T})$, and $u^{\mathrm{su,max}}_j(W)$ gives the corresponding limit.  From these constraints, we define $\mathcal{X}^{\mathrm{u}}_j := \{\{u^{\mathrm{on}}_{jt},u^{\mathrm{su}}_{jt},u^{\mathrm{sd}}_{jt}\}_{t \in \mathcal{T}} \,:\, \eqref{constr:UCExtraSet}\}$ for all $j \in \mathcal{J}^{\mathrm{sd}}$, relying on the separability of the constraints over devices.

\subsection{Device-Level Real Power and Reserve Constraints}
\label{ecsec:PISet}

Here, we detail the device-level real power constraints of the sets $\mathcal{X}^{\mathrm{p}}_{jt}$.  We consider three reserve categories: up-reserve products for online devices, down-reserve products for online devices, and up-reserve products for offline devices.  We let $\mathcal{R}_{\mathrm{on}\uparrow}$, $\mathcal{R}_{\mathrm{on}\downarrow}$, and $\mathcal{R}_{\mathrm{off}\uparrow}$ be the subsets of $\mathcal{R}$ containing products assigned to the three categories.  These reserve subsets are totally ordered on the quality of the product, e.g., \verb|Regulation| $>$ \verb|Synchronized| $>$ \verb|Ramping| for elements of $\mathcal{R}_{\mathrm{on}\uparrow}$.  For a complete description of modeled reserve products, see \citeappendix{holzer2023gridec}.  Real power injection and reserve products are constrained by the following logic:
\begin{subequations}
    \label{constr:PIRealSet}
    \begin{align}
        & \sum_{\substack{r' \in \mathcal{R}_a\\r' \geq r}} p^{\mathrm{res}}_{jtr'} \leq p^{\mathrm{res,max}}_{jr} u^{\mathrm{on}}_{jt} \quad & \forall a \in \{\mathrm{on}\uparrow,\mathrm{on}\downarrow\},\ j \in \mathcal{J}^{\mathrm{sd}},\ t \in \mathcal{T},\ r \in \mathcal{R}_a, \label{constr:PIAbsReserveOn}\\
        & \sum_{\substack{r' \in \mathcal{R}_{\mathrm{off}\uparrow}\\r' \geq r}} p^{\mathrm{res}}_{jtr'} \leq p^{\mathrm{res,max}}_{jr} (1-u^{\mathrm{on}}_{jt}) \quad & \forall j \in \mathcal{J}^{\mathrm{sd}},\ t \in \mathcal{T},\ r \in \mathcal{R}_{\mathrm{off}\uparrow}, \label{constr:PIAbsReserveOff}\\
        & p^{\mathrm{on}}_{jt} \leq p^{\mathrm{max}}_{jt} u^{\mathrm{on}}_{jt} - \begin{cases}
            {\displaystyle \sum_{r \in \mathcal{R}_{\mathrm{on}\uparrow}}} p^{\mathrm{res}}_{jtr} & \text{if } j \in \mathcal{J}^{\mathrm{pr}}\\
            {\displaystyle \sum_{r \in \mathcal{R}_{\mathrm{on}\downarrow}}} p^{\mathrm{res}}_{jtr} & \text{if } j \in \mathcal{J}^{\mathrm{cs}}\\
        \end{cases} \quad & \forall j \in \mathcal{J}^{\mathrm{sd}},\ t \in \mathcal{T}, \label{constr:PIRelReserveOnMax}\\
        & p^{\mathrm{on}}_{jt} \geq p^{\mathrm{min}}_{jt} u^{\mathrm{on}}_{jt} + \begin{cases}
            {\displaystyle \sum_{r \in \mathcal{R}_{\mathrm{on}\downarrow}} p^{\mathrm{res}}_{jtr}} & \text{if } j \in \mathcal{J}^{\mathrm{pr}}\\
            {\displaystyle \sum_{r \in \mathcal{R}_{\mathrm{on}\uparrow}} p^{\mathrm{res}}_{jtr}} & \text{if } j \in \mathcal{J}^{\mathrm{cs}}
        \end{cases} \quad & \forall j \in \mathcal{J}^{\mathrm{sd}},\ t \in \mathcal{T}, \label{constr:PIRelReserveOnMin}\\
        & p^{\mathrm{tot}}_{jt} \leq p^{\mathrm{max}}_{jt} (1-u^{\mathrm{on}}_{jt}) - \sum_{r \in \mathcal{R}_{\mathrm{off}\uparrow}} p^{\mathrm{res}}_{jtr} \quad & \forall j \in \mathcal{J}^{\mathrm{sd}},\ t \in \mathcal{T}. \label{constr:PIRelReserveOffMax}
    \end{align}
\end{subequations}
Constraints \eqref{constr:PIAbsReserveOn}--\eqref{constr:PIAbsReserveOff} enforce absolute limits on reserve products, where $p^{\mathrm{res,max}}_{jr}$ gives the maximum quantity of reserves at least as good as product $r$.  Whether a device can provide a reserve product depends on if the device has the appropriate commitment status for the reserve category.  Constraints \eqref{constr:PIRelReserveOnMax}--\eqref{constr:PIRelReserveOffMax} enforce reserve limits relative to dispatched real power.  The parameters $p^{\mathrm{max}}_{jt}$ and $p^{\mathrm{min}}_{jt}$ give maximum and minimum total real power.  These limits count reserves that cut into headroom to the minimum or maximum; that is, reserve products that could increase real power injection count against maximum power, and those that decrease injection against the minimum.  Whether a product increases or decreases power injection depends on whether the device produces or consumes power.  We define $\mathcal{X}^{\mathrm{p}}_{jt} := \{(u^{\mathrm{on}}_{jt},p^{\mathrm{tot}}_{jt},p^{\mathrm{on}}_{jt},\{p^{\mathrm{res}}_{jtr}\}_{r \in \mathcal{R}})\, :\, \eqref{constr:PIRealSet}\}$ for all $j \in \mathcal{J}^{\mathrm{sd}}$ and $t \in \mathcal{T}$.  This construction relies on the separability of the constraints over devices and time.

\subsection{Generalized Shift Factors and Fast Contingency Evaluation}
\label{ecsec:gsf}

This section defines the generalized shift factors $F_{kji}$ for any connected subnetwork $\mathcal{J} \subseteq \mathcal{J}^{\mathrm{ac}}$.

Let $M \in \mathbb{R}^{|\mathcal{I}| \times |\mathcal{J}^{\mathrm{ac}}|}$ be the bus-branch incidence matrix, where $$M_{ij} = \begin{cases}
    1 & \text{if } i_j = i\\
    -1 & \text{if } i'_j = i\\
    0 & \text{otherwise.}
\end{cases}$$
Let $B(\mathcal{J}) \in \mathbb{R}^{|\mathcal{J}^{\mathrm{ac}}| \times |\mathcal{J}^{\mathrm{ac}}|}$ be the diagonal branch susceptance matrix, where $$B(\mathcal{J})_{jj} = \begin{cases}
    -b_j \quad &\text{if } j \in \mathcal{J}\\
    0 \quad & \text{otherwise}.
\end{cases}$$ Under the DC approximation, the real power flows $p^{\mathrm{fl}} \in \mathbb{R}^{|\mathcal{J}^{\mathrm{ac}}|}$ over AC branches  are computed as a linear function of the bus phase angles $\theta \in \mathbb{R}^{|\mathcal{I}|}$:
$$p^{\mathrm{fl}} = B(\mathcal{J})M\tp \theta.$$
For simplicity, we ignore the dependence on phase shifts $\phi$.  Then, given a vector of net nodal real power injections $p^{\mathrm{inj}} \in \mathbb{R}^{|\mathcal{I}|}$, nodal power balance requires 
$$p^{\mathrm{inj}} = Mp^{\mathrm{fl}} = M B(\mathcal{J}) M \tp \theta.$$

The matrix $MB(\mathcal{J})M\tp$ is the bus admittance matrix. For some reference bus $i_0 \in \mathcal{I}$, we construct
$$H(\mathcal{J}) _{ii'} = \begin{cases}
    (MB(\mathcal{J})M\tp)_{ii'} & \text{if } i_0 \not \in \{i,i'\}\\
    1 & \text{if } i = i' = i_0\\
    0 & \text{otherwise.}
\end{cases}$$
If the branches $\mathcal{J}$ form a connected network, $H(\mathcal{J})$ is invertible.  Feasible bus angles can then be computed by inverting $H(\mathcal{J})$ and setting $\theta_{i_0} = 0$; that is, 
$$\theta = SH(\mathcal{J})^{-1} p^{\mathrm{inj}},$$
where $S = I - \mathbf{e}_{i_0} \mathbf{e}_{i_0}\tp$ sets the $i_0$-th entry to $0$.  Here, $I$ denotes the identity matrix and $\mathbf{e}_{i_0}$ is the $i_0$-th standard basis vector.  Now, the real power flows are given by $$p^{\mathrm{fl}} = B(\mathcal{J})M \tp SH(\mathcal{J})^{-1} p^{\mathrm{inj}}.$$
The coefficient $(B(\mathcal{J})M \tp SH(\mathcal{J})^{-1})_{ji}$ gives the GSF from bus $i$ to line $j$ under network $\mathcal{J}$.  Therefore, the GSF from bus $i$ to line $j$ under contingency $k$ is defined by $$F_{kji} := (B(\mathcal{J}^{\mathrm{ac}} \setminus \{j_k\}) M \tp SH(\mathcal{J}^{\mathrm{ac}} \setminus \{j_k\})^{-1})_{ji}.$$

Computing factors $F_{kji}$ requires the inverse of matrix $H(\mathcal{J}^{\mathrm{ac}} \setminus \{j_k\})$, which may be difficult to compute for all contingencies $k$.
Instead, we express this matrix as a rank-1 update to the matrix $H(\mathcal{J}^{\mathrm{ac}})$.  If $i_0 \not \in \{i_{j_k},i'_{j_k}\}$, then
$$H(\mathcal{J}^{\mathrm{ac}} \setminus \{j_k\}) = H(\mathcal{J}^{\mathrm{ac}}) + b_{j_k} (\mathbf{e}_{i_{j_k}} - \mathbf{e}_{i'_{j_k}}) (\mathbf{e}_{i_{j_k}} - \mathbf{e}_{i'_{j_k}}) \tp.$$
 If $i_0 \in \{i_{j_k},i'_{j_k}\}$, a similar rank-1 update can be derived.  We leverage this structure by first evaluating $H(\mathcal{J}^{\mathrm{ac}})^{-1}$, then computing $H(\mathcal{J}^{\mathrm{ac}} \setminus \{j_k\})^{-1}$ for relevant $k$ by the Sherman-Morrison formula \citepappendix{hager1989updatingec}. This method for the efficient computation of GSFs across multiple contingencies is similar to that of \citeappendix{holzer2023fastec}.

\subsection{Stylized Notation for Temporally-Decomposed Subproblems}
\label{ecsec:stylizednotation}

To aid in the analysis of Algorithm~\ref{alg:basicAPM}, we introduce a generic form of the models $Z^{\mathrm{T}}_t(\overline{p}_t,\overline{q}_t,\overline{u}_t;\rho)$:
\begin{equation}
    \label{eq:genericSubproblem}
    \begin{aligned}
        \max_{x,y,z} \quad & F(x,y,z) - \rho \Gamma (x,\overline{x})\\
        \text{s.t.} \quad & x \in \mathcal{X},\\
        & y = A_{z}(x),\\
        & z \in \mathcal{Z}.
    \end{aligned}
\end{equation}
The objects that define this model have the following properties:
\begin{enumerate}
    \item The feasible set of \eqref{eq:genericSubproblem} is bounded with diameter $D$;
    \item $\mathcal{X}$ is polyhedral;
    \item $A_z$ is a linear function for all $z \in \mathcal{Z}$;
    \item $F(x,y,z)$ is Lipschitz continuous with parameter $L$ over the feasible set of \eqref{eq:genericSubproblem}, so that all $\nabla F \in \partial F(x,A_z(x),z)$ with $x \in \mathcal{X}$ and $z \in \mathcal{Z}$ satisfy $\norm{\nabla F}_2 \leq L$;
    \item Penalty function $\Gamma(x,\overline{x}) = \norm{E(x - \overline{x})}_2^2$ where $E \tp E = E$; and
    \item If $(\overline{p}_j,\overline{q}_j,\overline{u}_j) \in \mathcal{X}^{\mathrm{uc}}_j$ for all $j \in \mathcal{J}^{\mathrm{sd}}$ and $\overline{u}^{\mathrm{sh}}_{jt} \in \mathcal{X}^{\mathrm{sh}}_{jt}$ for all $j \in \mathcal{J}^{\mathrm{sh}}$ and $t \in \mathcal{T}$, then $\overline{x} \in \mathcal{X}$.
\end{enumerate}

We demonstrate that, for any $t \in \mathcal{T}$ and solution $(\overline{p},\overline{q},\overline{u})$, the model $Z^{\mathrm{T}}_t(\overline{p}_t,\overline{q}_t,\overline{u}_t;\rho)$ is of the form \eqref{eq:genericSubproblem}.  
Let $$\mathcal{Z} = \{(p^{\mathrm{fr}}_t,p^{\mathrm{to}}_t,q^{\mathrm{fr}}_t,q^{\mathrm{to}}_t,v_t) \,:\, \exists(\Delta_t,\theta_t,\tau_t,\phi_t) \text{ s.t. } \eqref{constr:ACAngleDiff}-\eqref{constr:DCBalance},\,\eqref{constr:BusSets}-\eqref{constr:DCSets}\}$$ and 
\begin{equation*}
    \renewcommand\arraystretch{1.5}
    \mathcal{X} = \left \{
        (p_t,q_t,u_t) \setminus (p^{\mathrm{fr}}_t,p^{\mathrm{to}}_t,p^{\mathrm{sh}}_t,q^{\mathrm{fr}}_t,q^{\mathrm{to}}_t,q^{\mathrm{sh}}_t)\,:\ \begin{array}{lr}
        (p_{jt},q_{jt},u^{\mathrm{on}}_{jt}) \in \mathcal{Y}^{\mathrm{uc}}_{jt} & \quad  \forall j \in \mathcal{J}^{\mathrm{sd}},\\
        u^{\mathrm{sh}}_{jt} \in \mathcal{Y}^{\mathrm{sh}}_{jt} & \quad \forall j \in \mathcal{J}^{\mathrm{sh}}
    \end{array}
    \right \},
\end{equation*}
where we abuse the notation $(p_t,q_t,u_t) \setminus (p^{\mathrm{fr}}_t,p^{\mathrm{to}}_t,p^{\mathrm{sh}}_t,q^{\mathrm{fr}}_t,q^{\mathrm{to}}_t,q^{\mathrm{sh}}_t)$ to represent the projection of the solution $(p_t,q_t,u_t)$ onto all coordinates except the branch and shunt power flow variables.  The variables $x$ lie in $\mathcal{X}$ and the variables $z$ lie in $\mathcal{Z}$.  Let 
\begin{equation*}
    \renewcommand\arraystretch{1.5}
    A_z(x) = \begin{bmatrix}
        g^{\mathrm{sh}}_j v_{i_j t}^2 u^{\mathrm{sh}}_{jt} & \multirow{2}{*}{$\quad \forall j \in \mathcal{J}^{\mathrm{sh}}$}\\
        -b^{\mathrm{sh}}_j v_{i_j t}^2 u^{\mathrm{sh}}_{jt} & 
    \end{bmatrix} = \begin{bmatrix}
        p^{\mathrm{sh}}_t\\q^{\mathrm{sh}}_t
    \end{bmatrix}.
\end{equation*}
For fixed $z$, the function $A_z$ is linear in $x$ as the variables $v_{i_j t}$ are contained in $z$ and the variables $u^{\mathrm{sh}}_{jt}$ are contained in $x$.  
Any $x \in \mathcal{X}$ contains all the elements of $(p_t,q_t,u_t)$ except for branch and shunt power flows, while $y = A_z(x)$ contains the shunt power flow variables $(p^{\mathrm{sh}}_t,q^{\mathrm{sh}}_t)$.  Any $(x,y,z)$ with $y = A_z(x)$ satisfies constraints \eqref{constr:ShuntReal}--\eqref{constr:ShuntReactive}.

Under these constructions, the feasible set of \eqref{eq:genericSubproblem} corresponds to the feasible set of $Z^{\mathrm{T}}_t$, which is bounded under Assumption~\ref{assump:DataAssumptions}.  We call the diameter of this set $D$.  Under Assumption~\ref{assump:DataAssumptions}, $\mathcal{X}$ is polyhedral because the sets $\mathcal{Y}^{\mathrm{uc}}_{jt}$ and $\mathcal{Y}^{\mathrm{sh}}_{jt}$ are polyhedral.

With the sets $(\mathcal{X},\mathcal{Z})$ and functions $A_z$ constructed in this way, any feasible $(x,y,z)$ contains exactly one copy of each variable in $(p_t,q_t,u_t)$.  We define $F(x,y,z) = R^{\mathrm{T}}_t(p_t,q_t)$, which is Lipschitz continuous over any bounded domain.  To ensure that all subgradients on the domain of \eqref{eq:genericSubproblem} are properly bounded, we define the constant $L$ to be the Lipschitz constant of $F$ over some bounded set that strictly contains this domain.   

We construct $$\overline{x} = (\overline{p}_t,\overline{q}_t,\overline{u}_t) \setminus (\overline{p}^{\mathrm{fr}}_t,\overline{p}^{\mathrm{to}}_t,\overline{p}^{\mathrm{sh}}_t,\overline{q}^{\mathrm{fr}}_t,\overline{q}^{\mathrm{to}}_t,\overline{q}^{\mathrm{sh}}_t),$$ and accordingly define $$\Gamma(x,\overline{x}) = \norm{E(x - \overline{x})}^2_2 = \Gamma_2(p_t,q_t,u_t,\overline{p}_t,\overline{q}_t,\overline{u}_t),$$ 
where $E$ is a diagonal matrix with an entry of $1$ if the corresponding element of $x$ is a copied variable (i.e., penalized in $\Gamma_2$) and an entry of $0$ otherwise.  This matrix clearly satisfies $E \tp E = E$.  In the generic form, $x \in \mathcal{X}$ contains every copied variable, so we can represent $\Gamma_2$ as a function of $Ex$ and $E \overline{x}$.

Finally, suppose that $(\overline{p}_j,\overline{q}_j,\overline{u}_j) \in \mathcal{X}^{\mathrm{uc}}_j$ for all $j \in \mathcal{J}^{\mathrm{sd}}$ and $\overline{u}^{\mathrm{sh}}_{jt} \in \mathcal{X}^{\mathrm{sh}}_{jt}$ for all $j \in \mathcal{J}^{\mathrm{sh}}$ and $t \in \mathcal{T}$.  By Lemma~\ref{lemma:setRelaxations}, it directly follows that $\overline{x} \in \mathcal{X}$.
This demonstrates that $Z^{\mathrm{T}}_t(\overline{p}_t,\overline{q}_t,\overline{u}_t;\rho)$ satisfies the structure of problem \eqref{eq:genericSubproblem}.

\subsection{Technical Lemmas}
\label{ecsec:technicallemma}

This section presents two technical lemmas that are used to prove results from the body of this paper.  Lemma~\ref{lemma:basicAPMDevicePenalty} bounds the magnitude of the penalty term for the device-level block subproblems by the magnitude of the penalty term from the previous iteration solution to the temporally-decomposed block subproblems.  Lemma~\ref{lemma:shiftedSequence} establishes a fundamental property on translating the limit point of a convergent sequence contained in a polyhedron.

\begin{lemma}
    \label{lemma:basicAPMDevicePenalty}
    Let $(p^{(\tau)},q^{(\tau)},u^{(\tau)},\overline{p}^{(\tau)},\overline{q}^{(\tau)},\overline{u}^{(\tau)})$ be a sequence of iterates generated by Algorithm~\ref{alg:basicAPM}, where the subproblems $Z^{\mathrm{J}}_j$ are solved to global optimality.  Let $R$ be the same constant as in the statement of Theorem~\ref{thm:basicAPMInfeasibility}.  Then,
    $$\Gamma_1(p^{(\tau)},q^{(\tau)},u^{(\tau)},\overline{p}^{(\tau)},\overline{q}^{(\tau)},\overline{u}^{(\tau)}) \leq \frac{R}{\rho_{\tau}} + \Gamma_1 (p^{(\tau)},q^{(\tau)},u^{(\tau)},\overline{p}^{(\tau-1)},\overline{q}^{(\tau-1)},\overline{u}^{(\tau-1)}) \quad \forall \tau \in \lBrack \tau \rBrack \setminus \{1\}.$$
\end{lemma}

\begin{proof}
    Take any $j \in \mathcal{J}^{\mathrm{sd}}$ and iteration $\tau > 1$.  Then, $(\overline{p}^{(\tau)}_j,\overline{q}^{(\tau)}_j,\overline{u}^{(\tau)}_j)$ is globally optimal for
    \linebreak
    $Z^{\mathrm{J}}_j(p^{(\tau)}_j,q^{(\tau)}_j,u^{(\tau)}_j;\rho_\tau)$.  Under Assumption~\ref{assump:DataAssumptions}, the variables which appear in the objective function $R^{\mathrm{J}}_j$ have bounded domain, so $\underline{R}_j \leq R^{\mathrm{J}}_j(p,u) \leq \overline{R}_j$ for some bounds $(\underline{R}_j,\overline{R}_j)$ and the constants $R_j = \overline{R}_j - \underline{R}_j$ used to define $R$ are well defined.  

    As the feasible region of $Z^{\mathrm{J}}_j$ does not change between iterations, $(\overline{p}^{(\tau-1)}_j,\overline{q}^{(\tau-1)}_j,\overline{u}^{(\tau-1)}_j)$ remains feasible.  Applying global optimality for the maximization, we have for all $j \in \mathcal{J}^{\mathrm{sd}}$ that
    \begin{equation*}
        \begin{aligned}
            & R^{\mathrm{J}}_j(\overline{p}^{(\tau-1)}_j,\overline{u}^{(\tau-1)}_j) - \rho_{\tau} \Gamma_1'(p^{(\tau)}_j,q^{(\tau)}_j,u^{(\tau)}_j,\overline{p}^{(\tau-1)}_j,\overline{q}^{(\tau-1)}_j,\overline{u}^{(\tau-1)}_j)\\
            \leq \ & R^{\mathrm{J}}_j(\overline{p}^{(\tau)}_j,\overline{u}^{(\tau)}_j) - \rho_{\tau} \Gamma_1'(p^{(\tau)}_j,q^{(\tau)}_j,u^{(\tau)}_j,\overline{p}^{(\tau)}_j,\overline{q}^{(\tau)}_j,\overline{u}^{(\tau)}_j).
        \end{aligned}
    \end{equation*}
    Additionally, for $j \in \mathcal{J}^{\mathrm{sh}}$, we have that 
    $$\rho_{\tau} \norm{u^{\mathrm{sh} (\tau)}_j - \overline{u}^{\mathrm{sh}(\tau-1)}_j}_1 \geq \rho_{\tau} \norm{u^{\mathrm{sh}(\tau)}_j - \overline{u}^{\mathrm{sh}(\tau)}_j}_1$$
    by definition of subproblems $Z^{\mathrm{SH}}_j$.
    These properties yield the relation
    \begin{equation*}
        \begin{aligned}
            & \rho_{\tau} \Gamma_1 (p^{(\tau)},q^{(\tau)},u^{(\tau)},\overline{p}^{(\tau)},\overline{q}^{(\tau)},\overline{u}^{(\tau)})\\
            = \ & \rho_{\tau} \left ( \sum_{j \in \mathcal{J}^{\mathrm{sd}}} \Gamma_1'(p^{(\tau)}_j,q^{(\tau)}_j,u^{(\tau)}_j,\overline{p}^{(\tau)}_j,\overline{q}^{(\tau)}_j,\overline{u}^{(\tau)}_j) + \sum_{j \in \mathcal{J}^{\mathrm{sh}}} \norm{{u^{\mathrm{sh}(\tau)}_j} - {\overline{u}^{\mathrm{sh}(\tau)}_j}}_1 \right)\\
            \leq \ & \sum_{j \in \mathcal{J}^{\mathrm{sd}}} \left ( R^{\mathrm{J}}_j(\overline{p}^{(\tau)}_j,\overline{u}^{(\tau)}_j) - R^{\mathrm{J}}_j(\overline{p}^{(\tau-1)}_j,\overline{u}^{(\tau-1)}_j) + \rho_{\tau} \Gamma_1'(p^{(\tau)}_j,q^{(\tau)}_j,u^{(\tau)}_j,\overline{p}^{(\tau-1)}_j,\overline{q}^{(\tau-1)}_j,\overline{u}^{(\tau-1)}_j) \right )\\
            & \qquad + \rho_{\tau} \sum_{j \in \mathcal{J}^{\mathrm{sh}}} \norm{{u^{\mathrm{sh}(\tau)}_j} - {\overline{u}^{\mathrm{sh}(\tau-1)}_j}}_1\\
            \leq \ & \sum_{j \in \mathcal{J}^{\mathrm{sd}}} R_j + \rho_{\tau} \Gamma_1 (p^{(\tau)},q^{(\tau)},u^{(\tau)},\overline{p}^{(\tau-1)},\overline{q}^{(\tau-1)},\overline{u}^{(\tau-1)})
        \end{aligned}
    \end{equation*}
    where the first line follows from the definitions of $\Gamma_1$ and $\Gamma'_1$, the second from global optimality for $Z^{\mathrm{J}}_j$ and $Z^{\mathrm{SH}}_j$, and the third from the objective function bounds and again the definitions of $\Gamma_1$ and $\Gamma'_1$.  As $R = \sum_{j \in \mathcal{J}^{\mathrm{sd}}} R_j$, this proves the result.
\end{proof}

\begin{lemma}
    \label{lemma:shiftedSequence}
    Let $\mathcal{X}$ be a polyhedron  with $\{x^*,\tilde{x}\} \in \mathcal{X}$.  Consider a sequence $\{x_i\}_{i=1}^\infty \subset \mathcal{X}$ such that $\lim_{i \rightarrow \infty} x_i = x^*$.  Then, there exists some $\{\lambda_i\}_{i=1}^\infty \subset \mathbb{R}_\geq$ such that the sequence $\tilde{x}_i := x_i + \lambda_i (\tilde{x} - x^*)$ satisfies $\{\tilde{x}_i\}_{i=1}^\infty \subset \mathcal{X}$ and $\lim_{i \rightarrow \infty} \tilde{x}_i = \tilde{x}$.
\end{lemma}

\begin{proof}
    Define $(A,b)$ with $A \in \mathbb{R}^{m \times n}$ and $b \in \mathbb{R}^m$ such that $\mathcal{X} = \{x \,:\, Ax \leq b\}$.  Let $M := \{j \in \lBrack m \rBrack \,:\, a_j \tp (\tilde{x} - x^*) > 0\}$, where $a_j$ is the $j$-th row of matrix $A$.  We define a sequence of multipliers $\lambda_i$ by 
    $$\lambda_i := \begin{cases}
        \min\ \left \{\underset{j \in M}{\min}\  \frac{b_j - a_j \tp x_i}{a_j \tp (\tilde{x} - x^*)},\, 1 \right \} & \quad \text{ if }M \neq \emptyset\\
        1 & \quad \text{ if } M = \emptyset.
    \end{cases}$$
    These multipliers yield the sequence of iterates $\tilde{x}_i := x_i + \lambda_i (\tilde{x} - x^*)$.  Note that $a_j \tp x_i \leq b_j$ as $x_i \in \mathcal{X}$ and $a_j \tp (\tilde{x} - x^*) > 0$ for $j \in M$, so $\lambda_i \geq 0$.

    First, we establish that the sequence $\tilde{x}_i$ lies in $\mathcal{X}$.  Consider some index $i$.  For $j \in \lBrack m \rBrack \setminus M$,
    $$a_j \tp \tilde{x}_i = a_j \tp x_i + \lambda_i a_j \tp (\tilde{x} - x^*) \leq a_j \tp x_i \leq b,$$
    as $\lambda_i \geq 0$ and $a_j \tp (\tilde{x} - x^*) \leq 0$.  For $j \in M$, we have that $\lambda_i \leq \frac{b_j - a_j \tp x_i}{a_j \tp (\tilde{x} - x^*)}$ and $a_j \tp (\tilde{x} - x^*) > 0$, so
    \begin{equation*}
        \begin{aligned}
            a_j \tp \tilde{x}_i =\ & a_j \tp x_i + \lambda_i a_j \tp (\tilde{x} - x^*)\\
            \leq\ & a_j \tp x_i  + \left ( \frac{b_j - a_j \tp x_i}{a_j \tp (\tilde{x} - x^*)} \right ) a_j \tp (\tilde{x} - x^*)\\
            =\ & b_j.
        \end{aligned}
    \end{equation*}
    These results demonstrate that $A \tilde{x}_i \leq b$ for all $i$, and thus $\{\tilde{x}_i\}_{i=1}^\infty \subset \mathcal{X}$.

    Next, we consider the limit of the iterates $\lambda_i$.  If $M = \emptyset$, clearly $\lim_{i \rightarrow \infty} \lambda_i = 1$.  Otherwise, 
    \begin{equation*}
        \begin{aligned}
            \lim_{i \rightarrow \infty} \lambda_i =\ & \lim_{i \rightarrow \infty}\ \min\ \left \{\underset{j \in M}{\min}\  \frac{b_j - a_j \tp x_i}{a_j \tp (\tilde{x} - x^*)},\, 1 \right \}\\
            =\ & \min\  \left \{\underset{j \in M}{\min}\  \lim_{i \rightarrow \infty}\ \frac{b_j - a_j \tp x_i}{a_j \tp (\tilde{x} - x^*)}, 1 \right \}\\
            =\ & \min\  \left \{\underset{j \in M}{\min}\  \frac{b_j - a_j \tp x^*}{a_j \tp (\tilde{x} - x^*)}, 1 \right \}\\
            \geq\ & \min\ \left \{ \underset{j \in M}{\min}\  \frac{b_j - a_j \tp x^*}{b_j - a_j \tp x^*}, 1 \right \}\\
            =\ & 1,
        \end{aligned}
    \end{equation*}
    where the second line follows from continuity of the $\min$ function, the third line from the convergence of $\{x_i\}_{i=1}^\infty$ to $x^*$, and the fourth from the feasibility of $\tilde{x}$ and $x^*$ for $\mathcal{X}$.  
    As $\lambda_i \leq 1$, we have shown that $\lim_{i \rightarrow \infty} \lambda_i = 1$.  Therefore, $\lim_{i \rightarrow \infty} \tilde{x}_i = x^* + (\tilde{x} - x^*) = \tilde{x}$. 
\end{proof}

\subsection{Proofs of Results}
\label{ecsec:Proofs}

\subsubsection*{Proof of Lemma~\ref{lemma:setRelaxations}}

\begin{proof}
We first consider property~\ref{lemprop:UCRelaxInclusion}.  Consider some $j \in \mathcal{J}^{\mathrm{sd}}$ and $(p_j,q_j,u_j) \in \mathcal{X}^{\mathrm{uc}}_j$.  By Assumption~\ref{assump:DataAssumptions}, such a $u_j$ satisfies \eqref{eq:SUSDRelaxation}.  Further, $(u_{j},p_{j},q_{j})$ satisfy \eqref{constr:PIReal}--\eqref{constr:PIReactive} by definition of $\mathcal{X}^{\mathrm{uc}}_j$.  As $p^{\mathrm{supc}}_{jtt'} \geq 0$, we have by \eqref{constr:PISUDef} that $$p^{\mathrm{su}}_{jt} = \sum_{t' \in \mathcal{T}^{\mathrm{supc}}_{jt}} p^{\mathrm{supc}}_{jtt'} u^{\mathrm{su}}_{jt'} \geq 0,$$
and thus \eqref{constr:RelaxPISUDef} is satisfied.  Similarly, we see that \eqref{constr:RelaxPISDDef} holds.  Therefore, with certificate $(u^{\mathrm{su}}_j,u^{\mathrm{sd}}_j)$, it holds that $(p_{jt},q_{jt},u^{\mathrm{on}}_{jt}) \in \mathcal{Y}^{\mathrm{uc}}_{jt}$ for all $t \in \mathcal{T}$, proving property~\ref{lemprop:UCRelaxInclusion}.  Next, consider property~\ref{lemprop:ShuntRelaxInclusion}.  Clearly, as any $u^{\mathrm{sh}}_{jt} \in \mathcal{X}^{\mathrm{sh}}_{jt}$ satisfies \eqref{constr:ShuntSets}, $u^{\mathrm{sh}}_{jt} \in [u^{\mathrm{sh,min}}_j,u^{\mathrm{sh,max}}_j]$ and thus $u^{\mathrm{sh}}_{jt} \in \mathcal{Y}^{\mathrm{sh}}_{jt}$.
\end{proof}

\subsubsection*{Proof of Proposition~\ref{prop:EQequivalence}}

\begin{proof}
We will prove that feasible solutions for each model can be mapped to feasible solutions for the other model with the same objective value and the same values for the copied variables.

First, consider a feasible solution $(p,q,u,v,\Delta,\theta,\tau,\phi)$ for \eqref{SC-ACOPF}.  Define $(\overline{p},\overline{q},\overline{u}) = (p,q,u)$.  By this definition, the last constraint of \eqref{EQ} is satisfied for solution $(p,q,u,\overline{p},\overline{q},\overline{u})$ and the copied variables have the same values in both solutions.  As $(p,q,u,v,\Delta,\theta,\tau,\phi)$ satisfies \eqref{constr:PowerFlow}, it holds that $(p_t,q_t,u^{\mathrm{sh}}_t) \in \mathcal{X}^{\mathrm{ac}}_t$ for all $t \in \mathcal{T}$.  Additionally, for all $j \in \mathcal{J}^{\mathrm{sh}}$ and $t \in \mathcal{T}$, it holds that $u^{\mathrm{sh}}_{jt} = \overline{u}^{\mathrm{sh}}_{jt} \in \mathcal{X}^{\mathrm{sh}}_{jt} \subseteq \mathcal{Y}^{\mathrm{sh}}_{jt}$ by Lemma~\ref{lemma:setRelaxations}. Similarly, as $(p,q,u)$ satisfy \eqref{constr:UC}--\eqref{constr:Ramp}, for all $j \in \mathcal{J}^{\mathrm{sd}}$ it holds that $(p_j,q_j,u_j) = (\overline{p}_j,\overline{q}_j,\overline{u}_j) \in \mathcal{X}^{\mathrm{uc}}_j$ and further, by Lemma~\ref{lemma:setRelaxations}, $(p_{jt},q_{jt},u^{\mathrm{on}}_{jt}) \in \mathcal{Y}^{\mathrm{uc}}_{jt}$ for all $t \in \mathcal{T}$.  This demonstrates that $(p,q,u,\overline{p},\overline{q},\overline{u})$ is feasible for \eqref{EQ}.  Considering the objective, $\sum_{t \in \mathcal{T}} R^{\mathrm{T}}_t(p_t,q_t) + \sum_{j \in \mathcal{J}^{\mathrm{sd}}} R^{\mathrm{J}}_j(p_j,u_j) = \sum_{t \in \mathcal{T}} R^{\mathrm{T}}_t(p_t,q_t) + \sum_{j \in \mathcal{J}^{\mathrm{sd}}} R^{\mathrm{J}}_j(\overline{p}_j,\overline{u}_j)$, and the solutions have the same objective value in their respective models.

Next, consider a feasible solution $(p,q,u,\overline{p},\overline{q},\overline{u})$ for \eqref{EQ}.  First, $\overline{u}^{\mathrm{sh}}_{jt} \in \mathcal{X}^{\mathrm{sh}}_{jt}$ for all $j \in \mathcal{J}^{\mathrm{sh}}$ and $t \in \mathcal{T}$ and thus satisfies \eqref{constr:ShuntSets}.  By the copy constraints, $u^{\mathrm{sh}} = \overline{u}^{\mathrm{sh}}$.  Then, $(p_t,q_t,\overline{u}^{\mathrm{sh}}_t) = (p_t,q_t,u^{\mathrm{sh}}_t) \in \mathcal{X}^{\mathrm{ac}}_t$ for all $t \in \mathcal{T}$ and we can construct a corresponding solution $(v_t,\Delta_t,\theta_t,\tau_t,\phi_t)$ such that constraints \eqref{constr:ACAngleDiff}--\eqref{constr:DCSets} are satisfied by the solution $(p,q,\overline{u},v,\Delta,\theta,\tau,\phi)$.  In the remainder of the proof, we refer to this solution as the \textit{candidate solution}.

The candidate solution is constructed so that the copied variables have the same values as in \linebreak 
$(p,q,u,\overline{p},\overline{q},\overline{u})$.  As $(\overline{p}_j,\overline{q}_j,\overline{u}_j) \in \mathcal{X}^{\mathrm{uc}}_j$ for all $j \in \mathcal{J}^{\mathrm{sd}}$, these variables satisfy \eqref{constr:UC}--\eqref{constr:Ramp}.  As constraints \eqref{constr:UC} only contain variables $\overline{u}$, the candidate solution satisfies these constraints.  By the copy constraints, $p^{\mathrm{su}} = \overline{p}^{\mathrm{su}}$ and $p^{\mathrm{sd}} = \overline{p}^{\mathrm{sd}}$, so $(p,\overline{u})$ satisfies constraints \eqref{constr:PISUDef}--\eqref{constr:PISDDef}.  For the remaining constraints in~\eqref{constr:PIReal}, specifically \eqref{constr:PIRealDef}, \eqref{constr:PIDeviceReserveSet}, and \eqref{constr:PINonneg}, note that $(p_{jt},q_{jt},u^{\mathrm{on}}_{jt}) \in \mathcal{Y}^{\mathrm{uc}}_{jt}$, so there are some $(\tilde{u}^{\mathrm{su}}_j,\tilde{u}^{\mathrm{sd}}_j)$ that allow $(p,u^{\mathrm{on}})$ to be feasible for \eqref{constr:PIReal}.  These remaining constraints depend on $u$ only through $u^{\mathrm{on}}$ and do not contain $\tilde{u}^{\mathrm{su}}$ or $\tilde{u}^{\mathrm{sd}}$.  As $u^{\mathrm{on}} = \overline{u}^{\mathrm{on}}$ by the copy constraints, $(p,\overline{u})$ also satisfies \eqref{constr:PIReal}.  As the copy constraints enforce $(p^{\mathrm{tot}},q^{\mathrm{tot}},q^{\mathrm{res}}) = (\overline{p}^{\mathrm{tot}},\overline{q}^{\mathrm{tot}},\overline{q}^{\mathrm{res}})$, the solution $(p,q,\overline{u})$ satisfies constraints \eqref{constr:PIReactive} because constraints~\eqref{constr:PIReactive} depend on $(p,q)$ only through $(p^{\mathrm{tot}},q^{\mathrm{tot}},q^{\mathrm{res}})$.  Similarly, again as $p^{\mathrm{tot}} = \overline{p}^{\mathrm{tot}}$, the solution $(p,q,\overline{u})$ satisfies constraints \eqref{constr:Ramp}.

Therefore, the candidate solution $(p,q,\overline{u},v,\Delta,\theta,\tau,\phi)$ is feasible for \eqref{SC-ACOPF}.  Comparing the objective, the terms in $R^{\mathrm{J}}_j$ that contain variable $p$ are $R^{\mathrm{pow}}_{jt}$, $C^{\mathrm{pow}}_{jt}$, and $C^{\mathrm{e}}_{j}$, all of which only contain variables $p^{\mathrm{tot}}$.  As the copy constraint enforces $p^{\mathrm{tot}} = \overline{p}^{\mathrm{tot}}$, it holds that $R^{\mathrm{J}}_j(\overline{p}_j,\overline{u}_j) = R^{\mathrm{J}}_j(p_j,\overline{u}_j)$.  Then, $\sum_{t \in \mathcal{T}} R^{\mathrm{T}}_t(p_t,q_t) + \sum_{j \in \mathcal{J}^{\mathrm{sd}}} R^{\mathrm{J}}_j(\overline{p}_j,\overline{u}_j) = \sum_{t \in \mathcal{T}} R^{\mathrm{T}}_t(p_t,q_t) + \sum_{j \in \mathcal{J}^{\mathrm{sd}}} R^{\mathrm{J}}_j(p_j,\overline{u}_j)$ and the solutions have the same objective value in their respective models. 
\end{proof}

\subsubsection*{Proof of Proposition~\ref{prop:DLFeasible}}

\begin{proof}
Under Assumption~\ref{assump:DataAssumptions}, there exists some solution $(\tilde{p},\tilde{q},\tilde{u},\tilde{v},\tilde{\Delta},\tilde{\theta},\tilde{\tau},\tilde{\phi})$ that is feasible for \linebreak
\eqref{SC-ACOPF}.  Take any $(\overline{p},\overline{q},\overline{u})$ such that $(\overline{p}_j, \overline{q}_j,\overline{u}_j) \in \mathcal{X}^{\mathrm{uc}}_j$ for all $j \in \mathcal{J}^{\mathrm{sd}}$ and $\overline{u}^{\mathrm{sh}}_{jt} \in \mathcal{X}^{\mathrm{sh}}_{jt}$ for all $j \in \mathcal{J}^{\mathrm{sh}}$ and $t \in \mathcal{T}$.
We construct the solution $(p,q,u)$ by the following rules.  Take $(p^{\mathrm{fr}},p^{\mathrm{to}},q^{\mathrm{fr}},q^{\mathrm{to}}) = (\tilde{p}^{\mathrm{fr}},\tilde{p}^{\mathrm{to}},\tilde{q}^{\mathrm{fr}},\tilde{q}^{\mathrm{to}})$ and $u = \overline{u}$.  Define $p^{\mathrm{sh}}_{jt} = g^{\mathrm{sh}}_j \overline{u}^{\mathrm{sh}}_{jt} \tilde{v}^2_{i_j t}$ and $q^{\mathrm{sh}}_{jt} = -b^{\mathrm{sh}}_j \overline{u}^{\mathrm{sh}}_{jt} \tilde{v}^2_{i_j t}$ for all $j \in \mathcal{J}^{\mathrm{sh}}$ and $t \in \mathcal{T}$.  All other values in $(p,q)$ take the corresponding values from $(\overline{p},\overline{q})$.  

This construction ensures that the solution $(p,q,u,\overline{p},\overline{q},\overline{u})$ is feasible for the copy constraints of \eqref{EQ}.  Further, the set constraints applied to the variables $(\overline{p},\overline{q},\overline{u})$ in \eqref{EQ} are satisfied by the selection of $(\overline{p},\overline{q},\overline{u})$.  By Lemma~\ref{lemma:setRelaxations}, $u^{\mathrm{sh}}_{jt} = \overline{u}^{\mathrm{sh}}_{jt} \in \mathcal{X}^{\mathrm{sh}}_{jt} \subseteq \mathcal{Y}^{\mathrm{sh}}_{jt}$ for all $j \in \mathcal{J}^{\mathrm{sh}}$ and $t \in \mathcal{T}$ and $(\overline{p}_{jt},\overline{q}_{jt},\overline{u}^{\mathrm{on}}_{jt}) \in \mathcal{Y}^{\mathrm{uc}}_{jt}$ for all $j \in \mathcal{J}^{\mathrm{sd}}$ and $t \in \mathcal{T}$.  As $u = \overline{u}$ satisfies \eqref{constr:UC}, by Assumption~\ref{assump:DataAssumptions}, $u$ satisfies \eqref{eq:SUSDRelaxation}. Further, as $(p,q,u)$ does not differ from $(\overline{p},\overline{q},\overline{u})$ in the variables that appear in constraints \eqref{constr:PIReal}--\eqref{constr:PIReactive} and \eqref{constr:RelaxPI}, we conclude that $(p_{jt},q_{jt},u^{\mathrm{on}}_{jt}) \in \mathcal{Y}^{\mathrm{uc}}_{jt}$ for all $j \in \mathcal{J}^{\mathrm{sd}}$ and $t \in \mathcal{T}$.

Now, consider the solution $(p_t,q_t,\overline{u}^{\mathrm{sh}}_t,\tilde{v}_t,\tilde{\Delta}_t,\tilde{\theta}_t,\tilde{\tau}_t,\tilde{\phi}_t)$.  The components of $(p,q)$ which appear in constraints \eqref{constr:ACAngleDiff}--\eqref{constr:DCSets}, namely $(p^{\mathrm{fr}},p^{\mathrm{to}},p^{\mathrm{sh}},q^{\mathrm{fr}},q^{\mathrm{to}},q^{\mathrm{sh}})$ are constructed to ensure feasibility for these constraints with respect to this solution.  As a result, we see that $(p_t,q_t,\overline{u}^{\mathrm{sh}}_t) = (p_t,q_t,u^{\mathrm{sh}}_t) \in \mathcal{X}^{\mathrm{ac}}_t$ for all $t \in \mathcal{T}$.  This shows that $(p,q,u,\overline{p},\overline{q},\overline{u})$ is feasible for \eqref{EQ}.
\end{proof}

\subsubsection*{Proof of Theorem~\ref{thm:basicAPMInfeasibility}}

\begin{proof}
    To prove the first result, consider any iteration $\tau > 1$ and time period $t \in \mathcal{T}$.  We analyze the generic form \eqref{eq:genericSubproblem} of the model $Z^{\mathrm{T}}_t(\overline{p}^{(\tau-1)}_t,\overline{q}^{(\tau-1)}_t,\overline{u}^{(\tau-1)}_t;\rho_\tau)$.
    
    Let the polyhedron $\mathcal{X}$ be described by $(B,b)$, so that $\mathcal{X} = \{x \,:\, B x \leq b\}$.  Let $(x,y,z)$ be a stationary point of \eqref{eq:genericSubproblem}, in that it satisfies the Karush-Kuhn-Tucker (KKT) conditions.  There must exist some $(\lambda_1,\lambda_2)$ such that $(x,y,z)$ satisfies the following conditions, which are a subset of the KKT conditions:
    \begin{subequations}
        \label{eq:KKT}
        \begin{align}
            & 0 \in -\begin{bmatrix}
                \partial_x (F(x,y,z) - \rho \Gamma(x,\overline{x}))\\
                \partial_y F(x,y,z)
            \end{bmatrix} + \begin{bmatrix}
                B\\
                0
            \end{bmatrix} \tp \lambda_1 + \begin{bmatrix}
                A(z)\\
                -I
            \end{bmatrix} \tp \lambda_2, \label{eq:KKTStationary}\\
            & \lambda_1 \tp (b - B x) = 0, \label{eq:KKTComplementary}\\
            & \lambda_1 \geq 0, \label{eq:KKTDualFeasible}
        \end{align}
    \end{subequations}
    where $I$ represents the identity matrix and $A(z)$ gives the matrix form of the linear map $A_z$.  Note that $\partial_x(F(x,y,z) - \rho \Gamma(x,\overline{x})) = \partial_x F(x,y,z) - \rho \nabla_x \Gamma(x,\overline{x})$ as $\Gamma$ is a smooth function \citepappendix[][Exercise 8.8]{rockafellar2009variationalec}.

    Fix $(\lambda_1,\lambda_2)$ that satisfies \eqref{eq:KKT} and denote by $\nabla F \in \begin{bmatrix}
        \partial_x F(x,y,z)\\
        \partial_y F(x,y,z)
    \end{bmatrix}$ the (partial) subgradient of $F$ at $(x,y,z)$ that satisfies \eqref{eq:KKTStationary}; that is, 
    $$\nabla F = \begin{bmatrix}
        2 \rho E(x - \overline{x}) + B \tp \lambda_1 + A(z) \tp \lambda_2\\
        - \lambda_2
    \end{bmatrix}.$$

    Observe that the iterates $(\overline{p}^{(\tau-1)}_j,\overline{q}^{(\tau-1)}_j,\overline{u}^{(\tau-1)}_j)$ are feasible for the subproblems $Z^{\mathrm{J}}_j$ for all $j \in \mathcal{J}^{\mathrm{sd}}$ and $\overline{u}^{\mathrm{sh}(\tau-1)}_{jt}$ are feasible for $Z^{\mathrm{SH}}_{jt}$ for all $j \in \mathcal{J}^{\mathrm{sh}}$.  Therefore, $\overline{x} \in \mathcal{X}$ by the structure of model \eqref{eq:genericSubproblem}. With $\overline{y} = A(z) \overline{x}$, we have that $(\overline{x},\overline{y},z)$ is feasible for \eqref{eq:genericSubproblem}.

    By the above properties,
    \begin{equation*}
        \begin{aligned}
            \nabla F \tp \begin{bmatrix}
                x - \overline{x}\\
                y - \overline{y}
            \end{bmatrix} & = \begin{bmatrix}
                2 \rho E (x - \overline{x}) + B \tp \lambda_1 + A(z) \tp \lambda_2\\
                - \lambda_2
            \end{bmatrix} \tp \begin{bmatrix}
                x - \overline{x}\\
                y - \overline{y}
            \end{bmatrix}\\
            & = 2 \rho (x - \overline{x})\tp E \tp (x - \overline{x}) + \lambda_1 \tp B (x - \overline{x}) + \lambda_2 \tp (A(z) (x - \overline{x}) - (y - \overline{y}))\\
             & = 2 \rho (x - \overline{x})\tp E \tp (x - \overline{x}) + \lambda_1 \tp (b - B \overline{x})\\
             & \geq 2 \rho (x - \overline{x})\tp E \tp (x - \overline{x})\\
             & = 2 \rho (x - \overline{x})\tp E \tp E (x - \overline{x}) = 2 \rho \norm{E (x - \overline{x})}_2^2.
        \end{aligned}
    \end{equation*}
    where the third relation follows from \eqref{eq:KKTComplementary} and the definitions $[y,\ \overline{y}] = A(z) [x,\ \overline{x}]$, the fourth relation follows from the feasibility of $\overline{x}$ for $\mathcal{X}$ and \eqref{eq:KKTDualFeasible}, and the fifth relation follows from the property $E \tp E = E$.

    We note that $\norm{\nabla F}_2 \leq L$, and that $(x,y,z)$ and $(\overline{x},\overline{y},z)$ are feasible for \eqref{eq:genericSubproblem}, so
        \begin{equation*}
    \begin{aligned}
            2 \rho \norm{E (x - \overline{x})}_2^2 & \leq \nabla F \tp \begin{bmatrix}
                x - \overline{x}\\
                y - \overline{y}
            \end{bmatrix} \leq \norm{\nabla F}_2 \norm{\begin{bmatrix}
                x - \overline{x}\\
                y - \overline{y}
            \end{bmatrix}}_2 \leq L D,
        \end{aligned}
    \end{equation*}
    where the second relation follows from the Cauchy-Schwarz inequality.  Therefore, $\Gamma(x,\overline{x}) = \norm{E (x - \overline{x})}_2^2 \leq \frac{LD}{2 \rho}$.

    We now apply this result to the problems $Z^{\mathrm{T}}_t(\overline{p}^{(\tau-1)}_t,\overline{q}^{(\tau-1)}_t,\overline{u}^{(\tau-1)}_t;\rho_\tau)$ in each period $t \in \mathcal{T}$, where the feasible region diameters are given by $D_t$ and the Lipschitz constants by $L_t$.  Recall that $d$ is the number of copied variables.  By the separability of $\Gamma_2$ and the standard inequality relating the $L_1$ and $L_2$ norm in $\mathbb{R}^d$,
    \begin{equation*}
        \begin{aligned}
            & \Gamma_1(p^{(\tau)},q^{(\tau)},u^{(\tau)},\overline{p}^{(\tau-1)},\overline{q}^{(\tau-1)},\overline{u}^{(\tau-1)})\\
            \leq\ & \sqrt{d \Gamma_2(p^{(\tau)},q^{(\tau)},u^{(\tau)},\overline{p}^{(\tau-1)},\overline{q}^{(\tau-1)},\overline{u}^{(\tau-1)})} \\
            =\ & \sqrt{d \sum_{t \in \mathcal{T}}  \Gamma_2(p^{(\tau)}_t,q^{(\tau)}_t,u^{(\tau)}_t,\overline{p}^{(\tau-1)}_t,\overline{q}^{(\tau-1)}_t,\overline{u}^{(\tau-1)}_t)}\\
            \leq\ & \sqrt{\frac{\sum_{t \in \mathcal{T}} d L_t D_t}{2 \rho_{\tau}}}.
        \end{aligned}
    \end{equation*}
    This proves the first result.

    Now, we consider the second result.  It holds that
    \begin{equation*}
        \begin{aligned}
            & \Gamma_1(\overline{p}^{(\tau)},\overline{q}^{(\tau)},\overline{u}^{(\tau)},\overline{p}^{(\tau-1)},\overline{q}^{(\tau-1)},\overline{u}^{(\tau-1)})\\
            \leq \ & \Gamma_1 (p^{(\tau)},q^{(\tau)},u^{(\tau)},\overline{p}^{(\tau-1)},\overline{q}^{(\tau-1)},\overline{u}^{(\tau-1)}) + \Gamma_1 (p^{(\tau)},q^{(\tau)},u^{(\tau)},\overline{p}^{(\tau)},\overline{q}^{(\tau)},\overline{u}^{(\tau)})\\
            \leq \ & 2 \Gamma_1 (p^{(\tau)},q^{(\tau)},u^{(\tau)},\overline{p}^{(\tau-1)},\overline{q}^{(\tau-1)},\overline{u}^{(\tau-1)}) + \frac{R}{\rho_{\tau}}\\
            \leq \ & \frac{R}{\rho_\tau} + \sqrt{\frac{2\sum_{t \in \mathcal{T}} d L_t D_t}{\rho_{\tau}}},
        \end{aligned}
    \end{equation*}
    where the first inequality follows from the triangle inequality, the second from Lemma~\ref{lemma:basicAPMDevicePenalty}, and the third from the first result of this theorem.  This application of Lemma~\ref{lemma:basicAPMDevicePenalty} requires the global optimality of subproblems $Z^{\mathrm{J}}_j$.  This proves the second result.

    Last, 
    \begin{equation*}
        \begin{aligned}
            & \Gamma_1(p^{(\tau+1)},q^{(\tau+1)},u^{(\tau+1)},p^{(\tau)},q^{(\tau)},u^{(\tau)})\\
            \leq \ & \Gamma_1 (p^{(\tau+1)},q^{(\tau+1)},u^{(\tau+1)},\overline{p}^{(\tau)},\overline{q}^{(\tau)},\overline{u}^{(\tau)}) + \Gamma_1 (p^{(\tau)},q^{(\tau)},u^{(\tau)},\overline{p}^{(\tau)},\overline{q}^{(\tau)},\overline{u}^{(\tau)})\\
            \leq \ & \sqrt{\frac{\sum_{t \in \mathcal{T}} d L_t D_t}{2 \rho_{\tau+1}}} + \frac{R}{\rho_{\tau}} + \Gamma_1 (p^{(\tau)},q^{(\tau)},u^{(\tau)},\overline{p}^{(\tau-1)},\overline{q}^{(\tau-1)},\overline{u}^{(\tau-1)})\\
            \leq \ & \frac{R}{\rho_{\tau}} + \sqrt{\frac{2\sum_{t \in \mathcal{T}} d L_t D_t}{\min\{\rho_{\tau},\rho_{\tau+1}\}}}
        \end{aligned}
    \end{equation*}
    where the first inequality follows from the triangle inequality and the remaining inequalities follow from the first result of this theorem and Lemma~\ref{lemma:basicAPMDevicePenalty}.
    This proves the third result.
\end{proof}

\subsubsection*{Proof of Theorem~\ref{thm:bAPMConvergence}}

\begin{proof}
    Consider some limit point $(p^*,q^*,u^*,\overline{p}^*,\overline{q}^*,\overline{u}^*)$ of the sequence $(p^{(\tau)},q^{(\tau)},u^{(\tau)},\overline{p}^{(\tau)},\overline{q}^{(\tau)},\overline{u}^{(\tau)})$ generated by Algorithm~\ref{alg:basicAPM}, and let $\{\tau_i\}_{i=1}^\infty$ give the subsequence of indices which converge to this limit point.  
    
    We first demonstrate that this point is feasible for \eqref{EQ}.  By Lemma~\ref{lemma:basicAPMDevicePenalty} and Theorem~\ref{thm:basicAPMInfeasibility},
    \begin{equation*}
        \begin{aligned}
            0 \leq &\  \lim_{i \rightarrow \infty} \Gamma_1(p^{(\tau_i)},q^{(\tau_i)},u^{(\tau_i)},\overline{p}^{(\tau_i)},\overline{q}^{(\tau_i)},\overline{u}^{(\tau_i)})\\
            \leq &\  \lim_{i \rightarrow \infty} \frac{R}{\rho_{\tau_i}} + \Gamma_1 (p^{(\tau_i)},q^{(\tau_i)},u^{(\tau_i)},\overline{p}^{(\tau_i-1)},\overline{q}^{(\tau_i-1)},\overline{u}^{(\tau_i-1)})\\
            \leq &\  \lim_{i \rightarrow \infty} \frac{R}{\rho_{\tau_i}} + \sqrt{\frac{C}{2 \rho_{\tau_i}}}= 0.
        \end{aligned}
    \end{equation*}
    As $\Gamma_1$ is continuous, this gives that
    $$\Gamma_1(p^*,q^*,u^*,\overline{p}^*,\overline{q}^*,\overline{u}^*) = \lim_{i\rightarrow \infty} \Gamma_1(p^{(\tau_i)},q^{(\tau_i)},u^{(\tau_i)},\overline{p}^{(\tau_i)},\overline{q}^{(\tau_i)},\overline{u}^{(\tau_i)}) = 0.$$
    This is equivalent to
    $$({p^{\mathrm{tot}*}},{p^{\mathrm{su}*}},{p^{\mathrm{sd}*}},{q^{\mathrm{tot}*}},{q^{\mathrm{res}*}},{u^{\mathrm{on}*}},{u^{\mathrm{sh}*}}) = ({\overline{p}^{\mathrm{tot}*}},{\overline{p}^{\mathrm{su}*}},{\overline{p}^{\mathrm{sd}*}},{\overline{q}^{\mathrm{tot}*}},{\overline{q}^{\mathrm{res}*}},{\overline{u}^{\mathrm{on}*}},{\overline{u}^{\mathrm{sh}*}}),$$
    and thus the limit point satisfies the copy constraint of \eqref{EQ}.  

    Algorithm~\ref{alg:basicAPM} guarantees that the solution $(p^{(\tau_i)}_t,q^{(\tau_i)}_t,u^{(\tau_i)}_t)$ is feasible for $Z^{\mathrm{T}}_t$ for all $t \in \mathcal{T}$, the solution $(\overline{p}^{(\tau_i)}_j,\overline{q}^{(\tau_i)}_j,\overline{u}^{(\tau_i)}_j)$ is feasible for $Z^{\mathrm{J}}_j$ for all $j \in \mathcal{J}^{\mathrm{sd}}$, and $\overline{u}^{\mathrm{sh}(\tau_i)}_{jt}$ is feasible for $Z^{\mathrm{SH}}_{jt}$ for all $j \in \mathcal{J}^{\mathrm{sh}}$ and $t \in \mathcal{T}$.  Under Assumption~\ref{assump:DataAssumptions}, the subproblem feasible regions are closed, and thus the limit $(p^*,q^*,u^*,\overline{p}^*,\overline{q}^*,\overline{u}^*)$ is contained in the appropriate sets.  This demonstrates that the limit point is feasible for \eqref{EQ}.

    Next, we show that the limit point is optimal for \eqref{EQ} with the first variable block $(p,q,u)$ fixed to $(p^*,q^*,u^*)$.  Consider the model \eqref{EQ} with the additional constraint that $(p,q,u) = (p^*,q^*,u^*)$, which can be written as
    \begin{equation}
        \label{proofeq:partialMinUC}
        \begin{aligned}
        \max_{\substack{\overline{p},\overline{q},\overline{u}}} \quad & \sum_{t \in \mathcal{T}} R^{\mathrm{T}}_t(p_t^*,q_t^*) + \sum_{j \in \mathcal{J}^{\mathrm{sd}}} R^{\mathrm{J}}_j(\overline{p}_j,\overline{u}_j)\\
        \text{s.t.} \quad & \begin{aligned}[t]
        & (\overline{p}_j, \overline{q}_j,\overline{u}_j) \in \mathcal{X}^{\mathrm{uc}}_j \quad & \forall j \in \mathcal{J}^{\mathrm{sd}},
        \end{aligned}\\
        & \overline{p}^{\mathrm{tot}} = {p^{\mathrm{tot}*}},\ \overline{p}^{\mathrm{su}} = {p^{\mathrm{su}*}},\ \overline{p}^{\mathrm{sd}} = {p^{\mathrm{sd}*}},\ \overline{q}^{\mathrm{tot}} = {q^{\mathrm{tot}*}},\ \overline{q}^{\mathrm{res}} = {q^{\mathrm{res}*}},\ \overline{u}^{\mathrm{on}} = {u^{\mathrm{on}*}},\ \overline{u}^{\mathrm{sh}} = {u^{\mathrm{sh}*}}.
        \end{aligned}
    \end{equation}
    As above, $(\overline{p}^*,\overline{q}^*,\overline{u}^*)$ is feasible for this problem.  Consider any feasible solution $(\tilde{p},\tilde{q},\tilde{u})$ to \eqref{proofeq:partialMinUC}.

    Theorem~\ref{thm:basicAPMInfeasibility} implies that the maximum difference between the copied variables of $(\overline{p}^{(\tau)},\overline{q}^{(\tau)},\overline{u}^{(\tau)})$ becomes arbitrarily small with increasing $\rho_{\tau}$, i.e., the copied variables of the iterates form a Cauchy sequence.  Thus, the iterates converge to $(\overline{p}^*,\overline{q}^*,\overline{u}^*)$ in the copied variables.  Binary variables admit only unit changes in value, which implies that there is some iteration $\tau'$ after which the copied binary variables $\overline{u}^{\mathrm{on}}$ and $\overline{u}^{\mathrm{sh}}$ never deviate from their values ${\overline{u}^{\mathrm{on}*}}$ and ${\overline{u}^{\mathrm{sh}*}}$ in the limit point.
    By constraints \eqref{constr:UCSUSDDef1}--\eqref{constr:UCSUSDDef2}, and \eqref{constr:UCBinary}, the variables $\overline{u}^{\mathrm{su}}$ and $\overline{u}^{\mathrm{sd}}$ are uniquely defined by $\overline{u}^{\mathrm{on}}$, and therefore, after iteration $\tau'$, all binary variables in solution $\overline{u}^{(\tau)}$ remain at their value in solution~$\overline{u}^*$.  Similarly, as feasibility for \eqref{proofeq:partialMinUC} requires $\tilde{u}^{\mathrm{on}} = u^{\mathrm{on}*}$ and $\tilde{u}^{\mathrm{sh}} = u^{\mathrm{sh}*}$, all binary variables in $\tilde{u}$ have the same value as in $\overline{u}^*$.

    Define by $$\mathcal{Z}^{\mathrm{uc}}(\overline{u}^*) := \{(\overline{p},\overline{q},\overline{u}) \,:\, \overline{u} = \overline{u}^*,\ (\overline{p}_j,\overline{q}_j,\overline{u}_j) \in \mathcal{X}^{\mathrm{uc}}_j\ \forall j \in \mathcal{J}^{\mathrm{sd}}\}.$$
    Clearly, $(\overline{p}^*,\overline{q}^*,\overline{u}^*) \in \mathcal{Z}^{\mathrm{uc}}(\overline{u}^*)$.  Also, by the logic above, $(\tilde{p},\tilde{q},\tilde{u}) \in \mathcal{Z}^{\mathrm{uc}}(\overline{u}^*)$ and  $(\overline{p}^{(\tau)},\overline{q}^{(\tau)},\overline{u}^{(\tau)}) \in \mathcal{Z}^{\mathrm{uc}}(\overline{u}^*)$ for all $\tau \geq \tau'$.  As all binary variables are fixed, $\mathcal{Z}^{\mathrm{uc}}(\overline{u}^*)$ is polyhedral under Assumption~\ref{assump:DataAssumptions}. 

    Consider some subsequence of $\{\tau_i\}_{i=1}^\infty$, denoted by $\{\tau'_i\}_{i=1}^\infty$, such that $\tau'_1 \geq \tau'$.  We apply Lemma~\ref{lemma:shiftedSequence} to polyhedron $\mathcal{Z}^{\mathrm{uc}}(\overline{u}^*)$, solutions $(\overline{p}^*,\overline{q}^*,\overline{u}^*)$ and $(\tilde{p},\tilde{q},\tilde{u})$, and sequence $(\overline{p}^{(\tau'_i)},\overline{q}^{(\tau'_i)},\overline{u}^{(\tau'_i)})$.  Denote by $\{\lambda_i\}_{i=1}^\infty$ a sequence of multipliers that satisfies the statement of Lemma~\ref{lemma:shiftedSequence}.  We then define $$(\tilde{p}^{(\tau'_i)},\tilde{q}^{(\tau'_i)},\tilde{u}^{(\tau'_i)}) = (\overline{p}^{(\tau'_i)},\overline{q}^{(\tau'_i)},\overline{u}^{(\tau'_i)}) + \lambda_i ((\tilde{p},\tilde{q},\tilde{u}) - (\overline{p}^*,\overline{q}^*,\overline{u}^*)).$$
    By application of Lemma~\ref{lemma:shiftedSequence}, the sequence $(\tilde{p}^{(\tau'_i)},\tilde{q}^{(\tau'_i)},\tilde{u}^{(\tau'_i)})$ is feasible for $\mathcal{Z}^{\mathrm{uc}}(\overline{u}^*)$ and thus is feasible for subproblems $Z^{\mathrm{J}}_j$, and has $\lim_{i \rightarrow \infty} (\tilde{p}^{(\tau'_i)},\tilde{q}^{(\tau'_i)},\tilde{u}^{(\tau'_i)}) = (\tilde{p},\tilde{q},\tilde{u})$.

    By feasibility for \eqref{proofeq:partialMinUC}, it holds that solutions $(\tilde{p},\tilde{q},\tilde{u})$ and $(\overline{p}^*,\overline{q}^*,\overline{u}^*)$ have the same values for the copied variables (which appear in $\Gamma_1$), and thus $(\tilde{p}^{(\tau'_i)},\tilde{q}^{(\tau'_i)},\tilde{u}^{(\tau'_i)})$ and $(\overline{p}^{(\tau'_i)},\overline{q}^{(\tau'_i)},\overline{u}^{(\tau'_i)})$ have the same values in those variables for all $i$.  As a result, for all $i \in \mathbb{Z}_{\geq}$ and $j \in \mathcal{J}^{\mathrm{sd}}$ and any solution $(p_j,q_j,u_j)$, $$\Gamma'_1(p_j,q_j,u_j,\overline{p}^{(\tau'_i)}_j,\overline{q}^{(\tau'_i)}_j,\overline{u}^{(\tau'_i)}_j) = \Gamma'_1(p_j,q_j,u_j,\tilde{p}^{(\tau'_i)}_j,\tilde{q}^{(\tau'_i)}_j,\tilde{u}^{(\tau'_i)}_j).$$  To conclude, we observe that
    \begin{equation*}
        \begin{aligned}
            \sum_{j \in \mathcal{J}^{\mathrm{sd}}} R^{\mathrm{J}}_j(\overline{p}_j^*,\overline{u}_j^*) =\ & \lim_{i \rightarrow \infty} \sum_{j \in \mathcal{J}^{\mathrm{sd}}} R^{\mathrm{J}}_j(\overline{p}_j^{(\tau'_i)},\overline{u}_j^{(\tau'_i)})\\
            =\ & \lim_{i \rightarrow \infty} \sum_{j \in \mathcal{J}^{\mathrm{sd}}} R^{\mathrm{J}}_j(\overline{p}_j^{(\tau'_i)},\overline{u}_j^{(\tau'_i)}) - \\
            & \qquad \rho_{\tau'_i} \left (\Gamma'_1(p^{(\tau'_i)}_j,q^{(\tau'_i)}_j,u^{(\tau'_i)}_j,\overline{p}^{(\tau'_i)}_j,\overline{q}^{(\tau'_i)}_j,\overline{u}^{(\tau'_i)}_j) - \Gamma'_1(p^{(\tau'_i)}_j,q^{(\tau'_i)}_j,u^{(\tau'_i)}_j,\tilde{p}^{(\tau'_i)}_j,\tilde{q}^{(\tau'_i)}_j,\tilde{u}^{(\tau'_i)}_j) \right )\\
            \geq\ & \lim_{i \rightarrow \infty} \sum_{j \in \mathcal{J}^{\mathrm{sd}}} R^{\mathrm{J}}_j(\tilde{p}_j^{(\tau'_i)},\tilde{u}_j^{(\tau'_i)}) - \\
            & \qquad \rho_{\tau'_i} \left (\Gamma'_1(p^{(\tau'_i)}_j,q^{(\tau'_i)}_j,u^{(\tau'_i)}_j,\tilde{p}^{(\tau'_i)}_j,\tilde{q}^{(\tau'_i)}_j,\tilde{u}^{(\tau'_i)}_j) - \Gamma'_1(p^{(\tau'_i)}_j,q^{(\tau'_i)}_j,u^{(\tau'_i)}_j,\tilde{p}^{(\tau'_i)}_j,\tilde{q}^{(\tau'_i)}_j,\tilde{u}^{(\tau'_i)}_j) \right )\\
            =\ & \lim_{i \rightarrow \infty} \sum_{j \in \mathcal{J}^{\mathrm{sd}}} R^{\mathrm{J}}_j(\tilde{p}_j^{(\tau'_i)},\tilde{u}_j^{(\tau'_i)})\\
            =\ & \sum_{j \in \mathcal{J}^{\mathrm{sd}}} R^{\mathrm{J}}_j(\tilde{p}_j,\tilde{u}_j),
        \end{aligned}
    \end{equation*}
    where the inequality follows from the global optimality of iterates $(\overline{p}^{(\tau)},\overline{q}^{(\tau)},\overline{u}^{(\tau)})$ for subproblems~$Z^{\mathrm{J}}_j$.  This proves that the solution $(\overline{p}^*,\overline{q}^*,\overline{u}^*)$ is optimal for \eqref{proofeq:partialMinUC}, and thus satisfies the global partial optimality condition for its block.

    Last, we consider the other block of variables and show that the limit point is optimal for \eqref{EQ} with the second variable block $(\overline{p},\overline{q},\overline{u})$ fixed to $(\overline{p}^*,\overline{q}^*,\overline{u}^*)$.  We again analyze the generic form \eqref{eq:genericSubproblem}, which has limit point $(x^*,y^*,z^*,\overline{x}^*,\overline{y}^*,\overline{z}^*)$ that corresponds to $(p^*,q^*,u^*,\overline{p}^*,\overline{q}^*,\overline{u}^*)$.  Problem \eqref{EQ} with the additional constraint $(\overline{p},\overline{q},\overline{u}) = (\overline{p}^*,\overline{q}^*,\overline{u}^*)$ can be written in this generic form as 
    \begin{equation}
        \label{proofeq:partialMinAC}
        \begin{aligned}
        \max_{x,y,z} \quad & \sum_{t \in \mathcal{T}} F_t(x_t,y_t,z_t) + \sum_{j \in \mathcal{J}^{\mathrm{sd}}} R^{\mathrm{J}}_j(\overline{p}_j^*,\overline{u}_j^*)\\
        \text{s.t.} \quad & \begin{aligned}[t]
            & x_t \in \mathcal{X}_t \quad & \forall t \in \mathcal{T},\\
            & y_t = A_{z_t}(x_t) \quad & \forall t \in \mathcal{T},\\
            & z_t \in \mathcal{Z}_t \quad & \forall t \in \mathcal{T},\\
            & E x_t = E \overline{x}^*_t \quad & \forall t \in \mathcal{T},
        \end{aligned}
        \end{aligned}
    \end{equation}
    where the objects $(F, \mathcal{X}, \mathcal{Z})$ for each time period are now indexed by $t$, and $Ex_t$ selects the variables which appear in the copy constraint.  We prove the result by demonstrating that $(x^*,y^*,z^*)$ is optimal for \eqref{proofeq:partialMinAC}.

    Consider any $(\tilde{x},\tilde{y},\tilde{z})$ feasible for \eqref{proofeq:partialMinAC}.  For each $t \in \mathcal{T}$, we apply Lemma~\ref{lemma:shiftedSequence} to polyhedron $\mathcal{X}_t$, solutions $x^*_t$ and $\tilde{x}_t$, and sequence $x^{(\tau_i)}_t$.  Note that $(x^{(\tau_i)},y^{(\tau_i)},z^{(\tau_i)})$ corresponds to the sequence $(p^{(\tau_i)},q^{(\tau_i)},u^{(\tau_i)})$.  Denote by $\{\lambda_{it}\}_{i=1}^\infty$ a sequence of multipliers that satisfies the statement of Lemma~\ref{lemma:shiftedSequence}.  We then define 
    $$\tilde{x}^{(\tau_i)}_t = x^{(\tau_i)}_t + \lambda_{it} ( \tilde{x}_t - x^*_t) \qquad \text{and} \qquad \tilde{y}^{(\tau_i)}_t = A_{\tilde{z}_t} (\tilde{x}^{(\tau_i)}_t).$$
    By Lemma~\ref{lemma:shiftedSequence}, the sequence $\tilde{x}^{(\tau_i)}_t \in \mathcal{X}_t$, so $(\tilde{x}^{(\tau_i)}_t,\tilde{y}^{(\tau_i)}_t,\tilde{z}_t)$ is feasible for \eqref{eq:genericSubproblem} at time $t$ and thus for subproblem $Z^{\mathrm{T}}_t$.  Note that this sequence holds the variables $\tilde{z}_t$ constant.  Lemma~\ref{lemma:shiftedSequence} gives that $\lim_{i \rightarrow \infty} \tilde{x}^{(\tau_i)}_t = \tilde{x}_t$.  As $\tilde{y}_t = A_{\tilde{z}_t}(\tilde{x}_t)$ and $A_{\tilde{z}_t}$ is continuous, $\lim_{i \rightarrow \infty} \tilde{y}^{(\tau_i)}_t = \tilde{y}_t$.

    By feasibility for \eqref{proofeq:partialMinAC}, the solutions $\tilde{x}_t$ and $x^*_t$ have the same values for the copied variables $E x_t$ that appear in penalty function $\Gamma$.  As a result, $\tilde{x}^{(\tau_i)}_t$ and $x^{(\tau_i)}_t$ have the same values in these variables for all $i$. Thus, for all $i \in \mathbb{Z}_{\geq}$ and $t \in \mathcal{T}$ and any solution $\overline{x}_t$, 
    $$\Gamma(x^{(\tau_i)}_t,\overline{x}_t) = \Gamma(\tilde{x}^{(\tau_i)}_t,\overline{x}_t).$$
    To conclude, we have that
    \begin{equation*}
        \begin{aligned}
            \sum_{t \in \mathcal{T}} F_t(x^*_t,y^*_t,z^*_t) & =\  \lim_{i \rightarrow \infty} \sum_{t \in \mathcal{T}} F_t(x^{(\tau_i)}_t,y^{(\tau_i)}_t,z^{(\tau_i)}_t)\\
            & =\ \lim_{i \rightarrow \infty} \sum_{t \in \mathcal{T}} F_t(x^{(\tau_i)}_t,y^{(\tau_i)}_t,z^{(\tau_i)}_t) - \rho_{\tau_i} \left (\Gamma(x^{(\tau_i)}_t,\overline{x}^{(\tau_i-1)}_t) - \Gamma(\tilde{x}^{(\tau_i)}_t,\overline{x}^{(\tau_i-1)}_t) \right ) \\
            & \geq\ \lim_{i \rightarrow \infty} \sum_{t \in \mathcal{T}} F_t(\tilde{x}^{(\tau_i)}_t,\tilde{y}^{(\tau_i)}_t,\tilde{z}_t) - \rho_{\tau_i} \left (\Gamma(\tilde{x}^{(\tau_i)}_t,\overline{x}^{(\tau_i-1)}_t) - \Gamma(\tilde{x}^{(\tau_i)}_t,\overline{x}^{(\tau_i-1)}_t) \right )\\
            & =\ \lim_{i \rightarrow \infty} \sum_{t \in \mathcal{T}} F_t(\tilde{x}^{(\tau_i)}_t,\tilde{y}^{(\tau_i)}_t,\tilde{z}_t)\\
            & =\ \sum_{t \in \mathcal{T}} F_t(\tilde{x}_t,\tilde{y}_t,\tilde{z}_t),
        \end{aligned}
    \end{equation*}
    where the inequality follows from the global optimality of iterates $(x^{(\tau)}_t,y^{(\tau)}_t,z^{(\tau)}_t)$ for subproblems~$Z^{\mathrm{T}}_t$.  This demonstrates the optimality of $(x^*,y^*,z^*)$ for \eqref{proofeq:partialMinAC}, and therefore $(p^*,q^*,u^*)$ satisfies the partial optimality condition for its block.  This demonstrates that $(p^*,q^*,u^*,\overline{p}^*,\overline{q}^*,\overline{u}^*)$ is a partial optimum for \eqref{EQ}.
\end{proof}

\subsubsection*{Proof of Proposition~\ref{prop:rampLPFeas}}

\begin{proof}
    Take any $j \in \mathcal{J}^{\mathrm{sd}}$ and $(\overline{p}_j,\overline{q}_j,\overline{u}_j)$ such that $(\overline{p}_j,\overline{q}_j,\overline{u}_j) \in \mathcal{X}_j^{\mathrm{uc}}$. For all $t \in \mathcal{T}$, construct
    \begin{equation*}
        \begin{aligned}
            \delta^-_{jt} & = \frac{1}{2} \min \left\{r^{\mathrm{d}}_{jt} + (\overline{p}^{\mathrm{tot}}_{jt} - \overline{p}^{\mathrm{tot}}_{j,t-1}),\ r^{\mathrm{u}}_{j,t+1} - (\overline{p}^{\mathrm{tot}}_{j,t+1} - \overline{p}^{\mathrm{tot}}_{jt}) \right\},\\
            \delta^+_{jt} & = \frac{1}{2} \min \left\{r^{\mathrm{u}}_{jt} - (\overline{p}^{\mathrm{tot}}_{jt} - \overline{p}^{\mathrm{tot}}_{j,t-1}),\ r^{\mathrm{d}}_{j,t+1} + (\overline{p}^{\mathrm{tot}}_{j,t+1} - \overline{p}^{\mathrm{tot}}_{jt}) \right\},
        \end{aligned}
    \end{equation*}
    which satisfies \eqref{constr:rampSafeDownHeur}--\eqref{constr:rampSafeUpHeur} and
    \begin{equation*}
        \begin{aligned}
            \delta^-_{jt} + \delta^+_{j,t-1} & \leq \frac{1}{2} \left (r^{\mathrm{d}}_{jt} + (\overline{p}^{\mathrm{tot}}_{jt} - \overline{p}^{\mathrm{tot}}_{j,t-1}) + r^{\mathrm{d}}_{jt} + (\overline{p}^{\mathrm{tot}}_{jt} - \overline{p}^{\mathrm{tot}}_{j,t-1}) \right ) = r^{\mathrm{d}}_{jt} + (\overline{p}^{\mathrm{tot}}_{jt} - \overline{p}^{\mathrm{tot}}_{j,t-1}),\\
            \delta^-_{j,t-1} + \delta^+_{jt} & \leq \frac{1}{2} \left (r^{\mathrm{u}}_{jt} - (\overline{p}^{\mathrm{tot}}_{jt} - \overline{p}^{\mathrm{tot}}_{j,t-1}) + r^{\mathrm{u}}_{jt} - (\overline{p}^{\mathrm{tot}}_{jt} - \overline{p}^{\mathrm{tot}}_{j,t-1}) \right ) = r^{\mathrm{u}}_{jt} - (\overline{p}^{\mathrm{tot}}_{jt} - \overline{p}^{\mathrm{tot}}_{j,t-1}).
        \end{aligned}
    \end{equation*}
    The construction thus satisfies the constraints \eqref{constr:rampSafeDown}--\eqref{constr:rampSafeUp} and is feasible for \eqref{eq:rampSafeLP}.

    Next, we demonstrate nonnegativity and boundedness.  Consider any $(\delta^-_j,\delta^+_j)$ that satisfies \eqref{constr:rampSafeDown}--\eqref{constr:rampSafeUpHeur}. Since $(\overline{p}_j,\overline{q}_j,\overline{u}_j)$ satisfies the ramping constraints \eqref{constr:Ramp}, we have $-r^{\mathrm{d}}_{jt} \leq \overline{p}^{\mathrm{tot}}_{jt} - \overline{p}^{\mathrm{tot}}_{j,t-1} \leq r^{\mathrm{u}}_{jt}$ for all $t \in \mathcal{T}$.  Then, constraints \eqref{constr:rampSafeDownHeur}--\eqref{constr:rampSafeUpHeur} imply nonnegativity. Constraint~\eqref{constr:rampSafeDown} provides an upper bound on~$\delta^-_{jt}$. Combining $\delta^+_{j,t-1} \geq 0$ and \eqref{constr:rampSafeDown} gives that $\delta^-_{jt} \leq r^{\mathrm{d}}_{jt} + (\overline{p}^{\mathrm{tot}}_{jt} - \overline{p}^{\mathrm{tot}}_{j,t-1})$. Similarly, from constraint~\eqref{constr:rampSafeUp}, $\delta^+_{jt} \leq r^{\mathrm{u}}_{jt} - (\overline{p}^{\mathrm{tot}}_{jt} - \overline{p}^{\mathrm{tot}}_{j,t-1})$. Thus, the set of feasible ($\delta^-_{jt},\delta^+_{jt}$) is bounded.
\end{proof}

\subsubsection*{Proof of Proposition~\ref{prop:rampLPValid}}

\begin{proof}
    Consider any $(\overline{p},\overline{u},\delta)$ that satisfies \eqref{constr:rampSafeDown}--\eqref{constr:rampSafeUp} for all $j \in \mathcal{J}^{\mathrm{sd}}$.  Take any $p$ such that $p^{\mathrm{tot}}_{jt}$ satisfies \eqref{eq:rampSafeConstraint}, i.e.,
    $$\overline{p}^{\mathrm{tot}}_{jt} - \delta^-_{jt} \leq p^{\mathrm{tot}}_{jt} \leq \overline{p}^{\mathrm{tot}}_{jt} + \delta^+_{jt} \quad \forall j \in \mathcal{J}^{\mathrm{sd}},\ t \in \mathcal{T}.$$
    Constraints \eqref{eq:rampSafeConstraint}, \eqref{constr:rampSafeDown}, and definition \eqref{constr:rampSafeDownRangeDef} yield
    $$p^{\mathrm{tot}}_{jt} - p^{\mathrm{tot}}_{j,t-1} \geq \overline{p}^{\mathrm{tot}}_{jt} - \overline{p}^{\mathrm{tot}}_{j,t-1} - (\delta^-_{jt} + \delta^+_{j,t-1}) \geq -r^{\mathrm{d}}_{jt} = -d_t \left ( p^{\mathrm{rd}}_{j} \overline{u}^{\mathrm{on}}_{jt} + p^{\mathrm{rd,sd}}_j (1 - \overline{u}^{\mathrm{on}}_{jt}) \right ),$$
    and thus $(p,\overline{u})$ satisfies \eqref{constr:RampDown}.
    Similarly, constraints \eqref{eq:rampSafeConstraint}, \eqref{constr:rampSafeUp}, and definition \eqref{constr:rampSafeUpRangeDef} yield
    $$p^{\mathrm{tot}}_{jt} - p^{\mathrm{tot}}_{j,t-1} \leq \overline{p}^{\mathrm{tot}}_{jt} - \overline{p}^{\mathrm{tot}}_{j,t-1} + \delta^-_{j,t-1} + \delta^+_{jt} \leq r^{\mathrm{u}}_{jt} = d_t \left ( p^{\mathrm{ru}}_j (\overline{u}^{\mathrm{on}}_{jt} - \overline{u}^{\mathrm{su}}_{jt}) + p^{\mathrm{ru,su}}_j (\overline{u}^{\mathrm{su}}_{jt} - \overline{u}^{\mathrm{on}}_{jt} + 1) \right ),$$
    and thus $(p,\overline{u})$ satisfies \eqref{constr:RampUp}.
\end{proof}

\subsubsection*{Proof of Proposition~\ref{prop:energyViolationRestriction}}

\begin{proof}
The result relies on the inequality $\left [ \sum_{i} x_i \right ]^+ \leq \sum_i [x_i]^+$ for any $x$.  Then,
\begin{equation*}
    \begin{aligned}
        & \sum_{t \in \mathcal{T}} \sum_{W \in \mathcal{W}^{\mathrm{en,min}}_{jt}} d_t \left [\frac{e^{\mathrm{min}}_j(W) - \sum_{t' \in W} d_{t'} \overline{p}^{\mathrm{tot}}_{jt'}}{\sum_{t' \in W} d_{t'}} + \overline{p}^{\mathrm{tot}}_{jt} - p^{\mathrm{tot}}_{jt} \right ]^+\\
        = & \sum_{W \in \mathcal{W}^{\mathrm{en,min}}_{j}} \sum_{t \in W} d_t \left [\frac{e^{\mathrm{min}}_j(W) - \sum_{t' \in W} d_{t'} \overline{p}^{\mathrm{tot}}_{jt'}}{\sum_{t' \in W} d_{t'}} + \overline{p}^{\mathrm{tot}}_{jt} - p^{\mathrm{tot}}_{jt} \right ]^+\\
        \geq & \sum_{W \in \mathcal{W}^{\mathrm{en,min}}_{j}} \left [\sum_{t \in W} d_t \left ( \frac{e^{\mathrm{min}}_j(W) - \sum_{t' \in W} d_{t'} \overline{p}^{\mathrm{tot}}_{jt'}}{\sum_{t' \in W} d_{t'}} + \overline{p}^{\mathrm{tot}}_{jt} - p^{\mathrm{tot}}_{jt} \right ) \right ]^+\\
        = & \sum_{W \in \mathcal{W}^{\mathrm{en,min}}_{j}} \left [e^{\mathrm{min}}_j(W) - \sum_{t \in W} d_t p^{\mathrm{tot}}_{jt} \right ]^+.
    \end{aligned}
\end{equation*}
The first relation follows from the definition of $\mathcal{W}^{\mathrm{en,min}}_{jt}$ and the second from the aforementioned property of $[\cdot]^+$ and nonnegativity of $d_t$.  Similarly, for the energy maximum terms,
\begin{equation*}
    \begin{aligned}
        & \sum_{t \in \mathcal{T}} \sum_{W \in \mathcal{W}^{\mathrm{en,max}}_{jt}} d_t \left [\frac{-e^{\mathrm{max}}_j(W) + \sum_{t' \in W} d_{t'} \overline{p}^{\mathrm{tot}}_{jt'}}{\sum_{t' \in W} d_{t'}} + p^{\mathrm{tot}}_{jt} - \overline{p}^{\mathrm{tot}}_{jt} \right ]^+\\
        = & \sum_{W \in \mathcal{W}^{\mathrm{en,max}}_{j}} \sum_{t \in W} d_t \left [\frac{-e^{\mathrm{max}}_j(W) + \sum_{t' \in W} d_{t'} \overline{p}^{\mathrm{tot}}_{jt'}}{\sum_{t' \in W} d_{t'}} + p^{\mathrm{tot}}_{jt} - \overline{p}^{\mathrm{tot}}_{jt} \right ]^+\\
        \geq & \sum_{W \in \mathcal{W}^{\mathrm{en,max}}_{j}} \left [\sum_{t \in W} d_t \left ( \frac{-e^{\mathrm{max}}_j(W) + \sum_{t' \in W} d_{t'} \overline{p}^{\mathrm{tot}}_{jt'}}{\sum_{t' \in W} d_{t'}} + p^{\mathrm{tot}}_{jt} - \overline{p}^{\mathrm{tot}}_{jt} \right ) \right ]^+\\
        = & \sum_{W \in \mathcal{W}^{\mathrm{en,max}}_{j}} \left [-e^{\mathrm{max}}_j(W) + \sum_{t \in W} d_t p^{\mathrm{tot}}_{jt} \right ]^+.
    \end{aligned}
\end{equation*}
Together, as $c^{\mathrm{e}} \geq 0$, these properties yield the bound $C^{\mathrm{e}}_j(p) \leq \sum_{t \in \mathcal{T}} \overline{C}^{\mathrm{e}}_{jt}(p;\overline{p})$.
\end{proof}

\subsubsection*{Proof of Proposition~\ref{prop:SOCRelaxation}}

\begin{proof}
    Consider $t \in \mathcal{T}$ and $(p_t,q_t,u^{\mathrm{sh}}_t) \in \mathcal{X}^{\mathrm{ac}}_t$ with $u^{\mathrm{sh}}_{jt} \in \mathcal{Y}^{\mathrm{sh}}_{jt}$ for all $j \in \mathcal{J}^{\mathrm{sh}}$.  Then, there is some solution $(v_t,\Delta_t,\theta_t,\tau_t,\phi_t)$ such that this solution and $(p_t,q_t,u^{\mathrm{sh}}_t)$ are feasible for \eqref{constr:ACAngleDiff}--\eqref{constr:DCSets}.  Construct the following objects:
        \begin{equation*}
        \begin{gathered}
            c_{jt} = \frac{v_{i_j t} v_{i'_j t}}{\tau_{jt}} \cos (\Delta_{jt}),\ 
            s_{jt} = \frac{v_{i_j t} v_{i'_j t}}{\tau_{jt}} \sin (\Delta_{jt}),\ 
            \mu_{jt} = \frac{v_{i_j t}^2}{\tau_{jt}^2} \quad \forall j \in \mathcal{J}^{\mathrm{ac}},\\
            \omega_{it} = v_{it}^2 \quad \forall i \in \mathcal{I}, \quad
            \mu^{\mathrm{sh}}_{jt} = u^{\mathrm{sh}}_{jt} v_{i_jt}^2 \quad \forall j \in \mathcal{J}^{\mathrm{sh}}.
        \end{gathered}
    \end{equation*}
    We consider the feasibility of the solution $(c_t,s_t,p_t,q_t,\mu_t,\omega_t)$ for constraints \eqref{constr:SOC}.  First, by \eqref{constr:ACpfr}--\eqref{constr:ACqto} we have for all $j \in \mathcal{J}^{\mathrm{ac}}$ that
    \begin{equation*}
        \begin{aligned}
            g_{i_j j} \mu_{jt} - g_j c_{jt} - b_j s_{jt} & = \frac{ g_{i_jj} v_{i_jt}^2 }{\tau_{jt}^2} + \frac{v_{i_jt} v_{i'_jt}}{\tau_{jt}} \left ( -g_j \cos (\Delta_{jt}) - b_j \sin(\Delta_{jt}) \right ) = p^{\mathrm{fr}}_{jt},\\
            g_{i'_j j} \omega_{i'_jt} - g_j c_{jt} + b_j s_{jt} & = g_{i'_jj} v_{i'_jt}^2 + \frac{v_{i_jt} v_{i'_jt}}{\tau_{jt}} \left ( -g_j \cos (\Delta_{jt}) + b_j \sin(\Delta_{jt}) \right )= p^{\mathrm{to}}_{jt},\\
            -b_{i_j j} \mu_{jt} + b_j c_{jt} - g_j s_{jt} & = \frac{ -b_{i_jj} v_{i_jt}^2}{\tau_{jt}^2} + \frac{v_{i_jt} v_{i'_jt}}{\tau_{jt}} \left ( b_j \cos (\Delta_{jt}) - g_j \sin(\Delta_{jt}) \right ) = q^{\mathrm{fr}}_{jt},\\
            -b_{i'_j j} \omega_{i'_j t} + b_j c_{jt} + g_j s_{jt} & = -b_{i'_jj} v_{i'_jt}^2 + \frac{v_{i_jt} v_{i'_jt}}{\tau_{jt}} \left ( b_j \cos (\Delta_{jt}) + g_j \sin(\Delta_{jt}) \right ) = q^{\mathrm{to}}_{jt}.
        \end{aligned}
    \end{equation*}
    Next, using \eqref{constr:ShuntReal}--\eqref{constr:ShuntReactive}, we have for all $j \in \mathcal{J}^{\mathrm{sh}}$ that 
    $$g^{\mathrm{sh}}_j \mu^{\mathrm{sh}}_{jt} = g^{\mathrm{sh}}_j u^{\mathrm{sh}}_{jt} v_{i_j t}^2 = p^{\mathrm{sh}}_{jt} \quad \text{and} \quad -b^{\mathrm{sh}}_j \mu^{\mathrm{sh}}_{jt} = -b^{\mathrm{sh}}_j u^{\mathrm{sh}}_{jt} v_{i_j t}^2 = q^{\mathrm{sh}}_{jt}.$$
    This establishes that the solution satisfies constraints \eqref{constr:SOCpfr}--\eqref{constr:SOCShuntReactive}.

    We now consider the remaining constraints.  It holds for all $j \in \mathcal{J}^{\mathrm{ac}}$ that
    $$c_{jt}^2 + s_{jt}^2 = \frac{v_{i_j t}^2 v_{i'_j t}^2}{\tau_{jt}^2} \left ( \cos^2 (\Delta_{jt}) + \sin^2 (\Delta_{jt})  \right ) = \frac{v_{i_j t}^2 v_{i'_j t}^2}{\tau_{jt}^2} = \mu_{jt} \omega_{i'_j t},$$
    and \eqref{constr:SOCCosSin} holds.  By constraint \eqref{constr:ACSets}, nonnegativity of $\tau^{\mathrm{min}}$ and $\tau^{\mathrm{max}}$, and $\omega_{it} = v_{it}^2 \geq 0$,
    $$\frac{\omega_{i_j t}}{(\tau^{\mathrm{max}}_j)^2} \leq \frac{\omega_{i_j t}}{\tau_{jt}^2} = \mu_{jt} = \frac{\omega_{i_j t}}{\tau_{jt}^2} \leq \frac{\omega_{i_j t}}{(\tau^{\mathrm{min}}_j)^2},$$
    and \eqref{constr:SOCTapRatio} holds.  By definition of $\mathcal{Y}^{\mathrm{sh}}_{jt}$, for all $j \in \mathcal{J}^{\mathrm{sh}}$ we have $u^{\mathrm{sh,min}}_j \leq u^{\mathrm{sh}}_{jt} \leq u^{\mathrm{sh,max}}_j$. Thus,
    $$u^{\mathrm{sh,min}}_j \omega_{i_j t} \leq u^{\mathrm{sh}}_{jt} \omega_{i_j t} = u^{\mathrm{sh}}_{jt} v_{i_j t}^2 = \mu^{\mathrm{sh}}_{jt} = u^{\mathrm{sh}}_{jt} v_{i_j t}^2 = u^{\mathrm{sh}}_{jt} \omega_{i_j t} \leq u^{\mathrm{sh,max}}_j \omega_{i_j t},$$
    again using $\omega \geq 0$, and \eqref{constr:SOCShuntDomain} is satisfied.  For all $i \in \mathcal{I}$, due to \eqref{constr:BusSets} and $(v^{\mathrm{min}}, v^{\mathrm{max}}) \geq 0$,
    $$(v^{\mathrm{min}}_i)^2 \leq v_{it}^2 = \omega_{it} = v_{it}^2 \leq (v^{\mathrm{max}}_i)^2,$$
    and \eqref{constr:SOCBusSets} holds.  The solution trivially satisfies \eqref{constr:SOCDC}.  Therefore, $(p_t,q_t,\mu_t,\omega_t) \in \mathcal{X}^{\mathrm{soc}}_t$.  By construction, $\frac{\mu^{\mathrm{sh}}_{jt}}{\omega_{i_j t}} = u^{\mathrm{sh}}_{jt}$ for all $j \in \mathcal{J}^{\mathrm{sh}}$ and the result holds.
\end{proof}

\subsubsection*{Proof of Proposition~\ref{prop:AltPIEquivalent}}

\begin{proof}
    Fix some $j \in \mathcal{J}^{\mathrm{sd}}$ and $t \in \mathcal{T}$.  First, consider a solution $(p_{jt},q_{jt},u^{\mathrm{on}}_{jt}) \in \mathcal{Y}^{\mathrm{uc}}_{jt}$.  Then, let $(u^{\mathrm{su}}_j,u^{\mathrm{sd}}_j) \geq 0$ certify inclusion in $\mathcal{Y}^{\mathrm{uc}}_{jt}$.  Define 
    $$\chi^{\mathrm{su}}_{jt} = \sum_{t' \in \mathcal{T}^{\mathrm{supc}}_{jt}} u^{\mathrm{su}}_{jt'} \quad \text{and} \quad \chi^{\mathrm{sd}}_{jt} = \sum_{t' \in \mathcal{T}^{\mathrm{sdpc}}_{jt}} u^{\mathrm{sd}}_{jt'}.$$
    By \eqref{eq:SUSDRelaxation}, we have $u^{\mathrm{on}}_{jt} + \chi^{\mathrm{su}}_{jt} \leq 1$ and $u^{\mathrm{on}}_{jt} + \chi^{\mathrm{sd}}_{jt} \leq 1$.  Using these properties and $u^{\mathrm{on}}_{jt} \geq 0$, we have $\chi^{\mathrm{su}}_{jt} \leq 1$ and $\chi^{\mathrm{sd}}_{jt} \leq 1$.  Further, we have $u^{\mathrm{su}}_{jt'} \geq 0$, so $\chi^{\mathrm{su}}_{jt} \geq 0$.  Similarly, $\chi^{\mathrm{sd}}_{jt} \geq 0$. Thus, the solution $(p_{jt},q_{jt},u^{\mathrm{on}}_{jt},\chi_{jt})$ satisfies \eqref{constr:AltPISUMax}--\eqref{constr:AltPIDomain}.  By constraint \eqref{constr:RelaxPISUDef} and nonnegativity of $u^{\mathrm{su}}_{jt'}$ and  $p^{\mathrm{supc}}_{jtt'}$,
    $$0 \leq p^{\mathrm{su}}_{jt} \leq \sum_{t' \in \mathcal{T}^{\mathrm{supc}}_{jt}} p^{\mathrm{supc}}_{jtt'} u^{\mathrm{su}}_{jt'} \leq \max_{t' \in \mathcal{T}^{\mathrm{supc}}_{jt}} p^{\mathrm{supc}}_{jtt'} \sum_{t' \in \mathcal{T}^{\mathrm{supc}}_{jt}} u^{\mathrm{su}}_{jt'} = \chi^{\mathrm{su}}_{jt} \max_{t' \in \mathcal{T}^{\mathrm{supc}}_{jt}} p^{\mathrm{supc}}_{jtt'},$$
    and thus the solution satisfies \eqref{constr:AltPISUDef}. Similarly, constraint \eqref{constr:RelaxPISDDef} implies \eqref{constr:AltPISDDef}. Next, let $a = \uparrow$ if $j \in \mathcal{J}^{\mathrm{pr}}$ and otherwise $a = \downarrow$.  By constraint \eqref{constr:PIReactReserveMax}, we have
    $$q^{\mathrm{tot}}_{jt} \leq q^{\mathrm{max}}_{jt} \left ( u^{\mathrm{on}}_{jt} + \sum_{t' \in \mathcal{T}^{\mathrm{supc}}_{jt}} u^{\mathrm{su}}_{jt'} + \sum_{t' \in \mathcal{T}^{\mathrm{sdpc}}_{jt}} u^{\mathrm{sd}}_{jt'} \right ) + \beta^{\mathrm{max}}_j p^{\mathrm{tot}}_{jt} - q^{\mathrm{res}}_{jta} = q^{\mathrm{max}}_{jt} \left ( u^{\mathrm{on}}_{jt} + \chi^{\mathrm{su}}_{jt} + \chi^{\mathrm{sd}}_{jt} \right ) + \beta^{\mathrm{max}}_j p^{\mathrm{tot}}_{jt} - q^{\mathrm{res}}_{jta},$$
    and \eqref{constr:AltPIReactReserveMax} is satisfied.  Similarly, by \eqref{constr:PIReactReserveMin}, we see that \eqref{constr:AltPIReactReserveMin} holds.  Therefore, $(p_{jt},q_{jt},u^{\mathrm{on}}_{jt},\chi_{jt})$ satisfies constraints \eqref{constr:AltPI}.  As constraints \eqref{constr:PIRealDef}, \eqref{constr:PIDeviceReserveSet}, \eqref{constr:PINonneg}, and \eqref{constr:PIReactNonneg} do not contain $u^{\mathrm{su}}_{jt}$ and $u^{\mathrm{sd}}_{jt}$, the solution also satisfies these constraints, and $(p_{jt},q_{jt},u^{\mathrm{on}}_{jt},\chi_{jt}) \in \overline{\mathcal{Y}}^{\mathrm{uc}}_{jt}.$

    In the other direction, take $(p_{jt},q_{jt},u^{\mathrm{on}}_{jt})$ such that $(p_{jt},q_{jt},u^{\mathrm{on}}_{jt},\chi_{jt}) \in \overline{\mathcal{Y}}^{\mathrm{uc}}_{jt}$ for some $\chi_{jt}$.  
    Let
    $$\overline{t}^{\mathrm{su}} \in \argmax_{t' \in \mathcal{T}^{\mathrm{supc}}_{jt}} \ p^{\mathrm{supc}}_{jtt'} \quad \text{and} \quad \overline{t}^{\mathrm{sd}} \in \argmax_{t' \in \mathcal{T}^{\mathrm{sdpc}}_{jt}} \ p^{\mathrm{sdpc}}_{jtt'}.$$
    We construct the solution $(u^{\mathrm{su}}_j,u^{\mathrm{sd}}_j)$ for all $t' \in \mathcal{T}$ by the following rule:
    \begin{equation*}
        \begin{aligned}
            u^{\mathrm{su}}_{jt'} = \begin{cases}
                \chi^{\mathrm{su}}_{jt} &\quad \text{if } t' = \overline{t}^{\mathrm{su}}\\
                0 & \quad \text{otherwise}
            \end{cases} 
            \qquad
            \mathrm{and}
            \qquad
            u^{\mathrm{sd}}_{jt'} = \begin{cases}
                \chi^{\mathrm{sd}}_{jt} & \quad \text{if } t' = \overline{t}^{\mathrm{sd}}\\
                0 & \quad \text{otherwise.}
            \end{cases}
        \end{aligned}
    \end{equation*}
    Consider the solution $(p_{jt},q_{jt},u^{\mathrm{on}}_{jt},u^{\mathrm{su}}_j,u^{\mathrm{sd}}_j)$.
    By construction, we have $\sum_{t' \in \mathcal{T}^{\mathrm{supc}}_{jt}} u^{\mathrm{su}}_{jt'} = \chi^{\mathrm{su}}_{jt}$ and \linebreak
    $\sum_{t' \in \mathcal{T}^{\mathrm{sdpc}}_{jt}} u^{\mathrm{sd}}_{jt'} = \chi^{\mathrm{sd}}_{jt}$.  Further, $(u^{\mathrm{on}}_{jt},u^{\mathrm{su}}_j,u^{\mathrm{sd}}_j) \in [0,1]^{2T+1}$ by constraint~\eqref{constr:AltPIDomain}.  Let $a = \uparrow$ if $j \in \mathcal{J}^{\mathrm{pr}}$ and otherwise $a = \downarrow$.  Then, by constraint \eqref{constr:AltPIReactReserveMax} we have
    $$q^{\mathrm{tot}}_{jt} \leq q^{\mathrm{max}}_{jt} \left ( u^{\mathrm{on}}_{jt} + \chi^{\mathrm{su}}_{jt} + \chi^{\mathrm{sd}}_{jt} \right ) + \beta^{\mathrm{max}}_j p^{\mathrm{tot}}_{jt} - q^{\mathrm{res}}_{jta} = q^{\mathrm{max}}_{jt} \left ( u^{\mathrm{on}}_{jt} + \sum_{t' \in \mathcal{T}^{\mathrm{supc}}_{jt}} u^{\mathrm{su}}_{jt'} + \sum_{t' \in \mathcal{T}^{\mathrm{sdpc}}_{jt}} u^{\mathrm{sd}}_{jt'} \right ) + \beta^{\mathrm{max}}_j p^{\mathrm{tot}}_{jt} - q^{\mathrm{res}}_{jta},$$
    and \eqref{constr:PIReactReserveMax} is satisfied.  Similarly, \eqref{constr:AltPIReactReserveMin} implies \eqref{constr:PIReactReserveMin}. By constraint \eqref{constr:AltPISUMax},
    $$u^{\mathrm{on}}_{jt} + \sum_{t' \in \mathcal{T}^{\mathrm{supc}}_{jt}} u^{\mathrm{su}}_{jt'} = u^{\mathrm{on}}_{jt} + \chi^{\mathrm{su}}_{jt} \leq 1,$$
    and similarly by \eqref{constr:AltPISDMax}, $u^{\mathrm{on}}_{jt} + \sum_{t' \in \mathcal{T}^{\mathrm{sdpc}}_{jt}} u^{\mathrm{sd}}_{jt'} \leq 1$, so constraints \eqref{eq:SUSDRelaxation} are satisfied.  By definition of $\overline{\mathcal{Y}}^{\mathrm{uc}}_{jt}$, the solution satisfies \eqref{constr:PIRealDef}, \eqref{constr:PIDeviceReserveSet}, \eqref{constr:PINonneg}, and \eqref{constr:PIReactNonneg}.
    Next, by \eqref{constr:AltPISUDef} and construction of $u^{\mathrm{su}}_{jt'}$, we have $$0 \leq p^{\mathrm{su}}_{jt} \leq  \chi^{\mathrm{su}}_{jt} \max_{t' \in \mathcal{T}^{\mathrm{supc}}_{jt}} p^{\mathrm{supc}}_{jtt'} = u^{\mathrm{su}}_{j \overline{t}^{\mathrm{su}}}  p^{\mathrm{supc}}_{jt\overline{t}^{\mathrm{su}}} = \sum_{t' \in \mathcal{T}^{\mathrm{supc}}_{jt}} p^{\mathrm{supc}}_{jtt'} u^{\mathrm{su}}_{jt'},$$
    and constraint \eqref{constr:RelaxPISUDef} is satisfied.  Similarly, \eqref{constr:AltPISDDef} and the construction of $u^{\mathrm{sd}}_{jt'}$ imply \eqref{constr:RelaxPISDDef}.  This establishes, with certificate $(u^{\mathrm{su}}_j,u^{\mathrm{sd}}_j)$, that $(p_{jt},q_{jt},u^{\mathrm{on}}_{jt}) \in \mathcal{Y}^{\mathrm{uc}}_{jt}$.
\end{proof}

\subsubsection*{Proof of Theorem~\ref{thm:tailoredAlgorithmFeasible}}

\begin{proof}
We first demonstrate that the subproblems $\overline{Z}^{\mathrm{T}}_t$, $\overline{Z}^{\mathrm{I}}_i$, and $Z^{\mathrm{SH}}_{jt}$ are feasible at every iteration.  Under Assumption~\ref{assump:DataAssumptions}, there is some feasible solution $(p,q,u,v,\Delta,\theta,\tau,\phi)$ for \eqref{SC-ACOPF}.  By definition, $(p_j,q_j,u_j) \in \mathcal{X}^{\mathrm{uc}}_j$ for all $j \in \mathcal{J}^{\mathrm{sd}}$, $(p_t,q_t,u^{\mathrm{sh}}_t) \in \mathcal{X}^{\mathrm{ac}}_t$ for all $t \in \mathcal{T}$, and $u^{\mathrm{sh}}_{jt} \in \mathcal{X}^{\mathrm{sh}}_{jt}$ for all $j \in \mathcal{J}^{\mathrm{sh}}$ and $t \in \mathcal{T}$.  By Lemma~\ref{lemma:setRelaxations}, $(p_{jt},q_{jt},u^{\mathrm{on}}_{jt}) \in \mathcal{Y}^{\mathrm{uc}}_{jt}$ for all $j \in \mathcal{J}^{\mathrm{sd}}$ and $t \in \mathcal{T}$, and $u^{\mathrm{sh}}_{jt} \in \mathcal{Y}^{\mathrm{sh}}_{jt}$ for all $j \in \mathcal{J}^{\mathrm{sh}}$ and $t \in \mathcal{T}.$  By Proposition~\ref{prop:SOCRelaxation}, there are some $(\mu_t,\omega_t)$ such that $(p_t,q_t,\mu_t,\omega_t) \in \mathcal{X}^{\mathrm{soc}}_t$ for all $t \in \mathcal{T}$. By Proposition~\ref{prop:AltPIEquivalent}, there is some $\chi$ such that  $(p_{jt},q_{jt},u^{\mathrm{on}}_{jt},\chi_{jt}) \in \overline{\mathcal{Y}}^{\mathrm{uc}}_{jt}$ for all $j \in \mathcal{J}^{\mathrm{sd}}$ and $t \in \mathcal{T}$.  Therefore, the subproblems $\overline{Z}^{\mathrm{T}}_t$, $\overline{Z}^{\mathrm{I}}_i$, and $Z^{\mathrm{SH}}_{jt}$ are always feasible.

Next, we demonstrate that the subproblems $H_j$ and $Z^{\mathrm{F}}_t$ are feasible at steps 10, 11, and 13 of the algorithm.  The iterate $(\overline{p}^{(\overline{\tau})},\overline{q}^{(\overline{\tau})},\overline{u}^{(\overline{\tau})})$ is feasible for the bus-level subproblems $\overline{Z}^{\mathrm{I}}_i$, so $(\overline{p}^{(\overline{\tau})}_j,\overline{q}^{(\overline{\tau})}_j,\overline{u}^{(\overline{\tau})}_j) \in \mathcal{X}^{\mathrm{uc}}_j$ for all $j \in \mathcal{J}^{\mathrm{sd}}$.  Then, Proposition~\ref{prop:rampLPFeas} gives that problems $H_j(\overline{p}^{(\overline{\tau})}_j,\overline{u}^{(\overline{\tau})}_j)$ are feasible for all $j \in \mathcal{J}^{\mathrm{sd}}$ and the solutions $\delta$ are nonnegative.  By Propositions~\ref{prop:EQequivalence}~and~\ref{prop:DLFeasible}, there is some feasible solution $(p,q,u,v,\Delta,\theta,\tau,\phi)$ for \eqref{SC-ACOPF} with $u = \overline{u}^{(\overline{\tau})}$and $p^{\mathrm{tot}} = {\overline{p}^{\mathrm{tot}(\overline{\tau})}}$.  Note that $u^{\mathrm{on}}$ and $p^{\mathrm{tot}}$ are copied variables, and $u^{\mathrm{su}}$ and $u^{\mathrm{sd}}$ are defined implicitly by $u^{\mathrm{on}}$ from constraints \eqref{constr:UCSUSDDef1}--\eqref{constr:UCSUSDDef2} and \eqref{constr:UCBinary}.  This solution satisfies constraints \eqref{constr:PIReal}--\eqref{constr:PIReactive} and \eqref{constr:PowerFlow}, and thus $(p_t,q_t,u^{\mathrm{sh}}_t) \in \mathcal{X}^{\mathrm{ac}}_t$ for all $t \in \mathcal{T}$.  Further, as $\delta$ is nonnegative, the solution is trivially feasible for \eqref{eq:rampSafeConstraint}.  Therefore, this solution is feasible for problems $Z^{\mathrm{F}}_t(\overline{p}^{(\overline{\tau})},\overline{q}^{(\overline{\tau})},\overline{u}^{(\overline{\tau})},\emptyset;\delta)$ and $Z^{\mathrm{F}}_t(\overline{p}^{(\overline{\tau})},\overline{q}^{(\overline{\tau})},\overline{u}^{(\overline{\tau})},\overline{\mathcal{N}};\delta)$.  This establishes the validity of the algorithm.

Finally, consider the output $(p,q,u)$.  This solution satisfies constraints \eqref{constr:PIReal}--\eqref{constr:PIReactive}.  As $u = \overline{u}^{(\overline{\tau})}$ and $(\overline{p}^{(\overline{\tau})},\overline{q}^{(\overline{\tau})},\overline{u}^{(\overline{\tau})}) \in \mathcal{X}^{\mathrm{uc}}_{j}$ for all $j \in \mathcal{J}^{\mathrm{sd}}$, it holds that $u$ satisfies constraints \eqref{constr:UC}.  Additionally, $(p_t,q_t,u^{\mathrm{sh}}_t) \in \mathcal{X}^{\mathrm{ac}}_t$ for all $t \in \mathcal{T}$, so there is some $(v,\Delta,\theta,\tau,\phi)$ such that the solution satisfies \eqref{constr:ACAngleDiff}--\eqref{constr:DCSets}.  As $u = \overline{u}^{(\overline{\tau})}$ and $\overline{u}^{\mathrm{sh}(\overline{\tau})}_{jt}$ is an integer on $[u^{\mathrm{sh,min}}_j,u^{\mathrm{sh,max}}_j]$ by definition of the subproblems $Z^{\mathrm{SH}}_{jt}$, the solution satisfies \eqref{constr:ShuntSets}.  By Proposition~\ref{prop:rampLPValid}, as $(p,u)$ satisfies constraint \eqref{eq:rampSafeConstraint} and $\delta$ is feasible for $H_j(\overline{p}^{(\overline{\tau})},\overline{u}^{(\overline{\tau})}) = H_j(\overline{p}^{(\overline{\tau})},u)$, the solution is feasible for constraint \eqref{constr:Ramp}.  Therefore, the solution $(p,q,u,v,\Delta,\theta,\tau,\phi)$ is feasible for \eqref{SC-ACOPF}.
\end{proof}

\subsubsection*{Proof of Theorem~\ref{thm:SOCUpperBound}}

\begin{proof}
  Consider any feasible solution $(p,q,u,v,\Delta,\theta,\tau,\phi)$ for \eqref{SC-ACOPF}.  By definition, $(p_j,q_j,u_j) \in \mathcal{X}^{\mathrm{uc}}_j$ for all $j \in \mathcal{J}^{\mathrm{sd}}$, $(p_t,q_t,u^{\mathrm{sh}}_t) \in \mathcal{X}^{\mathrm{ac}}_t$ for all $t \in \mathcal{T}$, and $u^{\mathrm{sh}}_{jt} \in \mathcal{X}^{\mathrm{sh}}_{jt}$ for all $j \in \mathcal{J}^{\mathrm{sh}}$ and $t \in \mathcal{T}$. By Lemma~\ref{lemma:setRelaxations}, $(p_{jt},q_{jt},u^{\mathrm{on}}_{jt}) \in \mathcal{Y}^{\mathrm{uc}}_{jt}$ for all $j \in \mathcal{J}^{\mathrm{sd}}$ and $t \in \mathcal{T}$, and $u^{\mathrm{sh}}_{jt} \in \mathcal{Y}^{\mathrm{sh}}_{jt}$ for all $j \in \mathcal{J}^{\mathrm{sh}}$ and $t \in \mathcal{T}.$  Further, by Proposition~\ref{prop:SOCRelaxation}, there are some $(\mu_t,\omega_t)$ such that $(p_t,q_t,\mu_t,\omega_t) \in \mathcal{X}^{\mathrm{soc}}_t$ for all $t \in \mathcal{T}$.  Thus, we have a corresponding feasible solution $(p_t,q_t,u_t,\mu_t,\omega_t)$ for each problem of \eqref{eq:SOCUB}.  

  We now establish fundamental properties of the objective functions.  For any $p$, as $c^{\mathrm{e}} \geq 0$, it holds that $C^{\mathrm{e}}_j(p) \geq 0$ for all $j \in \mathcal{J}^{\mathrm{sd}}$.  Similarly, for any $(p,q)$, as $c^{\mathrm{s}} \geq 0$ and $d_t \geq 0$, it holds that $C^{\mathrm{ctg}}_{kt} (p,q) \geq 0$ for all $k \in \mathcal{K}$ and $t \in \mathcal{T}$, and therefore $C^{\mathrm{ctg}}_t(p,q) \geq 0$ for all $t \in \mathcal{T}$.  Finally, for $u$ satisfying \eqref{constr:UC}, we have for all $j \in \mathcal{J}^{\mathrm{sd}}$ that
  \begin{equation*}
      \begin{aligned}
           C^{\mathrm{uc}}_j(u) - \sum_{t \in \mathcal{T}} c^{\mathrm{on}}_{jt} u^{\mathrm{on}}_{jt} & = \sum_{t \in \mathcal{T}} \left ( c^{\mathrm{su}}_j u^{\mathrm{su}}_{jt} + c^{\mathrm{sd}}_{j} u^{\mathrm{sd}}_{jt} - \sum_{\substack{t' \in \mathcal{T}\\ t' < t}} c^{\mathrm{dd}}_{jtt'} [u^{\mathrm{su}}_{jt} + u^{\mathrm{on}}_{jt'} - 1]^+ \right )\\
          & \geq \sum_{t \in \mathcal{T}} \left ( c^{\mathrm{su}}_j u^{\mathrm{su}}_{jt} + c^{\mathrm{sd}}_{j} u^{\mathrm{sd}}_{jt} - \sum_{\substack{t' \in \mathcal{T}\\ t' < t}} c^{\mathrm{dd}}_{jtt'} u^{\mathrm{su}}_{jt} \right )
          \geq \sum_{t \in \mathcal{T}} c^{\mathrm{sd}}_{j} u^{\mathrm{sd}}_{jt}
          \geq 0,
      \end{aligned}
  \end{equation*}
  where the first inequality follows from $[u^{\mathrm{su}}_{jt} + u^{\mathrm{on}}_{jt'} - 1]^+ \leq [u^{\mathrm{su}}_{jt}]^+ = u^{\mathrm{su}}_{jt}$ and nonnegativity of $c^{\mathrm{dd}}_{jtt'}$, the second inequality from Assumption~\ref{assump:DataAssumptions} ($c^{\mathrm{su}}_{j} \geq \sum_{\substack{t' \in \mathcal{T}\\t' < t}} c^{\mathrm{dd}}_{jtt'}$) and $u^{\mathrm{su}}_{jt} \geq 0$, and the third inequality from nonnegativity of $c^{\mathrm{sd}}_j$ and $u^{\mathrm{sd}}_{jt}$.
  Considering the objective function for \eqref{SC-ACOPF},
  \begin{equation*}
      \begin{aligned}
          & \sum_{t \in \mathcal{T}} R^{\mathrm{T}}_t(p,q) + \sum_{j \in \mathcal{J}^{\mathrm{sd}}} R^{\mathrm{J}}_j(p,u)\\
          =\ & \sum_{t \in \mathcal{T}} \left ( \sum_{j \in \mathcal{J}^{\mathrm{cs}}} R^{\mathrm{pow}}_{jt} (p) - \sum_{j \in \mathcal{J}^{\mathrm{pr}}} C^{\mathrm{pow}}_{jt} (p) - \sum_{j \in \mathcal{J}^{\mathrm{ac}}} C^{\mathrm{ac}}_{jt} (p,q) - \sum_{i \in \mathcal{I}} C^{\mathrm{bal}}_{it} (p,q) -C^{\mathrm{res}}_t (p,q) - C^{\mathrm{ctg}}_t (p,q) \right ) \\
          & - \sum_{j \in \mathcal{J}^{\mathrm{sd}}} \left (C^{\mathrm{uc}}_j(u) + C^{\mathrm{e}}_{j} (p) \right ) \\
          =\ & \sum_{t \in \mathcal{T}} \left ( C^{\mathrm{UB}}_t (p,q,u) + \sum_{j \in \mathcal{J}^{\mathrm{sd}}} c^{\mathrm{on}}_{jt} u^{\mathrm{on}}_{jt} - C^{\mathrm{ctg}}_t (p,q) \right )  - \sum_{j \in \mathcal{J}^{\mathrm{sd}}} \left (C^{\mathrm{uc}}_j(u) + C^{\mathrm{e}}_{j} (p) \right )
          \leq\ \sum_{t \in \mathcal{T}} C^{\mathrm{UB}}_t (p,q,u),
      \end{aligned}
  \end{equation*}
  where the inequality follows from the previously observed properties of the objective functions.  This implies that the feasible solution $(p,q,u,\mu,\omega)$ has an objective value in \eqref{eq:SOCUB} at least as large as the objective value for $(p,q,u,v,\Delta,\theta,\tau,\phi)$ in \eqref{SC-ACOPF}. 
  Because we are maximizing and this property applies for any feasible solution to \eqref{SC-ACOPF}, the optimal value of \eqref{eq:SOCUB} must be at least as high as the optimal value of~\eqref{SC-ACOPF}.
\end{proof}

\bibliographystyleappendix{informs2014}
\bibliographyappendix{AlternatingMethods_SCUCACOPF_appendix.bib}

\end{document}